\documentclass[12pt,reqno]{amsart}

\usepackage{amsmath,amsfonts,amsthm,mathrsfs,amssymb,amscd,enumerate,url,tikz-cd}
\usepackage{mathtools}
\usepackage{comment}
\usepackage{enumitem}
\usepackage{appendix}

\usepackage{xcolor}
\colorlet{mdtRed}{red!50!black}
\definecolor{dblue}{rgb}{0,0,.6}
\usepackage[colorlinks]{hyperref}
\hypersetup{linkcolor=blue,citecolor=dblue,filecolor=dullmagenta,urlcolor=mdtRed}

\newtheorem{theorem}{Theorem}[section]
\newtheorem{proposition}[theorem]{Proposition}
\newtheorem{lemma}[theorem]{Lemma}
\newtheorem{corollary}[theorem]{Corollary}
\numberwithin{equation}{section}

\theoremstyle{definition}
\newtheorem{definition}[theorem]{Definition}
\newtheorem{remark}[theorem]{Remark}

\newtheorem{claim}[theorem]{Claim}

\newcommand{\Z}{\mathbb{Z}}
\newcommand{\C}{\mathbb{C}}

\newcommand{\Q}{\mathbb{Q}}
\newcommand{\R}{\mathbb{R}}

\newcommand{\ch}{{\rm ch}}

\newcommand{\Spec}{\operatorname{Spec}}

\newcommand{\mb}[1]{\mathbb{#1}}
\newcommand{\mc}[1]{\mathcal{#1}}
\renewcommand{\t}[1]{\widetilde{#1}}

\DeclareMathOperator{\Pic}{Pic}
\DeclareMathOperator{\Cl}{Cl}
\DeclareMathOperator{\Nef}{Nef}
\DeclareMathOperator{\Eff}{Eff}
\DeclareMathOperator{\Mov}{Mov}
\DeclareMathOperator{\id}{id}

\DeclareMathOperator{\Quot}{Quot}

\DeclareMathOperator{\Hom}{Hom}
\DeclareMathOperator{\Ext}{Ext}
\DeclareMathOperator{\ext}{ext}
\DeclareMathOperator{\rank}{rank}
\DeclareMathOperator{\td}{td}
\DeclareMathOperator{\Proj}{Proj}
\DeclareMathOperator{\Sym}{Sym}
\DeclareMathOperator{\im}{im}
\DeclareMathOperator{\coker}{coker}

\DeclareMathOperator{\Tor}{Tor}
\DeclareMathOperator{\codim}{codim}
\DeclareMathOperator{\Supp}{Supp}

\def\P{\mathbb{P}}
\def\ep{\varepsilon}

\newcommand{\cale}{\mathcal {E}}

\newcommand{\calg}{\mathcal {G}}
\newcommand{\calh}{\mathcal {H}}

\newcommand{\calo}{\mathcal {O}}

\newcommand{\calq}{\mathcal {Q}}

\usepackage{marginnote}
\usetikzlibrary{calc}
\title[Birational geometry of Quot schemes on curves via stable pairs]{Birational geometry of Quot schemes on smooth projective curves via stable pairs}

\begin{document}

\author[C. Gangopadhyay]{Chandranandan Gangopadhyay} 
	
	\address{Department of Mathematics, Shiv Nadar University, NH91, Tehsil Dadri, Greater Noida, Uttar
Pradesh 201314, India} 
	
	\email{chandranandan.g@snu.edu.in} 

 \author[A. Ito]{Atsushi Ito}

\address{Department of Mathematics, Institute of Pure and Applied Sciences, University of Tsukuba, Tsukuba, Ibaraki 305-8571, Japan}

\email{ito-atsushi@math.tsukuba.ac.jp}

	

\begin{abstract}
Let $C$ be a smooth projective curve of genus $g \geq 2$ over $\C$, 
and let $E^0$ be a vector bundle on $C$. 
We investigate the birational geometry of the Quot scheme 
$\Quot_C(E^0, k, n)$, which parametrizes quotients of $E^0$ of rank $k$ 
and degree $n$, and its fiber $\mc Q_L$ over $\Pic^n(C)$ for $n \gg 0$. 
Our main tool is the moduli space of stable pairs, which yields small $\Q$-factorial modifications (SQMs) 
of $\Quot_C(E^0, k, n)$ and $\mc Q_L$. 
We explicitly describe the nef, movable, and effective cones of each SQM. 
Consequently, we prove that $\mc Q_L$ is a Mori dream space and that 
the determinant morphism $\Quot_C(E^0, k, n) \to \Pic^n(C)$ is a Mori dream morphism.
\end{abstract}
\maketitle

\tableofcontents

\section{Introduction}

Let $C$ be a smooth projective curve over the field of complex numbers $\C$. Let $E^0$ be a vector bundle on $C$ of rank $r$ and degree $e$. 
For integers $0 \leq k \leq r-1$ and $n$,
let ${\Quot}_C(E^0,k,n)$ be the Quot scheme parametrizing quotients of $E^0$ of rank $k$ and degree $n$. 
The Quot scheme has played a fundamental role in the construction and study of moduli spaces of vector bundles (see e.g.\ \cite{Drezet-Narasimhan}, \cite{PR}, \cite{BGL}, \cite{Marian_Oprea_strange_duality}).  
Since the Quot scheme for $E^0=\calo_C^{r}$ can be regarded as a compactification of the space of morphisms from $C$ to the Grassmannian,
it has been extremely useful in Gromov--Witten theory (see e.g.\  \cite{BDW}, \cite{Bertram}, \cite{Marian_Oprea}, \cite{Oprea_tautological}).

The geometry of these Quot schemes has been studied extensively, in particular the case 
of $\mc Q_{\P^1}\coloneqq{\Quot}_{\P^1}(\calo^r_{\P^1},k,n)$.
If $k=r-1$ or $n=0$,  $\mc Q_{\P^1}$ is a projective space or a Grassmannian, respectively.
If $0\leq k \leq r-2$ and $n>0$,
Str\o mme in \cite{Str} showed that $\mc Q_{\P^1}$ is a smooth projective variety with  $\Pic (\mc Q_{\P^1}) = \Z^2$
and the nef cone $\Nef (\mc Q_{\P^1})$ is generated by two globally generated line bundles. 
In \cite{Jow} and \cite{Venkatram}, the effective cone $\Eff(\mc Q_{\P^1})$ and the movable cone $\Mov(\mc Q_{\P^1})$ were computed respectively. 
In \cite{Ito_MDS}, the birational geometry of $\mc Q_{\P^1}$ was studied and it was shown that $\mc Q_{\P^1}$ is a Mori dream space. 
Recall that a normal projective variety $X$ is called a \emph{Mori dream space} if
\begin{enumerate}[label=\arabic*)]
\item $X$ is $\Q$-factorial and $\Pic(X) \otimes \Q =N^1(X)_{\Q}$,
\item $\Nef(X )$ is spanned by finitely many semiample line bundles,
\item  there exist finitely many small $\Q$-factorial modifications (SQMs) $ f_i : X \dashrightarrow X_i$ such that
each $X_i$ satisfies 2) and $\Mov(X) = \bigcup_i f_i^*(\Nef(X_i))$.
\end{enumerate}
Mori dream spaces were introduced by Hu and Keel in \cite{Hu_keel} and they have nice properties from the perspective of the minimal model program.
In \cite{Ito_MDS}, it was further shown that all the small $\Q$-factorial modifications of $\mc Q_{\P^1}$ have moduli-theoretic descriptions.

For $C$ with genus $g \geq 2$, the Quot scheme $\mc Q \coloneqq {\Quot}_C(E^0,k,n)$ can be reducible and non-reduced. 
However, for $n \gg 0$, $\mc Q$ is known to be irreducible and generically smooth for $E^0=\calo_C^r$ by \cite{BDW}, and for arbitrary $E^0$ by \cite{PR}.
Furthermore, $\mc Q$ is locally factorial and its Picard group was computed for $n \gg 0$ by \cite{GS-Picard}.

Note that there is a natural morphism  
\begin{align}\label{eq_intro_det}
\det:\mc Q\to {\Pic}^n(C), \quad [E^0 \twoheadrightarrow B]\mapsto \det(B).
\end{align}
For $[L]\in \Pic^n(C)$,
we set  $\mc Q_L={\Quot}_C(E^0,k,n)_L\coloneqq\det^{-1}([L])$.
If $ k=r-1$ and $n \gg 0$, 
\eqref{eq_intro_det} is a projective bundle and hence $\mc Q_L$ is a projective space.
For $0 \leq  k \leq r-2$ and $n\gg 0$, it was shown in \cite{GS-Picard} that \eqref{eq_intro_det} is flat,
and $\mc Q_L$ is also locally factorial with $\Pic(\mc Q_L)=\Z^2$. 
In \cite{GS-eff}, the nef and effective cones of $\mc Q_L$ were computed  for $E^0=\mathcal{O}_C^r$ and $n\gg 0$. 
In particular, the nef cone is generated by two globally generated line bundles.

Hence, $\mc Q_L$ for $n \gg 0$ shares several geometric properties with $\mc Q_{\P^1}$.
The purpose of this article is to study the birational geometry of $\mc Q_L$ for arbitrary $E^0$ and sufficiently large $n$.
In particular, we show that, similarly to $\mc Q_{\P^1}$, the space $\mc Q_L$ is a Mori dream space.

To produce all the SQMs of $\mc Q_L$, we work with the moduli space 
$M_{C,\delta}(E_0,s,d)$ for $E_0\coloneqq(E^{0})^\vee$ and a real number $ \delta>0$,
which parametrizes equivalence classes of  $\delta$-semistable pairs $(E,\alpha)$ with $E$ a vector bundle of rank $s$ and degree $d$ and $\alpha : E_0 \to  E$ a non-zero homomorphism. 
We refer to \S \ref{section_preliminaries} for the definition of $\delta$-(semi)stability.
This semistability condition was first introduced by Bradlow in \cite{Bradlow_special_metric} for $E_0 =\calo_C$.
In \cite{Bradlow_Daskalopoulos_Moduli}, the moduli space of semistable pairs $M_{C,\delta}(\mc O_C,s,d)$ was constructed as a compact K\"ahler manifold for certain values of $\delta$. In \cite{Thaddeus}, in the case $s=2$, it was shown that the moduli spaces for different values of $\delta$ are related by a sequence of flips. 
 In \cite{BDW}, $M_{C,\delta}(\mc O^r_C,s,d)$ was introduced for a certain range of $\delta$ and it was shown that for  $s=2$, they are again related by a sequence of flips. The moduli of semistable pairs $M_{C,\delta}(E_0,s,d)$ for any $E_0$ (in fact on any smooth projective variety) were constructed in \cite{Wan} and \cite{Lin}. 
As in the case of Quot schemes, there is a morphism
\[
 \det: M_{C,\delta}(E_0,s,d)\to \Pic^d(C), \quad [E_0 \xrightarrow{\alpha} E]\mapsto [\det(E)].
\]
For $[L]\in \Pic^d(C)$, we define $M_{C,\delta}(E_0,s,d)_L\coloneqq\det^{-1}([L])$.

 The relevance of the moduli space of stable pairs to our study of Quot schemes is that, for $\delta\gg 0$, there exists an isomorphism
\begin{align}\label{eq_intro1}
\Quot(E^0,k,n) \xrightarrow{\sim} M_{C,\delta}(E_0,r-k,n-e) , \quad  [E^0 \stackrel{q}{\twoheadrightarrow} B] \mapsto [E_0 \to (\ker q)^\vee].
\end{align}
Since $\det (\ker q)^\vee = \det B \otimes \det E_0$,
\eqref{eq_intro1} induces an isomorphism 
\begin{align}\label{eq_intro2}
\Quot(E^0,k,n)_L \simeq  M_{C,\delta}(E_0,r-k,n-e)_{L \otimes \det E_0}.
\end{align}

In the following Theorems \ref{thm_intro_SQM}, \ref{intro_thm_cones}, we work under the setting below.
\begin{itemize}
\item $C$ is a smooth projective curve of genus $g \geq 2$.
\item $E^0$ is a vector bundle on $C$ of rank $r$ and degree $e$.
\item $ 0 \leq k \leq r-2$ is an integer.
\item $n  $ is a sufficiently large integer depending on $C, E^0,k$.
\item $\mc Q_L  $ is the fiber of the morphism $\det :  \Quot(E^0,k,n) \to {\Pic}^n(C)$ over  $[L] \in \Pic^n(C) $.
\item Set $E_0 \coloneqq (E^0)^\vee, s \coloneqq r-k, d\coloneqq n-e$ and $ \tilde{L}\coloneqq L \otimes \det E_0 \in \Pic^d(C)$.
\end{itemize}

Under the above setting,
we see that there exists a sequence of rational numbers 
\begin{align*}
0 =\delta_0 < \delta_1 < \delta_2 < \cdots < \delta_N 
\end{align*}
which depends on $E_0,s,d$, 
such that for each $1 \leq i \leq N+1$, $M_{C,\delta}(E_0, s, d)$ and $M_{C,\delta}(E_0, s, d)_{\tilde{L}}$ do not depend on $\delta \in (\delta_{i-1}, \delta_i)$,
where $\delta_{N+1} =\infty$.
In particular, $M_{C,\delta}(E_0, s, d)_{\tilde{L}}$ is isomorphic to $\mc Q_L $ for $\delta > \delta_N$ by \eqref{eq_intro2}.

\begin{theorem}\label{thm_intro_SQM}
For $1 \leq i \leq N+1$ and $\delta \in (\delta_{i-1},\delta_i)$,
the following hold.
\begin{enumerate}
\setlength{\itemsep}{0mm}
\item $M_{C,\delta}(E_0, s, d)_{\tilde{L}}$ is irreducible, reduced, normal, locally complete intersection, and locally factorial.
\item The rational map defined by 
\begin{align*}
 \mc Q_L \dashrightarrow M_{C,\delta}(E_0, s, d)_{\tilde{L}}  : [E^0 \stackrel{q}{\twoheadrightarrow} B] \mapsto [E_0 \to (\ker q)^\vee],
\end{align*}
is a SQM of $ \mc Q_L$.
\end{enumerate}
\end{theorem}

If $k \geq 2$,
we see that 
\begin{align*}
 \mc Q_L={\Quot}_C(E^0,k,n)_L  \dashrightarrow  \mc Q'_{\tilde{L}} \coloneqq \Quot_C(E_0,s,d)_{\tilde{L}} : [E^0 \stackrel{q}{\twoheadrightarrow} B] \mapsto [E_0 \twoheadrightarrow (\ker q)^\vee]
\end{align*}
is a SQM as well.
Hence there exists a sequence of rational numbers 
\begin{align*}
0 =\delta'_0 < \delta'_1 < \delta'_2 < \cdots < \delta'_{N'} 
\end{align*}
such that for $1 \leq j \leq N'+1$ and $\delta' \in (\delta'_{j-1},\delta'_j)$,
the composite 
\begin{align*}
 \mc Q_L  \dashrightarrow  \mc Q'_{\tilde{L}} \dashrightarrow M_{C,\delta'}(E^0, k, n)_L : [E^0 \stackrel{q}{\twoheadrightarrow} B] \mapsto [E^0 \stackrel{q}{\twoheadrightarrow} B]
\end{align*}
is a SQM by  applying Theorem \ref{thm_intro_SQM} to $\mc Q'_{\tilde{L}}$.

In the following theorem, $M_{C,\delta}(E_0, s, d)_{\tilde{L}}$ is denoted by $M_{i,\tilde{L}} $ for $\delta \in (\delta_{i-1},\delta_i)$ for simplicity.
Similarly,
$M_{C,\delta'}(E^0, k, n)_{L}$ is denoted by $M'_{j,L} $ for $\delta' \in (\delta'_{j-1},\delta'_j)$.
By \eqref{eq_intro2}, we identify $ \mc Q_L$ (resp.\ $\mc Q'_{\tilde{L}}$) with $M_{N+1,\tilde{L}}$ (resp.\ $M'_{N'+1,L}$).

\begin{theorem}\label{intro_thm_cones}
There exists a basis $\lambda,\mu $ of {$N^1(\mc Q_L)_\R \simeq \R^2$}
such that  the following hold.

Let $\theta_i$ be classes in $ N^1(\mc Q_L)_\R=\R \lambda \oplus \R \mu$ defined by
\begin{align*}
\theta_i &= \left(\frac{d+\delta_{i}}{s}  -(g-1) \right)\lambda -\mu \quad \text{for} \quad 0 \leq i \leq N.
\end{align*}
If $k \geq 2$, we further define $\theta'_j \in  N^1(\mc Q_L)_\R$ by
\begin{align*}
\theta'_j &= \left(\frac{n+\delta'_{j}}{k}  +(g-1) \right)\lambda +\mu \quad  \text{for} \quad  0 \leq j \leq N'.
\end{align*}
If $k=1$, we define $\theta'_0$ in the same way by setting $\delta'_0=0$, that is, $\theta'_0 = \left(n +(g-1) \right)\lambda +\mu $.
Then
\begin{enumerate}
\setlength{\itemsep}{0mm}
\item It holds that
\begin{align*}
\Nef(M_{i,\tilde{L}}) &=\R_{\geq 0} \theta_{i-1}+ \R_{\geq 0} \theta_{i}  \quad \text{for} \quad 1\leq i \leq N, \\
\Nef(\mc Q_L)  =\Nef(M_{N+1,\tilde{L}}) &=\R_{\geq 0} \theta_{N}+ \R_{\geq 0} \lambda .
\end{align*}
Furthermore, if $k \geq 2$, it holds that 
\begin{align*}
\Nef(M'_{j,L}) &= \R_{\geq 0} \theta'_{j-1}+ \R_{\geq 0} \theta'_{j} \quad \text{for} \quad 1\leq j \leq N', \\
\Nef(\mc Q'_{\tilde{L}}) =\Nef(M'_{N'+1,L}) &= \R_{\geq 0} \theta'_{N'}+ \R_{\geq 0} \lambda .
\end{align*}
All edges of these nef cones are spanned by semiample line bundles.
\item If $k =0$ or $1$, 
\begin{align*}
 \Mov (\mc Q_L) = \bigcup_{i=1}^{N+1} \Nef(M_{i,\tilde{L}})  =\R_{\geq 0} \theta_0  + \R_{\geq 0} \lambda.
\end{align*}
If $k \geq 2$,
\begin{align*}
 \Mov (\mc Q_L) = \bigcup_{i=1}^{N+1} \Nef(M_{i,\tilde{L}})  \cup  \bigcup_{j=1}^{N'+1} \Nef(M'_{j,L}) =\R_{\geq 0} \theta_0  + \R_{\geq 0} \theta'_0.
\end{align*}
\item It holds that
\begin{align*}
\Eff (M_{N+1,\tilde{L}})  =  
   \begin{cases}
     \R_{\geq 0} \theta_0  + \R_{\geq 0} \lambda  &  \text{ if } \  k=0 \\
     \R_{\geq 0} \theta_0  + \R_{\geq 0} \theta'_0   & \text{ if } \ 1 \leq k \leq r-2
   \end{cases}
\end{align*}
\end{enumerate}
In particular,
$\mc Q_L $ is a Mori dream space.
\end{theorem}

The following figures are the slices of the movable cone $\Mov(\mc Q_L)$.
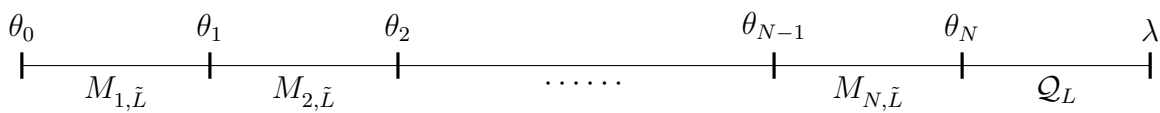
\begin{figure}[htbp]
\centering
\begin{tikzpicture}
\draw(0,0)--(15,0);

\foreach \x in {0,2.5,5,10,12.5,15}
  \draw[very thick] (15-\x,5pt)--(15-\x,-5pt);

\foreach \y/\ytext in {0/$\lambda$,2.5/$\theta_N$,5/$\theta_{N-1}$,10/$\theta_2$,12.5/$\theta_1$,15/$\theta_0$}
  \draw (15-\y,0) node[above=1ex] {\ytext};

\foreach \z/\ztext in {1.25/$\mc Q_L$,3.75/$M_{N,\tilde{L}}$,7.5/$ \cdots\cdots$,11.25/$M_{2,\tilde{L}}$,13.75/$M_{1,\tilde{L}}$}
  \draw (15-\z,0) node[below] {\ztext};
\end{tikzpicture}
\caption{The slice of the movable cone $\Mov(\mc Q_L)$ for $k=0$ or $1$} \label{mov_0,1}
\end{figure}

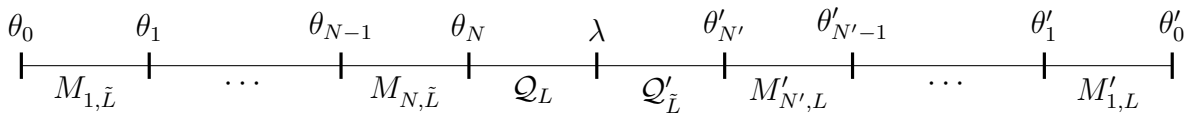
\begin{figure}[htbp]
\centering
\begin{tikzpicture}[x=0.85cm]
\draw(0,0)--(18,0); 

\foreach \x in {0,2,5,7,9,11,13,16,18} 
  \draw[very thick] (\x,5pt)--(\x,-5pt);

\foreach \y/\ytext in {
  0/$\theta'_0$,
  2/$\theta'_1$,
  5/$\theta'_{N'-1}$,
  7/$\theta'_{N'}$,
  9/$\lambda$,
  11/$\theta_N$,
  13/$\theta_{N-1}$,
  16/$\theta_1$,
  18/$\theta_0$}
  \draw (18-\y,0) node[above=1ex] {\ytext};

\foreach \z/\ztext in {
  1/$M'_{1,L}$,
  3.5/$\cdots$,
  6/$M'_{N',L}$,
  8/$\mc Q'_{\tilde{L}}$,
  10/$\mc Q_L$,
  12/$M_{N,\tilde{L}}$,
  14.5/$\cdots$,
  17/$M_{1,\tilde{L}}$}
  \draw (18-\z,0) node[below] {\ztext};
\end{tikzpicture}

\caption{The slice of the movable cone $\Mov(\mc Q_L)$ for $k\geq 2$} \label{mov_2}
\end{figure}

We further show that $\det:\Quot(E^0,k,n) \to \Pic^n(C)$ is a Mori dream morphism, which is the relative version of Mori dream spaces as introduced in \cite{Ohta}.
We also study $M_{C,\delta}(E_0,s,d)$ and $M_{C,\delta}(E_0,s,d)_L$ for $s >r$ and obtain similar results.

This paper is organized as follows. 
In \S \ref{section_preliminaries}, we recall the notion of $\delta$-stability.
In \S \ref{sec_small_or_large_delta}, we consider $M_{C,\delta}(E_0,s,d)$ for sufficiently large or small $\delta>0$.
In  \S \ref{sec_estimate_of_dimension},
we give some estimates of the dimensions of the moduli spaces. In 
\S \ref{sec_SQM}, we prove Theorem \ref{thm_intro_SQM}.
{ In \S \ref{sec_Gamma}, we determine $\delta_i$ in Theorem \ref{thm_intro_SQM}.}
In \S \ref{section_nef cone}, we determine the nef cones of  $M_{C,\delta}(E_0,s,d)_L$. In 
\S \ref{sec_canonical_div}, we determine the canonical divisor of $M_{C,\delta}(E_0,s,d)_L$. In 
\S \ref{sec_proof_of_main_result}, we give a proof of Theorem \ref{intro_thm_cones}.
In Appendices \ref{appendix_s=2}, \ref{sec_appendix_s-r=2},
we describe the Picard group of  $M_{C,\delta}(E_0,s,d)_L$ for $ s=2$ or $s-r=2$,
which is used in \S \ref{sec_SQM}.

Throughout this paper, we work over the field of complex numbers $\C$.
We use the notation $\P(V) =\Proj \Sym V^{\vee}$  for a vector space $V$ over $\C$.

\subsection*{Acknowledgments}
A.~Ito was supported by JSPS KAKENHI Grant Number 21K03201.

\section{Preliminaries}\label{section_preliminaries}

Throughout this paper,
$C$ is a smooth projective curve of genus $g \geq 2$
and $E^0$ is a vector bundle on $C$ of rank $r$ and degree $e$.
Let $E_0\coloneqq(E^0)^{\vee}$.

\subsection{$\delta$-stability}
Let $\delta>0$. 
In this section, we recall some notions on $\delta$-stability of pairs.
See \cite{Wan}, \cite{Lin} for details.

\begin{definition}
\begin{enumerate}
\item A \emph{pair} $(E,\alpha)$ consists of a coherent sheaf $E$ and a morphism $\alpha:E_0\to E$. 
If $\alpha \neq 0$, $(E,\alpha)$  is called \emph{nondegenerate}.
\item A \emph{subpair} of a pair $(E,\alpha)$ is a pair $(E',\alpha')$ where $\iota: E'\subset E$ 
is a subsheaf and
\[
\iota \circ \alpha' = \alpha \quad \text{if } \im  \alpha \subset E', \qquad \alpha' = 0 \quad \text{otherwise}.
\]
\item 
A \emph{quotient pair}  of $(E, \alpha)$ is a pair $(E'', \alpha'') $ where $q : E \to E''$ is a quotient sheaf  and $ \alpha'' = q \circ \alpha : E_0 \to E'' $. 
\item The \emph{$ \delta$-slope} of  a pair $ (E, \alpha) $ is
\[
\mu_\delta(E, \alpha) = \mu(E, \alpha) \coloneqq \frac{\deg E + \varepsilon(\alpha)\delta}{\rank  E},
\]
where $\varepsilon(\alpha) \coloneqq 0$ if $\alpha=0$ and $\varepsilon(\alpha) \coloneqq 1$ if $\alpha\neq 0$.
\item A \emph{morphism} of  pairs $\varphi : (E, \alpha) \to (F,\beta)$ is a morphism of sheaves $\varphi : E \to F$ such that there exists a constant $ c \in \C$ with $\varphi \circ \alpha =c \beta$.
\item A pair $ (E, \alpha) $ is $ \delta $-\emph{stable} if $ E $ is locally free and
\begin{equation}
\mu(E', \alpha') < \mu(E, \alpha) \tag{1.1}
\end{equation}
for any proper subpair $(E',\alpha')$. 
Equivalently,
 $(E, \alpha)$ is $\delta$-stable if  $E$ is locally free and
\begin{equation}
\mu(E, \alpha) < \mu(E'', \alpha'') \tag{1.2}
\end{equation}
for every proper quotient pair $ (E'', \alpha'')$. 
$\delta$-semistability is defined by replacing the strict inequality  $ <$ with $ \leq $.
\item For a subpair $(E',\alpha')$  of a pair $(E,\alpha)$,
the quotient pair $(E,\alpha)/(E',\alpha') =(E/E',\alpha'')$ is defined by 
$\alpha''=0$ if $\alpha' \neq 0$, and $\alpha''=q \circ \alpha$ if $\alpha' =0$,
where $q : E \to E/E'$ is the quotient map.
\item For integers $s \geq 1$ and $d$,
let  $M_{C,\delta}(E_0, s, d) $  be the coarse moduli space of $S$-equivalence classes of (nondegenerate, that is, $ \alpha \neq 0$) $ \delta$-semistable pairs $ (E, \alpha) $ with $\rank E=s$, 
$ \deg E = d $.
\item 
The fiber of 
\[
\det : M_{C, \delta}(E_0, s, d) \to \operatorname{Pic}^d(C) :  \quad [(E, \alpha)] \mapsto [\det E]
\]
over $[L] \in \Pic^d(C)$ is denoted by $M_{C, \delta}(E_0, s, d)_L$.
\end{enumerate}
\end{definition}

We note that the moduli space $M_{C,\delta}(E_0, s, d) $ exists by \cite[Theorem 3.8]{Wan}, \cite[Theorem 1.1]{Lin}.
Furthermore, the locus $M^s_{C,\delta}(E_0, s, d) $ of $ \delta $-stable pairs is the fine moduli space.

\begin{lemma}\label{lem_preliminary}
\begin{enumerate}
\item For a locally free sheaf $E$, $(E,0)$ is $\delta$-(semi)stable if and only if $E$ is (semi)stable in the usual sense.
\item If $f: (E,\alpha) \to (F,\beta)$ is a non-zero morphism between $\delta$-stable pairs with $\mu(E,\alpha) =\mu(F,\beta)$, 
then $f$ is an isomorphism. 
\item If $F$ is semistable and  $(E, \alpha)$ is nondegenerate and $\delta$-stable with $\mu(F) =\mu(E,\alpha)$, then $\Hom(F,E) =0$.
\item Let $0 \to (E',\alpha') \to (E,\alpha) \to (E'',\alpha'') \to 0$ be exact, that is, $0 \to E' \to E \to E'' \to 0$ is exact and $ (E',\alpha') $ is a subpair of $(E,\alpha)$, and $(E'',\alpha'')$ is the corresponding quotient pair.
If $(E',\alpha')$ and $(E'',\alpha'')$ are $\delta$-semistable and $\mu(E',\alpha') =\mu(E'',\alpha'')$,
then $(E,\alpha) $ is $\delta$-semistable.
\item Let $0 \to (E',\alpha') \to (E,\alpha) \to (E'',0) \to 0$ be exact with $\alpha \neq 0$.
Assume $(E',\alpha')$ is $\delta$-stable and $E''$ is semistable with $\mu_\delta(E',\alpha') =\mu(E'')$.
If $(E,\alpha)$ is not $(\delta-\ep)$-stable for sufficiently small $\ep >0$,
there exists a subsheaf $G$ with $G \cap E' =0$ and $\mu(G)=\mu(E'')$.
\end{enumerate}
\end{lemma}

\begin{proof}
(1) follows from the definition.
(2) is a special case of  \cite[Lemma 2.11]{Lin}.\\
(3) Assume $f : F \to E$ is non-zero and let $(f(F), \alpha') $ be the induced subpair of $(E,\alpha)$.
Then $\mu(F) \leq \mu(f(F)) \leq \mu(f(F), \alpha')  \leq \mu(E,\alpha) =\mu(F)$, where $\mu  =\deg /\rank $ is the usual slope,
 and hence the equality holds in each inequality.
By $ \mu(f(F), \alpha')  =\mu(E,\alpha) $ and the $\delta$-stability of $(E,\alpha)$, 
we have $(f(F), \alpha') =(E,\alpha) $ and hence $\alpha'=\alpha \neq 0$.
On the other hand, we have $\alpha'=0$ by $ \mu(f(F)) = \mu(f(F), \alpha')$, which is a contradiction.\\
(4) Note that we have $\mu(E,\alpha) =\mu(E',\alpha') =\mu(E'',\alpha'') \eqqcolon \mu$.

Let $(F,\beta)$ be a subpair of $(E,\alpha)$.
If $\beta=0$, then we have  $\mu(F,\beta) =\mu(F) \leq \mu$ by
 $\mu(F\cap E') \leq \mu(E',\alpha')=\mu $ and $\mu(F/(F\cap E'))  \leq \mu(E'',\alpha'') =\mu$.
If $\beta \neq 0$, 
let $(F\cap E',\beta') \subset (E',\alpha') $ and $ (F/(F\cap E'),\beta'') \subset (E'',\alpha'')$ be the induced subpairs.
Hence we can check that $\mu(F,\beta) \leq \mu$ by using
 $\mu(F\cap E',\beta') \leq \mu(E',\alpha')=\mu $ and $\mu(F/(F\cap E'),\beta'')  \leq \mu(E'',\alpha'') =\mu$.\\
(5) By (4), $(E,\alpha)$ is $\delta$-semistable.
Since  $(E,\alpha)$ is not $(\delta-\ep)$-stable,
there exists a subpair $(G,\alpha_G) \subsetneq (E,\alpha)$ such that $\mu_{\delta-\ep} (G,\alpha_G)\geq \mu_{\delta-\ep} (E,\alpha)$.
Hence we have $\alpha_G=0$ and $\mu(G)= \mu_{\delta} (E,\alpha)$ by $\mu_{\delta} (G,\alpha_G)\leq \mu_{\delta} (E,\alpha)$. 

Then $(G\cap E',0)  \subsetneq (E',\alpha')$ is a subpair  and $ G/(G\cap E') \hookrightarrow E''$ is a subsheaf.
For $\mu \coloneqq \mu_\delta(E,\alpha) =\mu_\delta(E',\alpha') =\mu(E'') $,
we have 
\begin{itemize}
\item $G\cap E' =0$ or $\mu(G\cap E') <\mu $, and
\item $G/(G\cap E') =0$ or $\mu(G/(G\cap E') ) \leq \mu$.
\end{itemize}
Since $0 \to G\cap E' \to G \to G/(G\cap E') \to 0$ is exact and  $\mu(G)=\mu$, 
we have $G\cap E' =0$ and $\mu(G)=\mu(G/(G\cap E') ) \leq \mu$.
\end{proof}

\subsection{Mori dream morphisms}

\emph{Mori dream morphisms}, which are a relative version of Mori dream spaces,  were introduced in \cite{MR3158045}, \cite{Ohta}.
We refer the reader to  \cite{Ohta} for the details.

For a morphism $\pi : X \to U$ between quasi-projective varieties,
\[
\Pic(X/U)\coloneqq \Pic(X)/\pi^* \Pic(U)
\]
denotes the relative Picard group.
We set $ \Pic(X/U)_{\Q} \coloneqq \Pic(X/U) \otimes \Q$ and
$N^1(X/U) = \Pic(X/U)_{\Q}/\equiv$,
where $D \equiv D'$ if $(D.C)=(D'.C)$ for any complete curve $C\subset X$ contracted by $\pi$.
For a Cartier divisor $D$ on $X$, let 
\[
\alpha_{D} : \pi^*\pi_* \calo(D) \to \calo(D)
\]
be the natural map.
A $\Q$-Cartier $\Q$-divisor $D$ is called $\pi$-effective, $\pi$-movable, $\pi$-semiample
if $\codim \Supp (\coker (\alpha_{mD})) \geq 1$, $\codim \Supp (\coker (\alpha_{mD})) \geq 2$, $\coker (\alpha_{mD}) =0$ for some $m >0$, respectively.

\begin{definition}[{\cite[Definition 3.1]{Ohta}}]
A projective morphism $\pi : X \to U$ between quasi-projective varieties is called a \emph{Mori dream morphism} if it satisfies the following conditions.
\begin{enumerate}[label=\arabic*)]
\item $X$ is $\Q$-factorial and $\Pic(X/U)_{\Q} =N^1(X/U)_{\Q}$,
\item $\Nef(X/U)$ is spanned by finitely many $\pi$-semiample line bundles,
\item  there exist finitely many SQMs $ f_i : X \dashrightarrow X_i$ over $U$ such that
each $X_i$ satisfies 2) and $\Mov(X/U) = \bigcup_i f_i^*(\Nef(X_i/U))$.
\end{enumerate}

\end{definition}

\section{$M_{C,\delta}(E_0, s, d)$ for sufficiently large or small $\delta>0$}\label{sec_small_or_large_delta}

Throughout this section, fix integers $s >0$ and $d$.
For a locally free sheaf $F$ on $C$, the (usual) slope is defined by $\mu(F) =\deg F/\rank F.$

First, we consider the case when $\delta$ is sufficiently large.
For each integer $1 \leq i \leq r$, set
\begin{align*}
 e_i \coloneqq \min\{ \deg F \mid E_0 \twoheadrightarrow F, \rank F=i\} 
\end{align*}
and 
\begin{align}\label{eq_dagger}
 \delta_\sharp \coloneqq \inf \left\{\delta \geq 0 \mid \frac{d+\delta}{s} < \frac{e_i+\delta}{i} \text{ for any } 1 \leq i \leq \min\{s-1,r\}\right\} .
\end{align}
More explicitly, we can define $\delta_\sharp \coloneqq 0$ if $s=1$ and 
\begin{align*}
\delta_\sharp \coloneqq \max\limits_{1 \leq i \leq \min\{s-1,r\}} \frac{id-s e_i }{s-i}
\end{align*}
if $s \geq 2$.

\begin{lemma}\label{lem_suff_large_delta}
Let $(E, \alpha)$ be a locally free pair with $\rank E = s$, $\deg E = d$, $\alpha \neq 0$.
Assume $\delta > \delta_\sharp$.
Then $(E, \alpha)$ is $\delta$-semistable if and only if $(E, \alpha)$ is $\delta$-stable if and only if $\alpha$ is generically surjective.

In particular,
\begin{enumerate}
    \item If $s >r$,  $M_{C, \delta}(E_0, s, d)  =\emptyset$.
    \item If $1 \leq s \leq r$,
\[
M_{C, \delta}(E_0, s, d) \to \operatorname{Quot}_C(E^0, r - s, d+e), \quad [(E, \alpha)] \mapsto [E^0 \to E^0 / \alpha^\vee(E^\vee)]
\]
is an isomorphism, where $\alpha^\vee : E^\vee \to E_0^\vee= E^0$ is the dual of $\alpha$.
\end{enumerate}
\end{lemma}

\begin{proof} 
Assume $(E, \alpha)$ is $\delta$-semistable. If $\alpha$ is not generically surjective, then $s' \coloneqq\rank \im \alpha 
\leq \min\{s-1,r\}$. 
For the induced morphism $\alpha ' : E_0 \to \im \alpha$, 
$(\im \alpha, \alpha') $ is a subpair of $(E, \alpha)$ and hence we have
\[
\frac{d + \delta}{s} =\mu(E,\alpha) \geq \mu(\im \alpha, \alpha') = \frac{\deg (\im \alpha) + \delta}{s'} \geq \frac{e_{s'} +\delta}{s'} ,
\]
which contradicts $\delta > \delta_\sharp$.
Hence $\alpha$ is generically surjective.

If $\alpha$ is generically surjective, so is $\alpha'' : E_0 \to E''$ for any proper quotient $(E'', \alpha'')$ with locally free $E''$. Hence
\[
\mu(E'', \alpha'') = \frac{\deg E'' + \delta}{s''} \geq \frac{e_{s''} + \delta}{s''} > \frac{d + \delta}{s} = \mu(E, \alpha).
\]
Hence $(E, \alpha)$ is $\delta$-stable. 

Since $\delta$-stability implies $\delta$-semistability, we obtain the equivalence in the statement of this lemma.
The last statements (1), (2) are clear.
\end{proof}

Next we consider the case when $\delta >0$ is sufficiently small.
Define $\delta_\flat > 0$ by
\[
\frac{d + \delta_\flat}{s} = \min \left\{ \frac{d'}{s'} \, \middle| \,  s', d' \in \mathbb{Z}, 1 \leq s' \leq s - 1, \frac{d'}{s'} > \frac{d}{s} \right\}
\]
if $s\geq 2$.
Define $\delta_\flat \coloneqq \infty$ if $s=1$.

\begin{lemma}\label{lemma-small delta}
Assume $0 < \delta < \delta_\flat$.
If $(E, \alpha)$ is a $\delta$-semistable pair with $\rank E = s$, $\deg E = d$, $\alpha \neq 0$,
then $E$ is semistable in the usual sense. 
In particular, there exists a morphism
\[
f : M_{C, \delta}(E_0, s, d) \to M_C(s, d), \quad [(E, \alpha)] \mapsto [E]
\]
over $\Pic^d(C)$,
where $M_C(s, d)$ is the coarse moduli space of $S$-equivalence classes of semistable vector bundles of rank $s$ and degree $d$ on $C$.
\end{lemma}

\begin{proof}
Since any line bundle is stable, the case $s=1$ is clear.
Assume $s\geq 2$.

Let $(E, \alpha)$ be a $\delta$-semistable  pair with $\rank E = s$, $\deg E = d$.  
Let $E' \subset E$ be a subsheaf with  $1 \leq \rank E' \leq s-1$.
For the induced subpair $(E',\alpha')$ of $(E, \alpha)$, we have
\[
\mu(E') = \frac{\deg E'}{\rank E'} \leq \frac{\deg E'+\varepsilon(\alpha')\delta}{\rank E'} = \mu(E', \alpha')\leq  \mu(E, \alpha) =\frac{d + \delta}{s} < \frac{d + \delta_\flat}{s},
\]
which implies
\[
\mu(E') = \frac{\deg E'}{\rank E'} \leq \frac{d}{s} = \mu(E)
\]
by the  definition of $\delta_\flat$. Hence $E$ is semistable. 
The last statement is clear.
\end{proof}

\begin{remark}\label{rem_small_delta}
Let $E$ be a stable locally free sheaf with $\rank E=s, \deg E=d$.
Then any $(E, \alpha)$ with non-zero $\alpha : E_0 \to E$ is $\delta$-stable for $0 < \delta < \delta_\flat$. In fact, for any quotient bundle $q : E \to E''$ with $1 \leq\rank E'' \leq s - 1$, we have
\[
\mu(E'') > \mu(E) = \frac{d}{s}
\]
which implies
\[
\mu(E'', \alpha'') \geq \mu(E'') \geq \frac{d+\delta_\flat}{s} > \frac{d+\delta}{s} =\mu(E,\alpha).
\]
Hence the fiber of $f: M_{C, \delta}(E_0, s, d) \to M_C(s, d)$ over $[E]$ is the projective space $\mathbb{P}({\rm Hom}(E_0, E))$.

Furthermore,
if $\mu(E)=\frac{d}s > \mu_{\rm max}(E_0) +2g - 2$, where $\mu_{\rm max}(E_0) \coloneqq \max \{ \mu(F) \mid F \subset E_0\}$, then ${\rm ext}^1(E_0,E) = {\rm hom}(E, E_0\otimes \omega_C) = 0$ by the stability of $E$. Hence we have  
\[
\mathbb{P}({\rm Hom}(E_0, E))=\mathbb{P}^{se+ r(d -s(g-1)) - 1}
\]
and $f$ is a $\mathbb{P}^{se+ r(d - s(g-1)) - 1}$-bundle over the stable locus $M^s_C(s, d) \subset M_C(s, d)$.

Let $M_C(s,L)$ be the fiber of $ M_C(s, d) \to \Pic^d(C) $ over $[L]$.
By restriction,
we obtain $f_L : M_{C, \delta}(E_0, s, d)_L \to M_C(s, L)$, which is a projective bundle over $M^s_C(s, L) \coloneqq M_C(s, L) \cap M_C^s(s,d)$.

\end{remark}

\section{Estimate of dimensions}\label{sec_estimate_of_dimension}

Throughout this section, fix integers $s >0$ and $d$.

\subsection{Expected dimensions}

The following lemma is a consequence of \cite[Theorem 1.2]{Lin}.

\begin{lemma}\label{lem_dim}\label{subsec_tangent_space}
Let $\delta>0$ and  $[(E, \alpha)] \in M^s_{C, \delta}(E_0, s, d)$ be a point in the $\delta$-stable locus.
If  $\operatorname{Hom}(E, \omega_C \otimes \ker \alpha) =0$,
then
$M_{C, \delta}(E_0, s, d)$ is smooth of dimension $dr +se -s(r-s)(g-1)$ at $[(E,\alpha)]$.
\end{lemma}

\begin{proof}
Let $I^\bullet = \{ E_0 \xrightarrow{\alpha} E \}$ be a complex concentrated in degrees 0 and 1. 
By \cite[Theorem 1.2]{Lin}, we have
\begin{align}\label{eq_lem_dim_tangent_space}
\begin{split}
\hom (I^\bullet, E) &= \dim T_{[(E, \alpha)]} M_{C, \delta}(E_0, s, d)\\
&\geq \dim_{[(E, \alpha)]} M_{C, \delta}(E_0, s, d) \geq  \hom (I^\bullet, E) - \ext^1(I^\bullet, E).
\end{split}
\end{align}

Since there is a natural triangle 
\begin{align*}
  I^\bullet \to E_0 \xrightarrow{\alpha} E \xrightarrow{\beta} I^\bullet[1] , 
\end{align*}
we have an exact sequence
\begin{align}\label{eq_lem_dim}
\begin{aligned}
0 \to \operatorname{Ext}^{-1}(I^\bullet, E) &\to \operatorname{Hom}(E, E) \to \operatorname{Hom}(E_0, E) \to \operatorname{Hom}(I^\bullet, E)\\
&\to \operatorname{Ext}^1(E, E) \to \operatorname{Ext}^1(E_0, E) \to \operatorname{Ext}^1(I^\bullet, E) \to 0
\end{aligned}
\end{align}
and $\operatorname{Ext}^i(I^\bullet, E) = 0$ for $i \ne -1, 0, 1$. In fact, we have $\operatorname{Ext}^{-1}(I^\bullet, E) = 0$ as follows.

Take $\varphi$ in the kernel of $\operatorname{Hom}(E, E) \to \operatorname{Hom}(E_0, E)$, that is, $\varphi : E \to E$ such that $\varphi \circ \alpha = 0$. Then $\varphi : (E, \alpha) \to (E, \alpha)$ is a morphism of stable pairs, hence $\varphi = 0$ or an isomorphism by \cite[Lemma 2.11]{Lin}. Since $\alpha \ne 0$, we have $\varphi = 0$. 
Thus $\operatorname{Hom}(E, E) \to \operatorname{Hom}(E_0, E)$ is injective, and hence $\operatorname{Ext}^{-1}(I^\bullet, E) = 0$.

Then it holds that 
\[
\hom(I^\bullet, E) - \ext^1(I^\bullet, E) = \chi(I^\bullet, E) = \chi(E_0, E) - \chi(E, E) = dr +se- s(r - s)(g-1).
\]
Furthermore, by \eqref{eq_lem_dim} and Serre duality,
$\operatorname{Ext}^1(I^\bullet, E)$ is the dual of the kernel of
\[
\operatorname{Hom}(E, \omega_C \otimes E_0) \to \operatorname{Hom}(E, \omega_C \otimes E), \quad f \mapsto ({\rm id} \otimes \alpha) \circ f,
\]
and hence $\operatorname{Ext}^1(I^\bullet, E) = \operatorname{Hom}(E, \omega_C \otimes \ker \alpha)^\vee$.
Then this lemma follows from \eqref{eq_lem_dim_tangent_space}.
\end{proof}

\begin{remark}
For example, the condition $\operatorname{Hom}(E, \omega_C \otimes \ker \alpha) = 0$ in Lemma \ref{lem_dim} holds if
\begin{enumerate}
\item $r=1$, i.e., a Bradlow pair, or
\item $s = r$ and $ \delta >\delta_\sharp$, i.e., $M_{C,\delta}(E_0, s, d) = \operatorname{Quot}_C(E^0,0,d+e)$, or
\item  $s=1$ and $d > 2g-2 + \mu_{\max} (E_0)$, or $s \geq 2$ and $0 < \delta < \frac{d - 2s(g - 1)-s\mu_{\rm max}(E_0)}{s - 1}$.
\end{enumerate}

In fact, $\ker \alpha=0$ in the case (1) since $E_0$ is a line bundle and $\alpha \neq 0$.
For (2), $\ker \alpha=0$ also holds by Lemma \ref{lem_suff_large_delta}.
For (3),  assume $f : E \to \omega_C \otimes E_0$ is non-zero.
Then $(\im f, f \circ \alpha) $ is a quotient pair of $ (E,\alpha)$ and hence
\begin{align*}
\frac{d+\delta}{s} =\mu(E,\alpha) \leq \mu(\im f, f \circ \alpha)  &\leq \frac{\deg (\im f) +\delta}{\rank (\im f)} \\
&\leq \mu_{\max} (\omega_C \otimes E_0) + \delta =2g-2 +\mu_{\max}(E_0) +\delta
\end{align*}
since $(E,\alpha)$ is $\delta$-stable.
This contradicts the condition of $\delta$ in (3) and hence $\operatorname{Hom}(E, \omega_C \otimes E_0) = 0$.
In particular, $\operatorname{Hom}(E, \omega_C \otimes \ker \alpha) = 0$.

Hence $M^s_{C,\delta}(E_0, s, d)$ is smooth of dimension $dr+se - s(r - s)(g - 1)$ in these cases.
We note that  this statement for $\delta$ in (3) is proved in \cite[Theorem 3.20]{BDW} for $E_0=\calo_C^{\oplus r}$.
\end{remark}

\subsection{Upper bounds of the dimension of $M_{C, \delta}(E_0, s, d)$}

The purpose of this subsection is to show Lemma \ref{lemma-upper bound on dim}, which states that
\[
\dim M_{C, \delta}(E_0, s, d) \leq r(d-s \mu_{\min}(E_0)) + s \max \{s,r\}  (g - 1)
\]
for $s \geq 1$ and $\delta>0$,
where $\mu_{\min}(E_0) =\min \{ \mu(F) \mid E_0 \twoheadrightarrow F\}$.

\begin{lemma}\label{lemma-JH filtration}
Assume a pair $(E,\alpha)$ with $\alpha \neq 0$ is $\delta$-semistable with $\rank E=s, \deg E=d$. 
Let
\[
0 = (F_0,\alpha_0) \subset (F_1,\alpha_1) \subset (F_2,\alpha_2) \subset \cdots \subset (F_\ell,\alpha_\ell) = (E,\alpha)
\]
be a Jordan-H\"older filtration as in \cite[Proposition 2.13]{Lin}.
If $(E,\alpha)$ is $(\delta - \varepsilon)$-semistable for $0 < \varepsilon \ll 1$, then 
$\im \alpha \subset F_1$. Furthermore, if $\ell \geq 2$, then $F_1$ is the unique minimal bundle among those with $\im \alpha \subset F$ and $\mu(E/F) = \frac{d + \delta}{s}$.
\end{lemma}

\begin{proof}
By definition,
$(F_i,\alpha_i)/(F_{i-1},\alpha_{i-1})$ is $\delta$-stable with $\mu_\delta((F_i,\alpha_i)/(F_{i-1},\alpha_{i-1}))= \frac{d+\delta}{s}$.
If $\ell = 1$, the claim is trivial, so assume $\ell \geq 2$.

If $\im \alpha \not\subset F_1$, then
\[
\mu_{\delta-\ep}(F_1,\alpha_1) =\mu(F_1)=\mu_\delta(F_1,\alpha_1) = \frac{d+\delta}{s} > \frac{d+\delta-\ep}{s} =\mu_{\delta-\ep}(E,\alpha),
\]
and hence $(E, \alpha)$ is not $(\delta - \varepsilon)$-semistable.

For the second statement, suppose $F\subset E$ satisfies $\im  \alpha \subset F$ and $\mu(E/F)=\frac{d+\delta}{s}$. Suppose $F_1\not\subset F$. 
Then we have
\begin{align}\label{eq_lemma-JH filtration}
0 \to F_1 / (F_1 \cap F) \to E/F \to E/(F + F_1) \to 0.
\end{align}
Since 
$(F_1,\alpha_1)$ is $\delta$-stable with $\alpha_1 \neq 0$ and $ (F_1/(F_1\cap F),0 )$ is a proper quotient of $(F_1,\alpha_1)$ , we have
\begin{align}\label{eq_lemma-JH filtration2}
\mu(F_1 / (F_1 \cap F)) =\mu_\delta(F_1 / (F_1 \cap F),0) > \mu_\delta(F_1,\alpha_1) = \frac{d + \delta}{s}.
\end{align}
On the other hand, $E/F_1$ is semistable since it is obtained as successive extensions by $F_i/F_{i-1}$ for $2 \leq i \leq \ell$ and
$(F_i/F_{i-1},0)=(F_i,\alpha_i)/(F_{i-1},\alpha_{i-1}) $'s are $\delta$-stable with the same slope.
Since $E/F_1 \to E/(F + F_1)$ is a quotient,
\[
\mu(E/(F + F_1)) \geq \mu(E/F_1) =\mu_\delta(E/F_1,0) = \frac{d + \delta}{s},
\]
contradicting $\mu(E/F) = \frac{d + \delta}{s}$ by \eqref{eq_lemma-JH filtration} and \eqref{eq_lemma-JH filtration2}. Hence we get that $F_1\subset F$.
\end{proof}

We note that $(F_i, \alpha_i)/(F_{i-1}, \alpha_{i-1})=(F_i/F_{i-1}, 0) $ for all but exactly one $1 \leq i \leq \ell$ for any $\delta$-semistable pair $(E,\alpha)$ with $\alpha \neq 0$ in the above Jordan-H\"older filtration.
Furthermore, $\bigoplus_{i=1}^\ell (F_i, \alpha_i)/(F_{i-1}, \alpha_{i-1})$ does not depend on the filtration \cite[Proposition 2.13]{Lin}.
Hence we obtain the following decomposition of $M_{C,\delta}(E_0, s, d)$.

\begin{lemma}\label{lem_disjoint_union}
Fix $\delta >0$. Let $1 \leq s' \leq s - 1$ and $d'$ be integers  such that $\frac{d + \delta}{s} = \frac{d'}{s'}$. Then there exists an embedding
\[
M^s_{C,\delta}(E_0, s - s', d - d') \times M_C(s', d') \hookrightarrow M_{C,\delta}(E_0, s, d),
\]
given by
\[
([(E', \alpha')], [F]) \mapsto [(E' \oplus F, \alpha' \oplus 0)].
\]
Furthermore, $M_{C,\delta}(E_0, s, d)$ is the disjoint union of these images and $M^s_{C,\delta}(E_0, s, d)$.
\end{lemma}

\begin{proof}
We can check that $\frac{d-d' + \delta}{s-s'}$ coincides with $\frac{d + \delta}{s} = \frac{d'}{s'}$.
Thus $(E' \oplus F, \alpha' \oplus 0)$ in the statement of this lemma is $\delta$-semistable and hence we obtain the embedding.
The last statement follows from the uniqueness of the associated graded objects $\oplus (F_i, \alpha_i)/(F_{i-1}, \alpha_{i-1})$ in the Jordan-H\"older filtration.
\end{proof}

Recall that we fix $s,d$ in this section.
Now fix also $\delta$.

We note that $\delta'$-stability coincides with $\delta'$-semistability  if $\frac{d+\delta'}{s} \not \in \bigcup_{s'=1}^{s-1} \frac{1}{s'} \Z$ .
Hence, for $0 < \varepsilon \ll \delta, s,d$, if $(E, \alpha)$ is $(\delta + \varepsilon)$-semistable (resp.\ $(\delta - \varepsilon)$-semistable), then it is both $(\delta + \varepsilon)$-stable (resp.\ $(\delta - \varepsilon)$-stable) and $\delta$-semistable.
In particular,
there are natural morphisms
\[
\pi_{\delta}^{\pm}: M_{C, \delta \pm \varepsilon}(E_0, s, d) \to M_{C, \delta}(E_0, s, d).
\]

\begin{lemma}\label{lemma-fiber of contraction}
Let $s',d'$ be as in Lemma \ref{lem_disjoint_union}.
Let 
\[
([(E', \alpha')], [F]) \in M^s_{C,\delta}(E_0, s - s', d - d') \times M_C(s', d') \subset M_{C,\delta}(E_0, s, d)
\] 
and assume that $F$ is stable.
Then the fiber of $\pi_{\delta}^{-} : M_{C, \delta - \varepsilon}(E_0, s, d) \to M_{C,\delta}(E_0, s, d)$ over $([(E', \alpha')], [F])$ 
is the projective space
\[
\mathbb{P}({\rm Ext}^1(F, E')) = \mathbb{P}^{s'(\delta + (s - s')(g - 1)) - 1}.
\]
\end{lemma}

\begin{proof}
Let $[(E, \alpha)] \in M_{C, \delta - \varepsilon}(E_0, s, d)$ be a point in the fiber over $([(E', \alpha')], [F])$. Then there exists an exact sequence
\begin{align}\label{eq_lemma-fiber of contraction}
    0 \to E' \xrightarrow{\iota} E \to F \to 0
\end{align}
with $\alpha=\iota \circ \alpha'$
by Lemma \ref{lemma-JH filtration} and the stability of $F$. 
Note that this exact sequence is unique up to scalar
since $E'$ is the unique minimal subsheaf with $\im  \alpha\subset E'$,
$\mu(E/E')=\frac{d+\delta}{s}$, and $(E',\alpha'), (F,0)$ are $\delta$-stable.
If \eqref{eq_lemma-fiber of contraction} splits, we have an injection $f : F \to E$.
Then $f$ is a non-zero morphism $ (F,0 ) \to (E,\alpha)$ between $(\delta-\varepsilon)$-stable pairs with 
\[
\mu_{\delta-\ep}(F,0) = \mu(F)=\frac{d'}{s'} =\frac{d+\delta}{s} > \frac{d+\delta-\ep}{s} =\mu_{\delta-\ep}(E,\alpha),
\]
a contradiction.
Hence \eqref{eq_lemma-fiber of contraction} does not split. 
Thus, we have a point in $\P({\rm Ext}^1(F, E'))$.

Conversely, we have the following claim.

\begin{claim}\label{claim_lemma-fiber of contraction}
Let $0 \to E' \to E \to F \to 0$ be a non-split exact sequence.
For $\alpha : E_0 \xrightarrow{\alpha'} E' \to E$,
$(E, \alpha)$  is $(\delta-\ep)$-stable.
\end{claim}

\begin{proof}[Proof of Claim \ref{claim_lemma-fiber of contraction}]
If $(E, \alpha)$  is not $(\delta-\ep)$-stable, there exists a subsheaf $G \subset E$ with $G \cap E'=0$ and $\mu(G)=\mu(F)$
by Lemma \ref{lem_preliminary} (5).
Since $G \simeq G/(G\cap E') \subset F$ and $F$ is stable,
we have $G \simeq G/(G\cap E') = F$,
which  contradicts the non-splitness of  $0 \to E' \to E \to F \to 0$.
Hence $(E, \alpha)$ is $(\delta-\ep)$-stable.
\end{proof}

Thus the fiber over $([(E', \alpha')], [F])$ is $\mathbb{P}({\rm Ext}^1(F, E'))$.

Any $f \in {\rm Hom}(F, E')$ induces a morphism of pairs $f : (F, 0) \to (E', \alpha')$. Since both $(F, 0)$ and $(E', \alpha')$ are $\delta$-stable with the same slope, $f$ is either an isomorphism or zero. Since $\alpha' \ne 0$, we have $f = 0$ and hence ${\rm Hom}(F, E') = 0$.

Since $F^\vee \otimes E'$ has rank $s'(s - s')$ and degree $s'd - sd'$, we have
\[
\ext^1(F, E') = -s'd + sd' + s'(s - s')(g - 1) = s'\delta + s'(s - s')(g - 1).\qedhere
\]
\end{proof}

\begin{lemma}\label{lem_closure_of_stable_F}
In the setting of Lemma \ref{lemma-fiber of contraction},
the dimension of the inverse image $W$ of $ M^s_{C,\delta}(E_0, s - s', d - d') \times M_C(s', d') $ by $\pi_\delta^-$ is 
\begin{align*}
\dim M^s_{C,\delta}(E_0, s - s', d - d') + { s'(\delta + s(g-1))} 
\end{align*}
\end{lemma}

\begin{proof}
Let $W^o \subset W $ 
be the  inverse image of $ M^s_{C,\delta}(E_0, s - s', d - d') \times M^s_C(s', d') $ 
by $\pi_\delta^-$. 
Since $\pi_\delta^-|_{W^o} : W^o \to M^s_{C,\delta}(E_0, s - s', d - d') \times M^s_C(s', d') $ is a $\P^{s'(\delta+(s-s')(g-1))-1}$-bundle by Lemma \ref{lemma-fiber of contraction}
and $\dim M_C^s(s', d') =  s'^2(g-1)  +1$,
it suffices to show that $W$  is  contained in the closure of $W^o$.

Let $[(E, \alpha)] \in W$ and let $F_i$ be the subsheaves of $E$ which appear in a Jordan-H\"older filtration in Lemma \ref{lemma-JH filtration}. Then $\im \alpha \subset F_1 \eqqcolon E'$, and $E/E'$ is semistable with
\[
\mu(E/E') = \frac{d + \delta}{s} = \mu(E', \alpha').
\]

Let $\{G_t\}_{t \in T}$ be a flat family over a curve $T$ such that $G_0 = E/E'$ and $G_t$ is stable with $\mu(G_t) = \mu(G_0)$ for $t \ne 0$. 
Let $\eta \in {\rm Ext}^1(E/E', E')$ be the class corresponding to $0 \to E' \to E \to E/E' \to 0$. 
By Lemma \ref{lem_preliminary} (3), $\Hom (G_t, E') = 0$ for any $t$
since $G_t$ is semistable and $(E',\alpha')$ is $\delta$-stable with $\mu(G_t)=\mu(E',\alpha')$.
Thus $\{{\rm Ext}^1(G_t, E')\}_t$ is a vector bundle on $T$ and hence
we can deform $\eta \in {\rm Ext}^1(E/E', E')$ to $\eta_t \neq 0 \in {\rm Ext}^1(G_t, E')$ near $t = 0$.

For the exact sequence $0 \to E' \to E_t \to G_t \to 0$ corresponding to $\eta_t$, 
set $\alpha_t : E_0 \xrightarrow{\alpha'} E' \subset E_t$.
For $t \neq 0$, $G_t $ is stable and hence  $(E_t, \alpha_t)$ is $(\delta-\ep)$-stable by Claim \ref{claim_lemma-fiber of contraction}.
Since $0 \subset (E' , \alpha') \subset (E_t,\alpha_t)$ is a Jordan-H\"older filtration of $(E_t,\alpha_t)$,
we have $[(E_t,\alpha_t)] \in W^o$ for $t \neq 0$.
Hence $[(E,\alpha)]$ is contained in the closure of $W^o$.
\end{proof}

\begin{lemma}\label{lemma-upper bound on dim of Quot}
If $1 \leq s \leq r$,
it holds that
\[
\dim \Quot_C(E^0, r - s, d+e)  \leq r(d-s \mu_{\min}(E_0)) + s(r-s),
\]
where we set $\dim \emptyset =-\infty$.
\end{lemma}

\begin{proof}
Recall that $\mu_{\min}(E_0) =\min \{ \mu(F) \mid E_0 \twoheadrightarrow F\}$.
If $d  \ll 0$, this lemma is trivial since $\Quot_C(E^0, r - s, d+e) =\emptyset$.

Let $Z_i = \{ [E^0 \to F] \in \Quot_C(E^0, r - s, d+e) \mid \deg \Tor F=i  \}$ for $i \geq 0$.
Then we have a morphism 
\begin{align*}
Z_i \to \Quot_C(E^0, r - s, d+e-i) : [E^0 \to F] \mapsto [E^0 \to F/\Tor F].
\end{align*}
The fiber of this morphism over $[E^0 \xrightarrow{q} F'] \in \Quot_C(E^0, r - s, d+e-i)$ with locally free $F'$ is
$ \Quot_C(\ker q, 0, i)$, whose dimension is $si$.

Since there exists a natural embedding 
\[
Z_0 \hookrightarrow \Quot_C(E_0,  s, d) : [E^0 \to F] \mapsto [E_0 \to E_0/F^\vee],
\]
it holds that
\begin{align*}
\dim Z_0 \leq \dim  \Quot_C(E_0,  s, d) \leq r(d- s \mu_{\min}(E_0))+ s(r-s)
\end{align*}
by \cite[Theorem 4.1]{PR}.
For $i \geq 1$,
by the induction hypothesis on $d$, we have 
\begin{align*}
\dim Z_i &\leq \dim  \Quot_C(E^0, r - s, d+e-i) +si \\
&\leq r(d-i-s \mu_{\min}(E_0)) + s(r-s) + si \\
&\leq r(d- s \mu_{\min}(E_0))+ s(r-s).
\end{align*}
Hence this lemma holds.
\end{proof}

\begin{lemma}\label{lemma-upper bound on dim of M_dagger-ep, s >r_pre}
If $s >r$, then for $\delta_\sharp$ in \eqref{eq_dagger},
$ M_{C,\delta_\sharp}(E_0, s,d) $ is the disjoint union of 
\[
M^s_{C,\delta_\sharp}(E_0, k, e_k) \times M_C(s-k,d-e_k) 
\]
for $1 \leq k \leq r$
such that $\delta_\sharp= \frac{kd-s e_k }{s-k}$.
Furthermore, for such $k$ and $[(E',\alpha')] \in M^s_{C,\delta_\sharp}(E_0, k, e_k)$,
$\alpha'$ is surjective.
\end{lemma}

\begin{proof}
Recall that $M_{C,\delta_\sharp}(E_0,s,d)$ is the disjoint union of $M^s_{C,\delta_\sharp}(E_0,s,d) $ and
\[
M^s_{C,\delta_\sharp}(E_0, s-s',d-d') \times M_C(s',d') 
\]
for $1 \leq s' \leq s-1, d' \in \Z$ with $\frac{d+\delta_\sharp}{s} =\frac{d'}{s'} =\frac{d-d'+\delta_\sharp}{s-s'}$ as in Lemma \ref{lem_disjoint_union}.
Let  $[(E,\alpha)] \in M^s_{C,\delta_\sharp}(E_0,s,d) $ and set $k =\rank (\im \alpha)$.
Since $\im \alpha \subsetneq E$ by $s >r$, it holds that 
\begin{align*}
\frac{d +\delta_\sharp }{s} =\mu_{\delta_\sharp}(E,\alpha)> \mu_{\delta_\sharp}(\im \alpha,\alpha) =\frac{\deg (\im \alpha) +\delta_\sharp }{k} \geq \frac{e_k +\delta_\sharp }{k}
\end{align*}
for $e_k=\min\{ \deg F \mid E_0 \twoheadrightarrow F, \rank F=k \}$
which contradicts the definition of $\delta_\sharp$.
Hence $M^s_{C,\delta_\sharp}(E_0,s,d) =\emptyset$.

Let $[(E',\alpha')] \in M^s_{C,\delta_\sharp}(E_0, s-s',d-d')$
and $k=\rank (\im \alpha')$.
If $\im \alpha' \subsetneq E'$, then
\begin{align*}
 \frac{d+\delta_\sharp}s =\frac{d-d' +\delta_\sharp }{s-s'} =\mu_{\delta_\sharp}(E',\alpha') > \mu_{\delta_\sharp}(\im \alpha',\alpha')  =\frac{\deg (\im \alpha') +\delta_\sharp }{k} \geq \frac{e_k +\delta_\sharp }{k},
\end{align*}
which contradicts the definition of $\delta_\sharp$.
Hence  $\alpha'$ is surjective (in particular $s-s'=k \leq r$) and 
\[
 \frac{d+\delta_\sharp}s =\frac{d-d' +\delta_\sharp }{s-s'} = \frac{\deg E' +\delta_\sharp }{k} \geq \frac{e_k +\delta_\sharp }{k},
\]
which implies $d-d'= \deg E'=e_k$
and $\delta_\sharp = \frac{kd-s e_k }{s-k}$ by the definition of $\delta_\sharp$.
Thus $ M_{C,\delta_\sharp}(E_0, s,d) $ is the disjoint union of 
$M^s_{C,\delta_\sharp}(E_0, k, e_k) \times M_C(s-k,d-e_k) $ for $1 \leq k \leq r$
such that $\delta_\sharp= \frac{kd-s e_k }{s-k}$.
\end{proof}

\begin{lemma}\label{lemma-upper bound on dim of M_dagger-ep, s >r}
If $s >r$, we have
\[
\dim M_{C,\delta_\sharp -\ep} (E_0,s,d) \leq r(d-s \mu_{\min}(E_0)) + s^2 (g-1)
\]
for $\delta_\sharp$ in \eqref{eq_dagger} and sufficiently small $\ep >0$.
\end{lemma}

\begin{proof}
Let $k$ be as in Lemma \ref{lemma-upper bound on dim of M_dagger-ep, s >r_pre}.
Since $\alpha'$ is surjective for any $[(E',\alpha')] \in M^s_{C,\delta_\sharp}(E_0, k, e_k)$,
there exists an embedding $M^s_{C,\delta_\sharp}(E_0,k,e_k)\subset \Quot(E_0,k,e_k)$. 
By Lemmas \ref{lemma-fiber of contraction}, \ref{lem_closure_of_stable_F}, \ref{lemma-upper bound on dim of M_dagger-ep, s >r_pre} and \cite[Theorem 4.1]{PR},
$\dim M_{C,\delta_\sharp-\varepsilon}(E_0,s,d)$ is bounded above by the maximum of 
\begin{align*}
\dim \Quot(E_0,k,e_k)  &+ { (s-k)(\delta_\sharp+s(g-1))} \\
     & { \leq } k(r-k) +(s-k)\delta_\sharp + (s-k)s(g-1) \\
    & = (kd-s e_k) + k(r-k)+ (s-k)s(g-1) \\
    & \leq  k(d-s\mu_{\min}(E_0))+k(r-k)(g-1)+ (s-k)s(g-1)\\
    & =  k(d-s\mu_{\min}(E_0))+(s^2-(s-r)k-k^2)(g-1) \\
    & \leq r(d-s\mu_{\min}(E_0)) +s^2(g-1)
\end{align*}
over such $k$.
\end{proof}

\begin{lemma}\label{lemma-upper bound on dim}
For $s \geq 1$ and $\delta' > 0$, we have
\[
\dim M_{C, \delta'}(E_0, s, d) \leq r(d-s \mu_{\min}(E_0)) + s t (g - 1)
\]
for $t \coloneqq \max \{s,r\}$.
\end{lemma}

\begin{proof}
We prove this lemma by induction on $s$.
If $s=1$,
$ M_{C,\delta} (E_0,s,d)= \Quot_C(E^0,r-s,d+e)$ by Lemma \ref{lem_suff_large_delta}
and hence the inequality in this lemma holds by Lemma \ref{lemma-upper bound on dim of Quot}.

Let $s >1$.
We show this lemma for $s $ by descending induction on $\delta'$. 
If $\delta' > \delta_\sharp$, the statement of this lemma holds.
In fact, this is trivial if $s >r$ since $M_{C, \delta'}(E_0, s, d) =\emptyset$ by Lemma \ref{lem_suff_large_delta}.
If $s \leq r$,
then $M_{C, \delta'}(E_0, s, d) = \Quot_C(E^0, r - s, d+e)$ and hence this lemma follows from 
Lemma \ref{lemma-upper bound on dim of Quot} in this case.

Assume $\delta' \leq \delta_\sharp$.
By Lemma \ref{lemma-fiber of contraction},
$\pi_{\delta'}^- : M_{C, \delta' -\ep}(E_0, s, d) \to M_{C, \delta'}(E_0, s, d) $ is surjective for sufficiently small $\ep >0$.
Hence it suffices to show $\dim M_{C, \delta' - \ep }(E_0, s, d) \leq r(d-s \mu_{\min}(E_0)) + s t (g - 1)$.

Let $\delta > 0$ and define
\begin{align}\label{eq_W^-}
W^-_\delta \coloneqq \left\{ [(E, \alpha)] \in M_{C,\delta - \varepsilon}(E_0, s, d) \,\middle|\, (E, \alpha) \text{ is not } \delta\text{-stable} \right\}
\end{align}
for $0 < \varepsilon \ll \delta$. 
By Lemmas \ref{lem_disjoint_union}, \ref{lem_closure_of_stable_F},
we have
\begin{align}\label{eq_dim_W^-}
\dim W^-_\delta &=\max_{(s',d')} \{ \dim M^s_{C,\delta}(E_0, s - s', d - d') + s'(\delta + s(g - 1)) \},
\end{align}
where we take the maximum for all $(s',d')$ as in Lemma \ref{lemma-fiber of contraction}.
Since 
\[
\dim M^s_{C,\delta}(E_0, s - s', d - d') \leq r(d - d'- (s-s')\mu_{\min}(E_0) ) + (s - s')t(g - 1)
\]
by induction on $s$,
$\dim W^-_\delta $ is at most the maximum of 
\begin{align*}
r(d - d'&- (s-s')\mu_{\min}(E_0) ) + (s - s')t (g - 1)  + s'(\delta + s(g - 1)) \\
&=r(d-s \mu_{\min}(E_0) ) + st(g - 1) -  (rd' - rs' \mu_{\min}(E_0) - s'\delta) - s'(t - s)(g - 1).
\end{align*}
By the condition $\frac{d+\delta}{s} =\frac{d'}{s'}$, we have
\begin{align}
\label{eq_lemma-upper bound on dim}
\begin{split}
rd' - rs' \mu_{\min}(E_0) - s'\delta =rs'\left( \frac{d + \delta}{s} -  \mu_{\min}(E_0) -  \frac{\delta}r \right) .
\end{split}
\end{align}
If $W^-_\delta  \neq \emptyset$,
we can take  $[(E,\alpha)] \in W^-_\delta $, which is $\delta$-semistable,  and hence
\begin{align*}
\frac{d + \delta}{s} =\mu_\delta(E,\alpha) \geq \mu_\delta(\im \alpha,\alpha) =\frac{\deg(\im \alpha) + \delta}{{\rm rank}(\im \alpha)} \geq \mu_{\min}(E_0)  +   \frac{\delta}{r}.
\end{align*}
Thus \eqref{eq_lemma-upper bound on dim} is non-negative.
Since $s'(t - s)(g - 1) $ is non-negative as well,
$\dim W^-_\delta $ is at most $r(d-s \mu_{\min}(E_0) ) + st(g - 1)$.

Since $M_{C, \delta - \varepsilon}(E_0, s, d) \setminus W^-_\delta$ is embedded into $M_{C, \delta + \varepsilon}(E_0, s, d)$, 
this lemma follows by Lemmas \ref{lemma-upper bound on dim of Quot}, \ref{lemma-upper bound on dim of M_dagger-ep, s >r} and descending induction on $\delta$.
\end{proof}

\subsection{Dimension of the locus where $M_{C,\delta - \varepsilon}(E_0, s, d) \to M_{C,\delta}(E_0, s, d)$ is not an isomorphism for $d \gg 0$}

In the rest of this paper, when we say that
$d$ is \emph{sufficiently large}, $d$ could depend on $g,E_0,r,s$ but not on $\delta$.

In this subsection, we estimate the dimension of 
\begin{align*}
W^-_\delta = \left\{ [(E, \alpha)] \in M_{C,\delta - \varepsilon}(E_0, s, d) \,\middle|\, (E, \alpha) \text{ is not } \delta\text{-stable} \right\}
\end{align*}
in \eqref{eq_W^-} for sufficiently large $d$.
Recall that the expected dimension of $M_{C,\delta - \varepsilon}(E_0, s, d) $ is 
\[
 \mathrm{exp.dim} \,  M_{C,\delta - \varepsilon}(E_0, s, d) \coloneqq rd +se -s(r-s)(g-1)
\]
by Lemma \ref{lem_dim}.

First, we consider the case $1 \leq s \leq r$.

\begin{lemma}\label{lemma-upper bound on dim of W}
Fix $1 \leq s \leq r$ and $c \geq 0$. 
Then there exists an integer $d_c$ such that for all $d \geq d_c$ and all $\delta >0$, it holds that
\[
\dim W^-_\delta \leq \mathrm{exp.dim} \, M_{C,\delta - \varepsilon}(E_0, s, d) - c.
\]
\end{lemma}

\begin{proof}
By the proof of Lemma \ref{lemma-upper bound on dim}, 
$\dim W^-_\delta $ is at most the maximum of 
\[
 r(d - d'- (s-s')\mu_{\min}(E_0) )  + (s - s')t(g - 1) + s'(\delta + s(g - 1)).
\]
Hence, $\mathrm{exp.dim} \, M_{C,\delta - \varepsilon}(E_0, s, d)  - \dim W^-_\delta $ is at least the minimum of 
\begin{align*}
rd+se &- s(r - s)(g - 1) -  r(d - d'- (s-s')\mu_{\min}(E_0) )
 - (s - s')t(g - 1) - s'(\delta + s(g - 1)) \\
&= rd' -s' \delta +se +r(s-s')\mu_{\min}(E_0)  - (s(r-s)  + (s-s')t +s's )(g - 1)
\\
&= s'r  \left( \frac{d+\delta}{s} -\frac{\delta}{r} \right) +se +r(s-s')\mu_{\min}(E_0)  -  (s(r-s)  + (s-s')t +s's )(g - 1)\\
& \geq  s'r \cdot \frac{d}s +se +r(s-s')\mu_{\min}(E_0)  - (s(r-s)  + (s-s')t +s's )(g - 1) ,
\end{align*}
which is larger than $c$ for sufficiently large $d$.
Note that we use $\frac{d+\delta}{s} =\frac{d'}{s'}$ (resp.\ $s \leq r$) in the second equality (resp.\ in the last inequality).
\end{proof}

Next, we consider the case $s > r$.

\begin{lemma}\label{lemma-upper bound on dim of W,s >r}
Fix $ s > r$.
Set $\kappa \coloneqq  \frac{(r-1)d-se_{r-1}}{s-r+1}$.
Then
\begin{enumerate}
\item For $c \geq 0$, there exists an integer $d_c$ such that for all $d \geq d_c$ and all $0 < \delta \leq  \kappa$,
it holds that
\[
\dim W^-_\delta \leq \mathrm{exp.dim} \, M_{C,\delta - \varepsilon}(E_0, s, d)  - c.
\]
\item Assume $d \gg 0$.
Then, for $\kappa < \delta <\delta_\sharp$,
$ M^s_{C,\delta}(E_0, s, d)$ is smooth with $\dim  M^s_{C,\delta}(E_0, s, d) =\mathrm{exp.dim} \, M_{C,\delta - \varepsilon}(E_0, s, d) $
if it is non-empty\footnote{We will see that $ M^s_{C,\delta}(E_0, s, d)\neq \emptyset$ in Proposition \ref{prop_SQM_s>r}.}.
Furthermore,
\begin{align*}
\dim W^-_\delta 
\begin{cases}
\leq \mathrm{exp.dim} \, M_{C,\delta - \varepsilon}(E_0, s, d) -2  & \text{ if } (s, \delta ) \neq (r+1,\kappa_\sharp) \\
= \mathrm{exp.dim} \, M_{C,\delta - \varepsilon}(E_0, s, d) -1  & \text{ if } (s, \delta ) = (r+1,\kappa_\sharp) 
\end{cases}
\end{align*}
for $\kappa_\sharp \coloneqq rd +(r+1)(e-1)$.
\item Assume $d \gg 0$. Then we have an isomorphism 
\[
 M_{C,\delta_\sharp}(E_0, s, d) \xrightarrow{\sim} M_C(s-r,d+e) : [(E,\alpha)] \mapsto [E/\im \alpha].
\]
\end{enumerate}
\end{lemma}

\begin{proof}
(1) By the same argument as in the proof of Lemma \ref{lemma-upper bound on dim of W},
$\mathrm{exp.dim} \, M_{C,\delta - \varepsilon}(E_0, s, d) - \dim W^-_\delta $ is at least the minimum of 
\begin{align*}
 s'r  \left( \frac{d+\delta}{s} -\frac{\delta}{r} \right) +se +r(s-s')\mu_{\min}(E_0)  -  (s(r-s)  + (s-s')t +s's )(g - 1).
\end{align*}
If $\delta \leq \kappa$, we have
\begin{align*}
\frac{d+\delta}{s} -\frac{\delta}{r}= \frac{rd -(s-r)\delta}{sr} \geq \frac{rd -(s-r)\kappa}{sr} = \frac{d+(s-r)e_{r-1}}{r(s-r+1)}
\end{align*}
and hence (1) holds.\\
(2) Recall that
\begin{align*}
\delta_\sharp &= \max\limits_{1 \leq i \leq r} \frac{id-s e_i }{s-i} =\min \left\{\delta \geq 0 \mid \frac{d+\delta}{s} \leq \frac{e_i+\delta}{i} \text{ for any } 1 \leq i \leq r\right\} ,
\end{align*}
which coincides with $ \frac{rd+se}{s-r} $ by $d \gg 0$ and $e_r=\det E_0=-e$.
Similarly, $\kappa$ coincides with
\begin{align*}
\max\limits_{1 \leq i \leq r-1} \frac{id-s e_i }{s-i}=\min \left\{\delta \geq 0 \mid \frac{d+\delta}{s} \leq \frac{e_i+\delta}{i} \text{ for any } 1 \leq i \leq r-1 \right\}
\end{align*}
by $d \gg 0$.
We also note that, for $d \gg 0$, 
we have 
\begin{align}\label{eq_lemma-upper bound on dim of W,s >r}
\kappa <\delta <\delta_\sharp \ \Leftrightarrow \ \frac{-e +\delta}{r} < \frac{d+\delta}{s} < \frac{e_{r-1}+\delta}{r-1}.
\end{align}

Let $\kappa < \delta < \delta_\sharp$ and $[(E,\alpha)] \in M_{C,\delta}(E_0, s, d)$.
If $\alpha$ is not injective,
for $k \coloneqq \rank (\im \alpha) \leq r-1$, it holds that
\begin{align*}
\frac{d+\delta}{s} =\mu_{\delta} (E,\alpha)  \geq \mu_{\delta} (\im \alpha,\alpha) = \frac{\deg (\im \alpha) + \delta}{k} \geq \frac{e_k+\delta}{k},
\end{align*}
which contradicts $\delta >\kappa$. 
Hence $\alpha$ is injective.
Then $ M^s_{C,\delta}(E_0, s, d)$ is smooth of dimension $\mathrm{exp.dim} \, M_{C,\delta - \varepsilon}(E_0, s, d) $ at $[(E,\alpha)]$ by Lemma \ref{lem_dim}.

By \eqref{eq_dim_W^-}, we have
\[
\dim W^-_\delta = \max_{(s',d')} \{ \dim M^s_{C,\delta}(E_0, s - s', d - d') + s'(\delta + s(g - 1)) \},
\]
where we take the maximum for all $(s',d')$ as in Lemma \ref{lemma-fiber of contraction} with $M^s_{C,\delta}(E_0, s - s', d - d') \neq \emptyset$.
For $[(E',\alpha')] \in M^s_{C,\delta}(E_0, s - s', d - d')$,
we can show that $\alpha'$ is injective 
by using $ \frac{d+\delta}{s} =\frac{d-d'+\delta}{s-s'}$ and the above argument.
In particular, $s-s' \geq r$ 
and $M^s_{C,\delta}(E_0, s - s', d - d')$ is smooth of dimension
\begin{align*}
r(d-d') + (s-s')e +(s-s')(s-s' -r) (g-1) 
\end{align*}
if $M^s_{C,\delta}(E_0, s - s', d - d') \neq \emptyset$ by Lemma \ref{lem_dim}.
Then $\mathrm{exp.dim} \, M_{C,\delta - \varepsilon}(E_0, s, d) - \dim W^-_\delta$ is the minimum of 
\begin{align}\label{eq_codim}
\begin{split}
(rd &+se +s(s-r)(g-1)) \\
&- (r(d-d') + (s-s')e +(s-s')(s-s' -r) (g-1) ) - s'(\delta + s (g - 1))  \\
&= rd' +s'e -s' \delta + s'(s-s'-r) (g-1)\\
&=rs'\left( \frac{d'}{s'} + \frac{e}{r}- \frac{\delta}{r} \right)+ s'(s-s'-r) (g-1)\\
&=rs'\left( \frac{d+\delta}{s} - \frac{-e +\delta}{r} \right)+ s'(s-s'-r) (g-1).
\end{split}
\end{align}
By \eqref{eq_lemma-upper bound on dim of W,s >r},
$ \frac{d+\delta}{s} - \frac{-e +\delta}{r}  $ is positive.
We already see that $ s-s' \geq r$.
Thus if $s -s' >r$, the integer \eqref{eq_codim} is at least two.

Assume $s-s'=r$.
Then the positive integer \eqref{eq_codim} is equal to
\begin{align*}
rs'\left( \frac{d+\delta}{s} - \frac{-e+\delta}{r} \right) &=rs'\left( \frac{d-d'+\delta}{s-s'} - \frac{-e+\delta}{r} \right)  \\
&= rs'\left( \frac{d-d'+\delta}{r} - \frac{-e+\delta}{r} \right) 
=(s-r)(d-d'+e).
\end{align*}
This is equal to one if and only if $s=r+1$ and $d' =d+e-1$.
By  $ \frac{d+\delta}{s} =\frac{d-d'+\delta}{s-s'}$,
this is equivalent to $ s=r+1$ and $ \delta= rd + (r+1)(e-1) =\kappa_\sharp$.

Summarizing the discussion above, we have 
$\mathrm{exp.dim} \, M_{C,\delta - \varepsilon}(E_0, s, d)  - \dim W^-_\delta \geq 2$  if $(s, \delta ) \neq (r+1,\kappa_\sharp) $.
If $(s, \delta ) = (r+1,\kappa_\sharp) $,
then $(s',d')=(1,d+e-1)$ satisfies $\frac{d+ \delta}{s} = \frac{d-d'+\delta}{s-s'}$ and 
\[
M^s_{\kappa_\sharp}(E_0,s-s',d-d') = M^s_{\kappa_\sharp}(E_0,r,1-e) = \Quot_C(E^0,0,1) =\P_C(E_0) \neq \emptyset.
\]
Hence we have $ \dim W^-_{\kappa_\sharp}  = \mathrm{exp.dim} \, M^s_{C,\kappa_\sharp - \varepsilon}(E_0, s, d) -1$ in this case.\\
(3) 
Since $d \gg 0$, the maximum in 
$\delta_\sharp =\max\limits_{1 \leq i \leq r} \frac{id-s e_i }{s-i}$ is attained uniquely by $i=r$.
Hence \begin{align*}
M_{C,\delta_\sharp}(E_0, s, d) &=M^s_{C,\delta_\sharp}(E_0, r, -e) \times M_C(s-r, d+e) \\
&=\{[E_0,\id_{E_0}]\} \times M_C(s-r, d+e)  \simeq M_C(s-r, d+e) .
\end{align*}
by Lemma \ref{lemma-upper bound on dim of M_dagger-ep, s >r_pre}.
By construction, this isomorphism maps $[(E,\alpha)] \in M_{C,\delta_\sharp}(E_0, s, d) $ to $[E/\im \alpha] \in M_C(s-r, d+e)$.
\end{proof}

\subsection{The case $ M_{C,\delta}(E_0, s, d)_L$}

In this subsection, we show results similar to those in the previous subsections for $ M_{C,\delta}(E_0, s, d)_L$.
We set
\[
 \mathrm{exp.dim} \, M_{C,\delta}(E_0, s, d)_L \coloneqq  \mathrm{exp.dim} \, M_{C,\delta}(E_0, s, d) -g =rd +se -s(r-s)(g-1)-g.
\]
The following is an analogue of a part of Lemma \ref{lemma-upper bound on dim of W,s >r} (2).

\begin{lemma}\label{lem_smoothness_M_L}
Let $ s > r$ and assume $d \gg 0$.
For  $\kappa < \delta <\delta_\sharp$,
$ M^s_{C,\delta}(E_0, s, d)_L$ is smooth with $\dim M_{C,\delta}(E_0, s, d)_L =  \mathrm{exp.dim} \, M_{C,\delta}(E_0, s, d)_L$  if it is non-empty.
\end{lemma}

\begin{proof}
Since $M^s_{C,\delta}(E_0, s, d)$ is smooth of dimension  $ rd +se -s(r-s)(g-1)$ by Lemma \ref{lemma-upper bound on dim of W,s >r} (2),
it suffices to show the smoothness of the morphism 
\[
M^s_{C,\delta}(E_0, s, d)\to \Pic^{d}(C), \ [(E,\alpha)]\mapsto [\det E].
\]

By \cite[Theorem 1.2]{Lin}, we have an isomorphism $ T_{[(E,\alpha)]} M^s_{C,\delta}(E_0, s, d) \simeq \Hom(I^\bullet, E)$, where $I^\bullet =\{E_0 \xrightarrow{\alpha} E\}$.
Since $\alpha$ is injective for such $[(E,\alpha)] $ by the proof of Lemma \ref{lemma-upper bound on dim of W,s >r} (2),
we have $I^\bullet \simeq E/E_0 [-1]$ and hence  $ T_{[(E,\alpha)]} M^s_{C,\delta}(E_0, s, d) \simeq \Ext^1(E/E_0, E)$.
We see that the isomorphism $ T_{[(E,\alpha)]} M^s_{C,\delta}(E_0, s, d) \simeq \Ext^1(E/E_0, E)$ is obtained as follows:

Let $\eta' \in \Ext^1(E/E_0, E) $
and let  $\eta \in  \Ext^1(E, E)$ be the image of $\eta'$ by the natural map $\Ext^1(E/E_0, E) \to  \Ext^1(E, E) $.
Then we have a diagram 
\[
\begin{tikzcd}
\eta : 0\ar[r] & E \ar[r, "\iota"]   \ar[d,equal] & \cale \ar[r, "\pi"]  \ar[d, "\tilde{q}"] & E \ar[r]  \ar[d, "q"] & 0\\
\eta' : 0\ar[r] & E \ar[r] & \cale' \ar[r] & E/E_0 \ar[r] & 0
\end{tikzcd}
\]
Hence $\ker \tilde{q} \simeq \ker q $, which is isomorphic to $E_0$ by $\alpha$.
Denote by $\beta$ the composite map  $E_0 \simeq \ker \tilde{q}  \hookrightarrow \cale$.
Then $\pi\circ  \beta =\alpha$.
For $\tilde{\alpha} =(\beta, \iota \circ \alpha) : E_0 \oplus \ep E_0 \to \cale$, we have a commutative diagram
\[
\begin{tikzcd}
0 \ar[r] &  E_0 \ar[r, "\times \ep"] \ar[d, "\alpha"] & E_0 \oplus \ep E_0  \ar[r, "p_1"] \ar[d, "\tilde{\alpha}"] & E_0 \ar[r] \ar[d, "\alpha"] & 0 \\
0\ar[r] & E \ar[r, "\iota"]  & \cale \ar[r, "\pi"]   & E \ar[r]   & 0
\end{tikzcd}
\]
where $p_1$ is the projection to the first factor.
Let $\tilde{C} =C \times \Spec \C[\ep]/(\ep^2) $.
The sheaf $\cale$ has a structure of  $\calo_{\tilde{C}}$-module by defining the multiplication map $\times \varepsilon \coloneqq \iota \circ \pi : \cale \to \cale$.
Then $\tilde{\alpha} :  E_0 \oplus \ep E_0 \to \cale$ is a homomorphism of $\calo_{\tilde{C}}$-modules.
In fact, 
\[
\begin{tikzcd}
E_0 \oplus \ep E_0 \ar[r, "\times \ep"] \ar[d, "\tilde{\alpha}"] & E_0 \oplus \ep E_0  \ar[d,"\tilde{\alpha}"]\\
\cale \ar[r, "\times  \ep"] & \cale
\end{tikzcd}
\]
is commutative since for $(a,\ep b) \in E_0 \oplus \ep E_0$,
\begin{align*}
\tilde{\alpha} (\ep (a,\ep b)) = \tilde{\alpha} (0,\ep a) = \iota (\alpha (a)) \quad \text{and} \quad  \ep \tilde{\alpha} (a,\ep b) =\iota \circ \pi  (\beta(a) + \iota (\alpha(b))) = \iota (\alpha(a)) 
\end{align*}
coincide.
Hence
we have a morphism $\Spec \C[\ep]/(\ep^2) \to M^s_{C,\delta}(E_0, s, d) $, which maps the closed point of $\Spec \C[\ep]/(\ep^2)$ to $[(E,\alpha)]$ by $\tilde{\alpha}|_C =\alpha$.
Thus we have a map $ \Ext^1(E/E_0, E) \to  T_{[(E,\alpha)]} M^s_{C,\delta}(E_0, s, d)$.

Conversely,
take an element of $T_{[(E,\alpha)]} M^s_{C,\delta}(E_0, s, d) $, which corresponds to 
a pair 
\[
 \tilde{\alpha} : E_0 \oplus \ep E_0 \to \cale 
\]
on $\tilde{C}$
such that $\cale$ is locally free  and  $ \tilde{\alpha}|_C  : E_0 \to \cale|_C$ coincides with $\alpha : E_0 \to E$.
By $0 \to (\ep)/(\ep^2) \to \C[\ep]/(\ep^2) \to \C \to 0$ and $(\ep)/(\ep^2)  \simeq \C$ as $\C[\ep]/(\ep^2)$-modules,
$ \cale \otimes (\ep)/(\ep^2) \simeq \cale|_C =E$ and hence 
we have a diagram
\begin{equation}\label{eqn_deformations}
\begin{tikzcd}
0 \ar[r] &  E_0 \ar[r, "\times \ep"] \ar[d, "\alpha"] & E_0 \oplus \ep E_0  \ar[r] \ar[d, "\tilde{\alpha}"] & E_0 \ar[r] \ar[d, "\alpha"] & 0 \\
0\ar[r] & E \ar[r] \ar[d, equal] & \cale \ar[r] \ar[d] & E \ar[r] \ar[d]& 0\\
0\ar[r] & E \ar[r] & \cale/\tilde{\alpha}(E_0 \oplus \ep 0) \ar[r] & E/E_0 \ar[r] & 0
\end{tikzcd}
\end{equation}
with exact rows.
We note that the lowest row is an exact sequence of $\calo_C$-modules, not  $\calo_{\tilde{C}}$-modules.
The lowest row corresponds to an element $\eta' \in \Ext^1(E/E_0,E)$ and hence we have  a map 
$T_{[(E,\alpha)]} M^s_{C,\delta}(E_0, s, d) \to  \Ext^1(E/E_0, E) $.
It is easy to check that this is the inverse of the above map.

To show the smoothness of $M^s_{C,\delta}(E_0, s, d)\to \Pic^{d}(C), \, [(E,\alpha)]\mapsto [\det E]$, it is enough to show the smoothness of the morphism
\[
f : M^s_{C,\delta}(E_0, s, d)\to \Pic^{d}(C)\xrightarrow{\sim} \Pic^{d+e}(C), \  [(E,\alpha)]\mapsto [\det (E/E_0)].
\]
Let
$\tilde{\alpha} : E_0 \oplus \ep E_0 \to \cale$
on $\tilde{C}$
be an element of $T_{[(E,\alpha)]} M^s_{C,\delta}(E_0, s, d) $ and $\eta'\in \Ext^1(E/E_0,E)$ be the corresponding element under the above identification. 
Then under the differential 
\begin{equation}\label{eqn_tangent_space}
df_{[(E,\alpha)]}: T_{[(E,\alpha)]} M^s_{C,\delta}(E_0, s, d) \to T_{[\det(E/E_0)]}{\Pic^{d+e}}(C),
\end{equation}
the morphism $\tilde{\alpha}$
is mapped to $\det(\cale/\tilde{\alpha}(E_0 \oplus \ep E_0))$, which is a line bundle on $\tilde{C}$.

Now recall that an element  
$\eta''$ of ${\rm Ext}^1(E/E_0, E/E_0)$, that is, an extension 
\[
\eta'':0\to E/E_0\to \mc E''\to E/E_0\to 0
\]
corresponds to a deformation of the sheaf $E/E_0$: the sheaf $\mc E''$ has a structure of  $\mc O_{\tilde{C}}$-module with multiplication by $\varepsilon$ given by $\mc E''\to E/E_0\hookrightarrow \mc E''$ and
$\mc E''$ is flat over $\C[\varepsilon]$ with $\mc E''|_C\simeq E/E_0$. 
Now $\cale/\tilde{\alpha}(E_0 \oplus \ep E_0)$ is a deformation of $E/E_0$ and it follows from \eqref{eqn_deformations} that we have the diagram
\[
\begin{tikzcd}
    0\ar[r] & E \ar[r] \ar[d] & \cale/\tilde{\alpha}(E_0 \oplus \ep 0) \ar[r] \ar[d] & E/E_0 \ar[r] \ar[d,equal] & 0 \\
    0\ar[r] & E/E_0 \ar[r] & \cale/\tilde{\alpha}(E_0 \oplus \ep E_0) \ar[r] & E/E_0 \ar[r] & 0
\end{tikzcd}
\]
Hence
the corresponding element $\eta''$ of $\cale/\tilde{\alpha}(E_0 \oplus \ep E_0)$ in $\Ext^1(E/E_0,E/E_0)$ is the image of $\eta'$ under the natural map 
$\Ext^1(E/E_0,E)\to \Ext^1(E/E_0,E/E_0)$. 

Now recall that we have the trace map
${\rm Ext}^1(E/E_0, E/E_0)\xrightarrow{\mathrm{tr}}H^1(C,\calo_C)$  \cite[\S 10.1.2]{HL}. 
Under the trace map, the deformation $\mc E''$  is mapped to  the deformation $\det \mc E''$ of the line bundle $\det E/E_0$ \cite[Theorem 3.23]{Thomas}. 
Therefore, we see that the differential \eqref{eqn_tangent_space} coincides with  the  composite map
\[
\Ext^1(E/E_0,E)\to \Ext^1(E/E_0,E/E_0)\xrightarrow{\mathrm{tr}} H^1(C,\calo_C)
\]
under the above identification.
The map $\Ext^1(E/E_0,E)\to \Ext^1(E/E_0,E/E_0)$ is surjective by $\dim C=1$,  and since $\rank E/E_0\geq 1$, the trace map is surjective \cite[Lemma 10.1.3]{HL}. 
Therefore \eqref{eqn_tangent_space} is surjective and hence $f$ is smooth.
\end{proof}

The following is an analogue of Lemmas \ref{lemma-upper bound on dim of W}, \ref{lemma-upper bound on dim of W,s >r} for $M_{C,\delta - \varepsilon}(E_0, s, d)_L  $.

\begin{lemma}\label{lem_similar_to_M_L}
For $\delta>0$,  set
\[
W^-_{\delta,L} \coloneqq W^-_{\delta}  \cap M_{C,\delta - \varepsilon}(E_0, s, d)_L = \left\{ [(E, \alpha)] \in M_{C,\delta - \varepsilon}(E_0, s, d)_L \,\middle|\, (E, \alpha) \text{ is not } \delta\text{-stable} \right\}.
\]
\begin{enumerate}
    \item Assume $1 \leq s \leq r$ (resp.\ $s >r$).
    For $c \geq 0$, there exists an integer $d_c$ such that for all $d \geq d_c$ and all $\delta >0$ (resp.\ $0 < \delta \leq \kappa$), 
    it holds that
\[
\dim W^-_{\delta,L} \leq \mathrm{exp.dim} \, M_{C,\delta - \varepsilon}(E_0, s, d)_L - c.
\]
\item Assume $s >r $ and $d \gg 0$.
Then for  $\kappa < \delta <\delta_\sharp$,
\begin{align*}
\dim W^-_{\delta,L} 
\begin{cases}
\leq \mathrm{exp.dim} \, M_{C,\delta - \varepsilon}(E_0, s, d)_L -2  & \text{ if } (s, \delta ) \neq (r+1,\kappa_\sharp) \\
= \mathrm{exp.dim} \, M_{C,\delta - \varepsilon}(E_0, s, d)_L -1  & \text{ if } (s, \delta ) = (r+1,\kappa_\sharp) 
\end{cases}
\end{align*}
for $\kappa_\sharp \coloneqq rd +(r+1)(e-1)$.
\end{enumerate}
\end{lemma}

\begin{proof}
Since $\mathrm{exp.dim} \, M_{C,\delta - \varepsilon}(E_0, s, d)_L =\mathrm{exp.dim} \, M_{C,\delta - \varepsilon}(E_0, s, d) -g$ and we fix $g$,
(1) follows from Lemma \ref{lemma-upper bound on dim of W}, \ref{lemma-upper bound on dim of W,s >r} (1).

For (2), we note that 
$M_{C,\delta}(E_0,s,d)_L \setminus M^s_{C,\delta}(E_0,s,d)_L $ is the disjoint union of 
\begin{align}\label{eq_lem_similar_to_M_L}
\left(   M^s_{C,\delta}(E_0, s - s', d - d') \times M_C(s', d')\right) \cap M_{C,\delta}(E_0,s,d)_L
\end{align}
for all $(s',d')$ as in Lemma \ref{lemma-fiber of contraction} with $M^s_{C,\delta}(E_0, s - s', d - d') \neq \emptyset$.
By the proof of Lemma \ref{lemma-upper bound on dim of W,s >r} (2),
$s-s' \geq r$ and  $\alpha'$ is injective for $[(E',\alpha')] \in M^s_{C,\delta}(E_0, s - s', d - d') $.
In particular,  $\dim M^s_{C,\delta}(E_0, s - s', d - d') = \mathrm{exp.dim} \, M^s_{C,\delta}(E_0, s - s', d - d') $.

\begin{claim}\label{claim_lem_similar_to_M_L}
For any $[L] \in \Pic^d (C)$, the dimension of  \eqref{eq_lem_similar_to_M_L} is 
\[
\dim M^s_{C,\delta}(E_0, s - s', d - d') + \dim M_C(s', d') -g 
\]
if it is non-empty.
\end{claim}

\begin{proof}[Proof of Claim \ref{claim_lem_similar_to_M_L}]
Consider the morphism
\begin{align}\label{eq_lem_similar_to_M_L_2}
  M^s_{C,\delta}(E_0, s - s', d - d') \times M_C(s', d') \xrightarrow{f_1 \times f_2} \Pic^{d-d'}(C) \times \Pic^{d'}(C) \xrightarrow{\otimes} \Pic^d(C)
\end{align}
defined by 
\[
([(E',\alpha')] ,[F]) \mapsto ([\det E'], [\det F]) \mapsto [\det E' \otimes \det F].
\]
The fiber of \eqref{eq_lem_similar_to_M_L_2} over $[L] \in \Pic^d(C)$ is nothing but \eqref{eq_lem_similar_to_M_L}.

By $\frac{d-d'+\delta}{s-s'} =\frac{d+\delta}{s}$ and $\delta <\delta_\sharp = \frac{rd+se}{s-r}$,
we have
\begin{align*}
d-d' = \frac{s-s'}{s}d -\frac{s'}{s} \delta >  \frac{s-s'}{s}d -\frac{s'}{s} \delta_\sharp = \frac{(s-s'-r)d -s'e}{s-r}.
\end{align*}
Hence if $s-s' >r$, $d-d'$ is sufficiently large.
Hence by $\frac{d-d'+\delta}{s-s'} =\frac{d+\delta}{s}$ and \eqref{eq_lemma-upper bound on dim of W,s >r},
we can apply Lemma \ref{lem_smoothness_M_L} and its proof to $M^s_{C,\delta}(E_0, s - s', d - d')_{L'} $ for $[L'] \in \Pic^{d-d'}(C)$.
Hence $f_1$ is smooth.
Since all the fibers of $f_2$ are isomorphic and $\otimes $ is smooth,
this claim holds if $s-s'> r$.

Assume $s-s'=r$.
In this case, $M^s_{C,\delta}(E_0, s - s', d - d') =\Quot_C(E^0, 0,d-d'+e)$ holds.
To see this, by Lemma \ref{lem_suff_large_delta}, it suffices to check
\begin{align*}
\frac{d-d' + \delta}{s-s'} < \frac{e_i+\delta}{i} \quad \text{for any} \quad 1 \leq i \leq r-1,
\end{align*}
which holds by $\frac{d-d'+\delta}{s-s'} =\frac{d+\delta}{s}$ and $\delta >\kappa$.
Then $f_1$ is decomposed as
\begin{align*}
M^s_{C,\delta}(E_0, r, d - d') =\Quot_C(E^0, 0,d-d'+e) \xrightarrow{\pi} \Sym^{d-d'+e} (C) \xrightarrow{\rho} \Pic^{d-d'+e} (C) \simeq  \Pic^{d-d'} (C),
\end{align*}
where $\pi$ is the Quot-to-Chow morphism sending the quotient $[E^0 \to F] $ to the effective divisor determined by the torsion sheaf $F$,
$\rho$ is the Abel--Jacobi map,
and the last isomorphism is defined by $ [L'] \mapsto [L' \otimes \det E_0]$.
By \cite[Corollary 6.6]{GS}, \cite[Theorem 1.2]{Birkar:2024aa},
the dimension of the fiber of $\pi$ is constant.
On the other hand, all the fibers of
\begin{align}
\varphi :   \Sym^{d-d'+e} (C) \times \Pic^{d'} (C) \xrightarrow{\rho \times f_2}  \Pic^{d-d'+e} (C) \times \Pic^{d'} (C)  \xrightarrow{\otimes}  \Pic^{d+e} (C)
\end{align}
are isomorphic since for $[L_1], [L_2] \in  \Pic^{d+e} (C)$, we have an isomorphism
\begin{align*}
\varphi^{-1}([L_1]) \xrightarrow{\sim} \varphi^{-1}([L_2]) , \ ([D], L') \mapsto ([D], L' \otimes L_1^{-1} \otimes L_2). 
\end{align*}
Hence the dimension of \eqref{eq_lem_similar_to_M_L_2} is constant and this claim holds in the case $s-s'=r$ as well.
\end{proof}
By Claim  \ref{claim_lem_similar_to_M_L},
(2) holds by the same argument as in the proof of Lemma \ref{lemma-upper bound on dim of W,s >r} (2).
\end{proof}

\section{SQM}\label{sec_SQM}

In this section,
we show that for $s \leq r$, $M_{C,\delta}(E_0, s, d)$ and $M_{C,\delta}(E_0, s, d)_L$ are SQMs of $\Quot_C(E^0, r - s, d+e)$ and $\Quot_C(E^0, r - s, d+e)_{L \otimes \det E^0}$ respectively
if $d $ is sufficiently large and $M_{C,\delta}(E_0, s, d) =M^s_{C,\delta}(E_0, s, d)$.
We also show a similar result in the case $s >r$.

\subsection{The case $1 \leq s \leq r$}

\begin{proposition}\label{prop_SQM_s_leq_r}
Assume $1 \leq s \leq r$ and $d$ is sufficiently large.
Let $\delta >0$, $ [L] \in \Pic^d (C)$ and assume $M_{C,\delta}(E_0, s, d) =M^s_{C,\delta}(E_0, s, d)$.
Then $M_{C,\delta}(E_0, s, d)$ (resp.\ $M_{C,\delta}(E_0, s, d)_L$) is irreducible, reduced, normal, locally complete intersection, locally factorial, 
and a SQM of $\Quot_C(E^0, r - s, d+e)$ (resp.\ $\Quot_C(E^0, r - s, d+e)_{L \otimes \det E^0}$).

In particular, it holds that
\begin{align*}
\dim M_{C,\delta}(E_0, s, d) &= dr+se - s(r - s)(g - 1),\\
\dim M_{C,\delta}(E_0, s, d)_L &= dr+se - s(r - s)(g - 1)- g
\end{align*}
for such $\delta$.
\end{proposition}

\begin{proof}
If $s=1$, $M_{C,\delta}(E_0, s, d) = \Quot_C(E^0, r - 1, d+e)$ for any $\delta >0$ by Lemma \ref{lem_suff_large_delta}.
Since $\Quot_C(E^0, r - 1, d+e)$ is a $\P^{e+r(d-g+1)-1}$ bundle over $\Pic^{d+e}(C) $ by Remark \ref{rem_small_delta} or \cite[Theorem 3.3]{GS-Picard},
the statement holds.

In the rest of the proof, we assume $s \geq 2$.
For sufficiently large $d$, $\Quot_C(E^0, r - s, d+e)$ is irreducible of dimension $dr+se - s(r - s)(g - 1)$ by \cite[Theorems 6.2, 6.4]{PR}. 
For each $\delta' >0$,
$M_{C,\delta' -\ep}(E_0, s, d)  \setminus W_{\delta'}^- $ is embedded into $M_{C,\delta' +\ep}(E_0, s, d)  $ for sufficiently small $\ep>0$.
Hence by Lemma \ref{lemma-upper bound on dim of W} and descending induction on $\delta$, it holds that
\begin{align}\label{eq_prop_SQM_s_leq_r_dim_M_leq_Q}
\dim M_{C,\delta}(E_0, s, d) \leq \dim \Quot_C(E^0, r - s, d+e) = dr+se - s(r - s)(g - 1),
\end{align}
and there exists at most one  irreducible component of this dimension.
Furthermore, 
\[
\dim  M_{C,\delta''}(E_0, s, d) \geq dr+se - s(r - s)(g - 1)
\]
for $0 < \delta'' < \delta_\flat$ by Lemma \ref{lemma-small delta}, Remark \ref{rem_small_delta}.
Hence by increasing induction on $\delta$,
the equality holds in \eqref{eq_prop_SQM_s_leq_r_dim_M_leq_Q}, and
there exists exactly one  irreducible component of this dimension.

By assumption,
$\delta$-semistability coincides with $\delta$-stability. For any $[(E, \alpha)] \in M_{C,\delta}(E_0, s, d)$, we have
\begin{align*}
\hom (I^\bullet, E) - \ext^1(I^\bullet, E) = dr+se - s(r - s)(g - 1) \leq  \dim_{[(E,\alpha)]} M_{C,\delta}(E_0, s, d)
\end{align*}
by \eqref{eq_lem_dim_tangent_space}.
Hence $M_{C,\delta}(E_0, s, d)$ is a locally complete intersection of the expected dimension at every point by \S \ref{subsec_tangent_space}, and hence is irreducible.

By Lemma \ref{lemma-upper bound on dim of W}, there exists a closed subset $Z \subset M_{C,\delta}(E_0, s, d)$ of codimension at least $4$ such that
\[
M_{C,\delta}(E_0, s, d) \setminus Z \hookrightarrow \Quot_C(E^0, r - s, d+e).
\]
By \cite[\S 4, 5]{GS-Picard}, $\Quot_C(E^0, r - s, d+e)$ is reduced and normal with singularities of codimension at least $4$.
Hence $M_{C,\delta}(E_0, s, d)$ is reduced, normal, and locally factorial as in the proof of \cite[Theorem 6.3]{GS-Picard}.

Take  $0 < \delta' < \delta_\flat$ for $\delta_\flat$ in Lemma \ref{lemma-small delta}. 
Since $\dim W_{\delta''}^- \leq dr+se -s(r-s)(g-1) -2$ for any $\delta''>0$, 
both birational rational maps
\[
M_{C,\delta'}(E_0, s, d) \dashrightarrow M_{C,\delta}(E_0, s, d) \dashrightarrow M_{C,\delta_\sharp +\ep}(E_0, s, d) = \Quot_C(E^0, r - s, d+e)
\]
do not contract divisors. 

\begin{claim}\label{claim_prop_SQM_s_leq_r}
It holds that $\Pic (\Quot_C(E^0, r - s, d+e)) \simeq \Pic(\Pic^0(C)) \oplus \Z^2$.
\end{claim}

\begin{proof}[Proof of Claim \ref{claim_prop_SQM_s_leq_r}]
If $s=r$, there exists a natural morphism $\Quot_C(E^0, 0, d+e)  \to \Sym^{d+e} (C)$ which maps a quotient $[E^0 \twoheadrightarrow F]$ to  the effective divisor on $C$ determined by the torsion sheaf $F$.
\cite[Theorem 11]{GS-nef} states that $\Pic (\Quot_C(E^0, 0, d+e)) \simeq \Pic( \Sym^{d+e} C ) \oplus \Z $.
Since the natural morphism $ \Sym^{d+e} (C) \to \Pic^{d+e}(C)$ is a projective bundle for $d +e \gg 0$, 
we have $\Pic (\Quot_C(E^0, r - s, d+e)) \simeq \Pic( \Pic^{d+e} C ) \oplus \Z^2$.
If $s=r-1$, this holds by  \cite[Theorem 9.1]{GS-Picard}.

Assume $2 \leq s \leq r-2$.
If $(g,r-s) \neq (2,2)$, $ \rho(\Quot_C(E^0, r - s, d+e)) = 2 + \rho(\Pic^d(C))$ holds by \cite[Theorem 1.2]{GS-Picard}.
If $(g,r-s) = (2,2)$, $\rho(\Quot_C(E^0, r - s, d+e)) = 2 + \rho(\Pic^d(C))$ holds by Lemma \ref{lem_appendix_picard_group_s=2},
which will be shown in the appendix.
\end{proof}

Thus it holds that
\[
\rho(M_{C,\delta'}(E_0, s, d)) \leq \rho(M_{C,\delta}(E_0, s, d)) \leq \rho(\Quot_C(E^0, r - s, d+e)) = 2 + \rho(\Pic^d(C)).
\]
On the other hand,
there is a morphism $f : M_{C,\delta'}(E_0, s, d) \to M_C(s, d)$ whose general fiber is $\mathbb{P}^{se+r(d -s(g-1)) - 1}$, and hence
\[
\rho(M_{C,\delta'}(E_0, s, d)) \geq 1 + \rho(M_C(s, d)) = 2 + \rho(\Pic^d(C))
\]
by $s \geq 2$.
Therefore, $\rho(M_{C,\delta}(E_0, s, d))= \rho(\Quot_C(E^0, r - s, d+e))$ holds
and hence $M_{C,\delta}(E_0, s, d)$ is a SQM of $\Quot_C(E^0, r - s, d+e)$.

The proof for $M_{C,\delta}(E_0, s, d)_L$ is similar.
For sufficiently large $d$, $\Quot_C(E^0, r - s, d+e)_{L\otimes \det E^0}$ is irreducible of dimension $dr+se - s(r - s)(g - 1)-g$ by \cite[Theorem 8.7]{GS-Picard}.
As in the case of $M_{C,\delta}(E_0, s, d)$, by using Lemma \ref{lem_similar_to_M_L} (1) instead of Lemma \ref{lemma-upper bound on dim of W},
we see that 
\begin{align*}
\dim M_{C,\delta}(E_0, s, d)_L = dr+se - s(r - s)(g - 1)-g,
\end{align*}
and there exists exactly one  irreducible component of this dimension.
On the other hand,
the dimension of each irreducible component of $M_{C,\delta}(E_0, s, d)_L$ is at least $dr+se - s(r - s)(g - 1)-g$
since $M_{C,\delta}(E_0, s, d)_L$ is a fiber of the morphism $M_{C,\delta}(E_0, s, d) \to \Pic^d(C)$ and
the dimension of each irreducible component of $M_{C,\delta}(E_0, s, d)_L$ is at least 
\[
\dim M_{C,\delta}(E_0, s, d) -\dim  \Pic^d(C) =dr+se - s(r - s)(g - 1)-g.
\]
Hence $ M_{C,\delta}(E_0, s, d)_L $ is irreducible and locally complete intersection since so is  $ M_{C,\delta}(E_0, s, d)$.
The reducedness and normality of $ M_{C,\delta}(E_0, s, d)_L $ follows from a similar argument
by Lemma  \ref{lem_similar_to_M_L} (1) since $\Quot_C(E^0, r - s, d+e)_{L\otimes \det E^0}$ is reduced and normal with singularities of
codimension at least $4$ by the proof of \cite[Proposition 8.1]{GS-Picard}.

\begin{claim}\label{claim_prop_SQM_s_leq_r_L}
It holds that $\Pic (\Quot_C(E^0, r - s, d+e)_{L\otimes \det E^0}) \simeq  \Z^2$.
\end{claim}

\begin{proof}[Proof of Claim \ref{claim_prop_SQM_s_leq_r_L}]
Assume $s=r$.
Let $\pi_1:C^{d+e}\to C^{(d+e)}$ be the quotient map. Let $\Delta \subset C^{(d+e)}$ be the diagonal and set $U\coloneqq \P(H^0(L\otimes \det E^0)) \cap (C^{(d+e)}\setminus \Delta)$.  Note that $U\neq \emptyset$ and it is the set of all reduced divisors of degree $d+e$ which are linearly equivalent to $L\otimes \det E^0$. Moreover, by \cite[Chapter III, Section 1, p.111-112]{ACGH}, $\pi_1^{-1}(U)$ is irreducible. Let $\phi:\Quot_C(E^0,0,d+e)_{L\otimes \det E^0}\to \mb P(H^0(L\otimes \det E^0))$ be the restriction of the Quot-Chow map. Let $\pi_2:\P(E_0)^{d+e}\to C^{d+e}$ be the natural map.
Then, as in \cite[Lemma 10]{GS-nef} we have a commutative diagram
\[
\begin{tikzcd}
    \pi_2^{-1}\pi_1^{-1}(U) \ar[r,"\widetilde{\pi}_1"] \ar[d,"\widetilde{\phi}"] & \phi^{-1}(U) \ar[d,"\phi"] \\
    \pi_1^{-1}(U) \ar[r,"\pi_1"]
   & U
 \end{tikzcd}
\]
Now let $\mc L$ be a line bundle on $\Quot(E^0,0,d+e)_{L\otimes \det E^0}$. Note that since 
$\pi_1^{-1}(U)$ is integral and the maps $\widetilde{\phi}\,,\widetilde{\pi}_1$ are $S_{d+e}$-equivariant, we have 
\[
\widetilde{\pi}_1^*\mc L\cong \widetilde{\phi}^{*}\mc L'\otimes \bigotimes\limits_{i=1}^{d+e} \mc O_i(m),
\]where $\mc L'$ is a line bundle on $\pi_1^{-1}(U)$ and $\mc O_{i}(1)$ is the pullback of the universal line bundle on $\P(E_0)$ via the $i$-th projection $(\P(E_0))^{d+e}\to \P(E_0)$. Now recall from the proof of \cite[Lemma 10]{GS-nef} that there exists a line bundle $\mc O_{\mc Q}(1)$ on $\Quot(E^0,0,d+e)$ whose restriction to fiber $\phi^{-1}(\{c_1,c_2,\ldots, c_{d+e}\})=\prod \P(E_0|_{c_i})$ is $\bigotimes\limits_i \mc O_{\P(E_0|_{c_i})}(1)$. Therefore $\mc L\otimes \mc O_{\mc Q}(-m)$ restricted to fiber of $\phi$ over $U$ is trivial. Since $\phi$ is flat and has integral fibers \cite[Corollary 6.4]{GS}, by \cite[Lemma 2.1.2]{MR-ss} we get that it is trivial over all the fibers of $\phi$.
This implies that $\mc L\otimes \mc O_{\mc Q}(-m)=\phi^*\mc L''$ for a line bundle $L''$ over $\mb P(H^0(L\otimes \det E^0))$. This completes the proof for the case $s=r$.

If $s=r-1$, this holds by  \cite[Theorem 9.1]{GS-Picard}.

Assume $2 \leq s \leq r-2$.
If $(g,r-s) \neq (2,2)$, this holds by \cite[Theorem 1.4]{GS-Picard}. 
If $(g,r-s) = (2,2)$, this holds by Lemma \ref{lem_appendix_picard_group_s=2,L}.
\end{proof}

Since rational maps {$M_{C,\delta'}(E_0, s, d)_L \dashrightarrow  M_{C,\delta}(E_0, s, d)_L  \dashrightarrow \Quot_C(E^0, r - s, d+e)_{L\otimes \det E^0}$} do not contract divisors and 
$\rho(M_{C,\delta'}(E_0, s, d)_L) \geq 1 + \rho(M_C(s, L)) = 2 $ for $0 < \delta' < \delta_\flat$ and $\delta>0$,
we obtain $\rho(M_{C,\delta'}(E_0, s, d)_L)=2$ and hence $M_{C,\delta}(E_0, s, d)_L$ is a SQM of $\Quot_C(E^0, r - s, d+e)_{L\otimes \det E^0}$.
\end{proof}

\begin{corollary}\label{cor_pic_s_leq_r}
    Assume $1\leq s\leq r$ and $d$ is sufficiently large. Let $\delta>0$ and assume $M_{C,\delta}(E_0,s,d)=M_{C,\delta}^s(E_0,s,d)$. Let $[L]\in \Pic^d(C)$.
    \begin{enumerate}
        \item If $s=1$, then $\Pic(M_{C,\delta}(E_0,s,d))=\Pic(\Pic^d(C)) \oplus \Z$ and
        \[
        \Pic(M_{C,\delta}(E_0,s,d)/\Pic^d(C)) \simeq \Pic(M_{C,\delta}(E_0, s, d)_L)= \Z .
        \]
 \item If $2\leq s\leq r$, then
$\Pic(M_{C,\delta}(E_0,s,d))=\Pic(\Pic^d(C))\oplus \Z^2$ and 
\[
\Pic(M_{C,\delta}(E_0,s,d)/\Pic^d(C)) \simeq  \Pic(M_{C,\delta}(E_0, s, d)_L)=\Z^2.
\]
    \end{enumerate}
\end{corollary}

\begin{proof}
(1) holds since $M_{C,\delta}(E_0,s,d) \to \Pic^d (C)$ is a projective bundle.\\
(2) holds by Proposition \ref{prop_SQM_s_leq_r} and Claims \ref{claim_prop_SQM_s_leq_r}, \ref{claim_prop_SQM_s_leq_r_L}.
\end{proof}

\subsection{The case $s >r$}

If $d $ is sufficiently large,
$\delta_\sharp =\frac{dr+se}{s-r}$ and we have an isomorphism 
\[
M_{C,\delta_\sharp}(E_0,s,d) \xrightarrow{\sim}M_C(s-r,d+e) , \  [(E,\alpha)] \mapsto [E/\im \alpha]
\]
by Lemma \ref{lemma-upper bound on dim of W,s >r} (3).  

\begin{lemma}\label{lemma-M_delta^dagger_s>r}
    Let $d\gg 0 $ and $0 < \varepsilon\ll 1$. 
    Then $M_{C,\delta_\sharp-\varepsilon}(E_0,s,d)$ is irreducible, smooth of dimension $rd+se+s(s-r)(g-1)$. Furthermore, the morphism 
\begin{align*}
\pi : M_{C,\delta_\sharp-\varepsilon}(E_0,s,d)\to M_{C,\delta_\sharp}(E_0,s,d)=M_{C}(s-r,d+e), \ [(E,\alpha)] \mapsto [E/\im\alpha]
\end{align*}    
is a projective bundle over  $M^s_{C}(s-r,d+e)$, and induces an isomorphism 
$\Pic (M_{C,\delta_\sharp-\varepsilon}(E_0,s,d))\simeq \Pic (M_{C}(s-r,d+e) ) \oplus \Z$.

Similarly, $M_{C,\delta_\sharp-\varepsilon}(E_0,s,d)_L$ is irreducible, smooth of dimension $rd+se+s(s-r)(g-1)-g$.
Furthermore, the restriction
\begin{align*}
\pi_L: M_{C,\delta_\sharp-\varepsilon}(E_0,s,d)_L\to M_{C,\delta_\sharp}(E_0,s,d)_L=M_{C}(s-r,L\otimes \det E^0)
\end{align*}    
of $\pi$ induces an isomorphism 
$\Pic (M_{C,\delta_\sharp-\varepsilon}(E_0,s,d))_L\simeq \Pic (M_{C}(s-r,L\otimes \det E^0)) \oplus \Z$.
\end{lemma}

\begin{proof}
By Lemma \ref{lemma-fiber of contraction},
$\pi : \pi^{-1}(M^s(s-r,d+e)) \to M^s(s-r,d+e)$ is a projective bundle.
Furthermore,
$W_{\delta_\sharp}^-$ defined by \eqref{eq_W^-} coincides with $M_{C,\delta_\sharp-\varepsilon}(E_0,s,d) $ since $M^s_{C,\delta_\sharp}(E_0,s,d) =\emptyset$
by $\mu_{\delta_\sharp}(E_0,\id) = \frac{d+\delta_\sharp}{s}$.
By the proof of Lemma \ref{lem_closure_of_stable_F},
$W_{\delta_\sharp}^-$ coincides with the closure of $ \pi^{-1}(M^s(s-r,d+e)) $.
Hence $M_{C,\delta_\sharp-\varepsilon}(E_0,s,d) =W_{\delta_\sharp}^-$ is irreducible. 
By Lemma \ref{lemma-upper bound on dim of W,s >r} (2),
$M_{C,\delta_\sharp-\varepsilon}(E_0,s,d) =M^s_{C,\delta_\sharp-\varepsilon}(E_0,s,d)$ is smooth of dimension $rd+se+s(s-r)(g-1)$.

To determine $\Pic (M_{C,\delta_\sharp-\varepsilon}(E_0,s,d))$, 
we introduce some notation.
    By $d\gg 0$, we can assume that every semistable bundle $F$ with degree $d+e$ and rank $s-r$ is globally generated and $H^1(F)=0$. Let $N=(d+e)-(s-r)(g-1)$.
    Let $R^{ss}\subset \Quot_C(\mc O_C^N,s-r,e+d)$ be the open subset parametrizing quotients $[\mc O^N_C\to F]$ such that $F$ is semistable and $H^0(\mc O^N_C)\xrightarrow{\sim} H^0(F)$. 
    Then $R^{ss}$ is smooth and equidimensional.
    Now note that for any $F$ semistable with degree $d+e$ and rank $s-r$,
     we have ${\rm Hom}(F,E_0)=0$ since $\mu(F) = \frac{d+e}{ s-r} > \mu_{\max}(E_0)$ by $d \gg 0$.
    Hence we have a projective bundle $\P \to R^{ss}$ whose fiber over $[\mc O^N_C\to F]$ is given by $\P({\rm Ext}^1(F,E_0))$. 
    Let $\mb U\subset \P$ be the open subset parametrizing non-split extensions $0\to E_0 \xrightarrow{\alpha} E\to F\to 0$ where $(E, \alpha)$ is $(\delta_\sharp-\varepsilon)$-stable.
      Now consider the natural morphism $\mb U\to M_{C,\delta_\sharp-\varepsilon}(E_0,s,d)$.

\begin{claim}\label{claim_surjective_U_to_P}
The morphism $\mb U\to M_{C,\delta_\sharp-\varepsilon}(E_0,s,d)$ is surjective.
\end{claim}

\begin{proof}[Proof of Claim \ref{claim_surjective_U_to_P}]
For $[(E,\alpha)] \in M_{C,\delta_\sharp-\varepsilon}(E_0,s,d)$,
$\alpha$ is injective by the proof of Lemma \ref{lemma-upper bound on dim of W,s >r} (3) since $[(E,\alpha)]$ is $\delta_\sharp$-semistable.
Furthermore,  $E/\im \alpha$ is semistable since $\mu_{\delta_\sharp}(E,\alpha) =\mu(E/\im \alpha)$. 
Hence $0 \to E_0 \xrightarrow{\alpha} E \to E/\im \alpha \to 0$ corresponds to a point in $\mb U$.
\end{proof}

Let $R^s\subset R^{ss}$ be the open  subset parametrizing quotients $[\mc O^N_C\to F]$
with $F$ stable. Let $\mb U^s$ be the inverse image of $R^s$ in $\P$. Note that $\mb U^s \subset \mb U$ by Claim \ref{claim_lemma-fiber of contraction}.

If $(g,s-r) \neq (2,2)$, we have that $\codim_{R^{ss}}(R^{ss}\setminus R^s)\geq 2$ by \cite[Proposition 1.2]{Bhosle}.
Hence $\codim_{\P}(\P\setminus \mb U)\geq \codim_{\P}(\P\setminus \mb U^s) \geq 2$ and 
\[
\codim_{M_{C,\delta_\sharp-\varepsilon}(E_0,s,d)}(M_{C,\delta_\sharp-\varepsilon}(E_0,s,d)\setminus \pi^{-1}(M^s(s-r,d+e)))\geq 2
\]
by applying \cite[Lemma 2.2]{GS-Picard} to the restriction of each irreducible component of $\mb U$.
Since $\pi : \pi^{-1}(M^s(s-r,d+e)) \to M^s(s-r,d+e)$ is a projective bundle,
we have 
\begin{align*}
\Pic (M_{C,\delta_\sharp-\varepsilon}(E_0,s,d)) &= \Pic (\pi^{-1}(M^s(s-r,d+e))) \\
&\simeq \Pic (M^s_{C}(s-r,d+e)) \oplus \Z = \Pic (M_{C}(s-r,d+e)) \oplus \Z.
\end{align*}
If $(g,s-r) = (2,2)$, we will show $\Pic (M_{C,\delta_\sharp-\varepsilon}(E_0,s,d)) \simeq  \Pic (M_{C}(s-r,d+e)) \oplus \Z$ in Lemma \ref{cor_appB_Pic} in Appendix \ref{sec_appendix_s-r=2}.
Thus the statement of this lemma for $M_{C,\delta_\sharp-\varepsilon}(E_0,s,d)$ is proved.

Since $M_{C,\delta_\sharp-\varepsilon}(E_0,s,d), M_{C}(s-r,d+e)$ are normal and the general fiber of $\pi$ is a projective space, which is connected,
all fibers of $\pi$ are connected.
Hence $M_{C,\delta_\sharp-\varepsilon}(E_0,s,d)_L=\pi^{-1}( M_{C}(s-r,L\otimes \det E^0)) $ is connected.
Thus $M_{C,\delta_\sharp-\varepsilon}(E_0,s,d)_L$ is irreducible, smooth of dimension $rd+se+s(s-r)(g-1)-g$ by Lemma \ref{lem_smoothness_M_L}.

Let $R^{ss}_L, R^s_L$ be the fibers of the natural morphisms from $R^{ss}, R^s$ to $\Pic^{d+e}(C)$ over $[L \otimes \det E^0]$.
If $(g,s-r) \neq (2,2)$,  we have that $\codim_{R^{ss}_L}(R_L^{ss}\setminus R_L^s)\geq 2$ by \cite[Corollary 1.3]{Bhosle}.
Then the rest of the proof is similar.
If $(g,s-r) = (2,2)$, we will show $\Pic (M_{C,\delta_\sharp-\varepsilon}(E_0,s,d)_L) \simeq  \Pic (M_{C}(s-r,L\otimes \det E^0)) \oplus \Z$ in  Lemma \ref{cor_appB_Pic} .
\end{proof}

\begin{lemma}\label{lem_s=r+1, last model}
Assume $s=r+1$, $d \gg 0$, and $ 0<  \ep \ll 1 $.
\begin{enumerate}
\setlength{\itemsep}{0mm}
\item $M_{C,\delta_\sharp-\varepsilon}(E_0, r+1, d)$ is a projective bundle over $ \Pic^{d}(C)$.
\item For $\kappa_\sharp<\delta<\delta_\sharp$, it holds that $M_{C,\delta}(E_0, r+1, d)=M_{C,\delta_\sharp-\varepsilon}(E_0, r+1, d)$.
\item There exists a divisorial contraction $ M_{C,\kappa_\sharp-\varepsilon}(E_0, r+1, d) \to M_{C,\delta_\sharp-\varepsilon}(E_0, r+1, d)$.
\item $\Pic  (M_{C,\kappa_\sharp-\varepsilon}(E_0, r+1, d))  \simeq \Pic (M_{C,\delta_\sharp-\varepsilon}(E_0, r+1, d)) \oplus \Z \simeq \Pic ( \Pic^{d}(C)) \oplus \Z^2$.
\end{enumerate}
Similar statements hold for  $M_{C,\delta}(E_0, r+1, d)_L$,
that is, 
\begin{enumerate}[label=(\arabic*)']
\setlength{\itemsep}{0mm}
\item $M_{C,\delta}(E_0, r+1, d)_L$ is a projective space.
\item For $\kappa_\sharp<\delta<\delta_\sharp$, it holds that $M_{C,\delta}(E_0, r+1, d)_L=M_{C,\delta_\sharp-\varepsilon}(E_0, r+1, d)_L$.
\item There exists a divisorial contraction $ M_{C,\kappa_\sharp-\varepsilon}(E_0, r+1, d)_L \to M_{C,\delta_\sharp-\varepsilon}(E_0, r+1, d)_L$.
\item  $\Pic  (M_{C,\kappa_\sharp-\varepsilon}(E_0, r+1, d))_L \simeq \Pic (M_{C,\delta_\sharp-\varepsilon}(E_0, r+1, d)_L) \oplus \Z \simeq \Z^2$.
\end{enumerate}
\end{lemma}

\begin{proof}
(1) By Lemma \ref{lemma-fiber of contraction},
\begin{align*}
\pi:  M_{C,\delta_\sharp-\varepsilon}(E_0, r+1, d) \to  M_{C,\delta_\sharp}(E_0, r+1, d) =M_C(1,d+e) =\Pic^{d+e}(C) \simeq \Pic^{d}(C) 
\end{align*}
 is a projective bundle.\\
(2) Let $\kappa_\sharp<\delta<\delta_\sharp$. 
Then we have $M_{C,\delta}(E_0, r+1, d) \setminus M^s_{C,\delta}(E_0, r+1, d) =\emptyset$.
Indeed, if it is non-empty, then there exists $1 \leq s' < s$ and $d'$ such that 
$\frac{d+\delta}{s}=\frac{d'}{s'}$ and $M_{C,\delta}^s(E_0,s-s',d-d')\neq \emptyset$ by Lemma \ref{lem_disjoint_union}. 
By the proof of Lemma \ref{lemma-upper bound on dim of W,s >r} (2) and $\delta>\kappa_\sharp > \kappa$,
we have $s-s' \geq r =s-1$, that is, $s-s'=r$ and $s'=1$. 
This means that $\frac{d+\delta}{s}=\frac{d-d'+\delta}{s-s'}=\frac{d-d'+\delta}{r}$ and hence
\begin{align*}
    d-d' =\frac{r}{r+1}d-\frac{\delta}{r+1}.
\end{align*}
By $\kappa_\sharp =  rd+(r+1)(e-1 ) < \delta < \delta_\sharp = rd+(r+1)e$,
we have $-e < d-d' < 1-e$, which is a contradiction.

Hence 
$M_{C,\delta}(E_0,s,d) = M_{C,\delta_\sharp-\varepsilon}(E_0,s,d)$ holds for any $\kappa_\sharp < \delta < \delta_\sharp$.\\
(3) Applying the above argument to $\delta=\kappa_\sharp $,
we obtain $s-s'=r$, $ d-d'=1-e$ and hence
\begin{align*}
M_{C,\kappa_\sharp}(E_0, r+1, d) \setminus M^s_{C,\kappa_\sharp}(E_0, r+1, d) 
&=   M^s_{C,\kappa_\sharp}(E_0, r, 1-e) \times M_C(1,d+e-1) \\
&=   \Quot_C(E^0,0,1)   \times \Pic^{d+e-1}(C) \\
&=   \P_C(E_0)  \times \Pic^{d+e-1}(C) ,
\end{align*}
where the second equality holds since $\kappa_\sharp =rd +(r+1)(e-1) $  is sufficiently large.

Let 
\[
W^-_{\kappa_\sharp} = \left\{ [(E, \alpha)] \in M_{C,\kappa_\sharp- \varepsilon}(E_0, r+1, d) \,\middle|\, (E, \alpha) \text{ is not } \kappa_\sharp\text{-stable} \right\}
\]
as in \eqref{eq_lemma-upper bound on dim of W,s >r}.
By Lemma \ref{lemma-fiber of contraction},
\begin{align*}
\pi_{\kappa_\sharp}^-:  M_{C,\kappa_\sharp-\varepsilon}(E_0, r+1, d) \to  M_{C,\kappa_\sharp}(E_0, r+1, d) 
\end{align*}
induces a projective bundle
\begin{align*}
\pi_{\kappa_\sharp}^-|_{W^-_{\kappa_\sharp} } : W^-_{\kappa_\sharp}  \to  \P_C(E_0)  \times \Pic^{d+e-1}(C).
\end{align*}
Since $\dim W^-_{\kappa_\sharp} =\dim M_{C,\kappa_\sharp-\varepsilon}(E_0, r+1, d) -1$ by Lemma \ref{lemma-upper bound on dim of W,s >r},
$ W^-_{\kappa_\sharp}$  is a prime divisor on $M_{C,\kappa_\sharp-\varepsilon}(E_0, r+1, d) $.

On the other hand,
we have a morphism 
\begin{align*}
\pi_{\kappa_\sharp}^+ :  M_{C,\delta_\sharp-\varepsilon}(E_0, r+1, d) =M_{C,\kappa_\sharp+\varepsilon}(E_0, r+1, d) \to  M_{C,\kappa_\sharp}(E_0, r+1, d)
\end{align*}
by (2), 
which is an isomorphism over $M^s_{C,\kappa_\sharp}(E_0, r+1, d)$.
By $\dim  M_{C,\kappa_\sharp}(E_0, r+1, d) >\dim \Pic^d(C)$,  we have $\rho( M_{C,\kappa_\sharp}(E_0, r+1, d)) \geq  \rho(\Pic^{d}(C)  ) +1 $.
On the other hand, (1) implies that
$\rho (M_{C,\delta_\sharp-\varepsilon}(E_0, r+1, d)) =\rho ( \Pic^{d}(C)  )+1 $
and hence $\pi_{\kappa_\sharp}^+$ is a finite morphism.
Thus
$\pi_{\kappa_\sharp}^+$ is the normalization of { $M_{C,\kappa_\sharp}(E_0, r+1, d)$ (with the reduced structure)}. 
Since $ M_{C,\kappa_\sharp-\varepsilon}(E_0, s, d)$ is smooth by Lemma \ref{lemma-upper bound on dim of W,s >r} (2),
the morphism $\pi_{\kappa_\sharp}^-$ factors as
\begin{align*}
M_{C,\kappa_\sharp-\varepsilon}(E_0, r+1, d) \xrightarrow{\mu} M_{C,\delta_\sharp-\varepsilon}(E_0, r+1, d) 
 \xrightarrow{\pi_{\kappa_\sharp}^+}  M_{C,\kappa_\sharp}(E_0, r+1, d).
\end{align*}
Since $\pi_{\kappa_\sharp}^-$ is an isomorphism on $M_{C,\kappa_\sharp-\varepsilon}(E_0, r+1, d) \setminus W^-_{\kappa_\sharp}$
and contracts the prime divisor $W^-_{\kappa_\sharp} $,
so is $\mu$.
Hence $\mu$ is a divisorial contraction.\\[1mm]
(4) follows from (1), (3).

(1)', (2)' follow from (1), (2).
We can show (3)' by the same argument if we check that 
$W^-_{\kappa_\sharp,L} =W^-_{\kappa_\sharp}   \cap  M_{C,\kappa_\sharp-\varepsilon}(E_0, r+1, d)_L$ is a prime divisor of $M_{C,\kappa_\sharp-\varepsilon}(E_0, r+1, d)_L$
and $ M_{C,\kappa_\sharp-\varepsilon}(E_0, s, d)_L$ is smooth.
The latter follows from Lemma \ref{lem_smoothness_M_L}.
To show the former,
$W^-_{\kappa_\sharp,L} $ is the fiber over $[L]$ of the morphism
\begin{align*}
  W^-_{\kappa_\sharp}  \xrightarrow{\pi_{\kappa_\sharp}^-|_{W^-_{\kappa_\sharp} }} \P_C(E_0)  \times \Pic^{d+e-1}(C) \xrightarrow{\varpi \times \id} C  \times \Pic^{d+e-1}(C) \xrightarrow{\tau} \Pic^{d+e} (C) \simeq \Pic^d(C),
\end{align*}
where $\varpi: \P_C(E_0) \to C$ is the natural morphism and $\tau$ is defined by $(p, [N]) \mapsto [\calo_C(p) \otimes N]$,
and the last isomorphism is defined by $\otimes \det E_0$.
Since $\pi_{\kappa_\sharp}^-|_{W^-_{\kappa_\sharp} }, \varpi \times \id$ are projective bundles and all fibers of $\tau$ are isomorphic to $C$,
the fiber $W^-_{\kappa_\sharp,L} $ is irreducible.
Hence $W^-_{\kappa_\sharp,L} $ is a prime divisor by Lemma \ref{lem_similar_to_M_L} (2).
Thus (3)' holds.

(4)' follows from (1)', (3)'.
\end{proof}

\begin{proposition}\label{prop_SQM_s>r}
 Assume $s  >r$ and $d$ is sufficiently large.
Let $0 < \delta < \delta_\sharp$ and assume $M_{C,\delta}(E_0, s, d) =M^s_{C,\delta}(E_0, s, d)$.
Then
\begin{enumerate}
\item
 $M_{C,\delta}(E_0, s, d)$ is irreducible, reduced, normal, locally complete intersection, locally factorial,
 of dimension $ dr+se - s(r - s)(g - 1)$.
\item If $\kappa<\delta<\delta_\sharp$, then $M_{C,\delta}(E_0, s, d)$ is smooth.
\item If $s-r\geq 2$, then $M_{C,\delta}(E_0, s, d)$
is a SQM of $M_{C,\delta_\sharp-\varepsilon}(E_0,s,d)$ and 
\[
\Pic (M_{C,\delta}(E_0, s, d))=\Pic (\Pic^d (C)) \oplus \Z^2.
\]
\item If $s-r=1$ and $0<\delta<\kappa_\sharp$, then $M_{C,\delta}(E_0, s, d)$ is a SQM of $M_{C,\kappa_\sharp-\varepsilon}(E_0, s, d)$ and 
\[
\Pic (M_{C,\delta}(E_0, s, d))=\Pic (\Pic^d (C))  \oplus \Z^2.
\]
\end{enumerate}
Similar statements hold for $M_{C,\delta}(E_0, s, d)_L$ for $[L] \in \Pic^d(C)$.
\end{proposition}

\begin{proof}
(2)  is proved in Lemma \ref{lemma-upper bound on dim of W,s >r}.\\[1mm]
(1) 
By Lemma \ref{lemma-M_delta^dagger_s>r}, $M_{C,\delta_\sharp-\varepsilon}(E_0,s,d)$ is irreducible of dimension $dr+se - s(r-s)(g - 1)$. 
Hence $M_{C,\delta}(E_0, s, d)$ is irreducible and locally complete intersection of dimension $ dr+se - s(r - s)(g - 1)$
by the argument in the proof of Proposition \ref{prop_SQM_s_leq_r}, by using $M_{C,\delta_\sharp-\varepsilon}(E_0,s,d)$ instead of $ \Quot_C(E^0,r-s,d+e)$.

By Lemma \ref{lemma-upper bound on dim of W,s >r} (1), if $\delta \leq \kappa$ then there exists a closed subset $Z \subset M_{C,\delta}(E_0, s, d)$ of codimension at least $4$ such that
\[
M_{C,\delta}(E_0, s, d) \setminus Z \hookrightarrow M_{C,\kappa+\varepsilon}(E_0,s,d).
\]
Since
$M_{C,\kappa+\varepsilon}(E_0,s,d)$ is smooth by (2), 
$M_{C,\delta}(E_0, s, d)$ is reduced, normal, and locally factorial. This completes the proof of (1).
\\[1mm]
(3)
Let $s-r\geq 2$. 
Take  $0 < \delta' < \delta_\flat$ for $\delta_\flat$ in Lemma \ref{lemma-small delta}. Then both birational rational maps
\[
M_{C,\delta'}(E_0, s, d) \dashrightarrow M_{C,\delta}(E_0, s, d) \dashrightarrow M_{C,\delta_\sharp-\ep}(E_0, s, d)
\]
do not contract divisors by Lemma \ref{lemma-upper bound on dim of W,s >r} (1), (2). Thus,
\[
\rho(M_{C,\delta'}(E_0, s, d)) \leq \rho(M_{C,\delta}(E_0, s, d)) \leq \rho(M_{C,\delta_\sharp-\varepsilon}(E_0,s,d)) = 2 + \rho(\Pic^d(C)),
\]
where the last equality holds by Lemma \ref{lemma-M_delta^dagger_s>r} and $ s-r \geq 2$.
On the other hand,
there is a morphism $f : M_{C,\delta'}(E_0, s, d) \to M_C(s, d)$ whose general fiber is $\mathbb{P}^{se+r(d - s(g-1)) - 1}$, so
\[
\rho(M_{C,\delta'}(E_0, s, d)) \geq 1 + \rho(M_C(s, d)) = 2 + \rho(\Pic^d(C)).
\]
Therefore, $\rho(M_{C,\delta}(E_0, s, d))= \rho(M_{C,\delta_\sharp-\varepsilon}(E_0,s,d))$ holds
and hence $M_{C,\delta}(E_0, s, d)$ is a SQM of $M_{C,\delta_\sharp-\varepsilon}(E_0,s,d)$. This completes the proof of (3).
\\[1mm]
(4) Let $s-r=1$ and $0<\delta<\kappa_\sharp$. 
Take  $0 < \delta' < \delta_\flat$ for $\delta_\flat$ in Lemma \ref{lemma-small delta}. Then both birational rational maps
\[
M_{C,\delta'}(E_0, s, d) \dashrightarrow M_{C,\delta}(E_0, s, d) \dashrightarrow M_{C,\kappa_\sharp-\ep}(E_0, s, d)
\]
do not contract divisors by Lemma \ref{lemma-upper bound on dim of W,s >r} (1), (2). Thus,
\[
\rho(M_{C,\delta'}(E_0, s, d)) \leq \rho(M_{C,\delta}(E_0, s, d)) \leq \rho(M_{C,\kappa_\sharp-\varepsilon}(E_0,s,d)) = 2 + \rho(\Pic^d(C)),
\]
where the last equality holds by Lemma \ref{lem_s=r+1, last model} (4).
The rest is the same as (3).

The proof for $M_{C,\delta}(E_0, s, d)_L$ is similar, by using Lemmas \ref{lem_smoothness_M_L}, \ref{lem_similar_to_M_L} instead of Lemma \ref{lemma-upper bound on dim of W,s >r} .
\end{proof}

\section{Description of $\delta$ with $M_{C,\delta}(E_0, s, d)  \neq M^s_{C,\delta}(E_0, s, d)$}\label{sec_Gamma}

In this section, 
we determine when $M_{C,\delta}(E_0, s, d) \neq M^s_{C,\delta}(E_0, s, d)$ holds for sufficiently large $d$.
By Lemma \ref{lem_disjoint_union},
$M_{C,\delta}(E_0, s, d) \setminus M^s_{C,\delta}(E_0, s, d)$ is the union of 
\[
M_{C,\delta}^s(E_0,s-s',d-d') \times M_C(s',d')
\]
 with 
$1 \leq s' \leq s-1,d' \in \Z$ and $\frac{d+\delta}{s} =\frac{d'}{s'}$.
Since $M_C(s',d')\neq \emptyset$,
it suffices to consider whether $M_{C,\delta}^s(E_0,s-s',d-d') $ is empty or not for such $s',d'$.

\subsection{The case $s \leq  r$}

In this subsection, we assume $2 \leq s \leq r$ and 
$d$ is sufficiently large.
In this case,
\begin{align}\label{eq_delta^dagger_suff_large_d, s_leq_r}
\delta_\sharp = \max\limits_{1 \leq i \leq s-1} \frac{id-s e_i }{s-i} = (s-1)d-s e_{s-1}.
\end{align}

\begin{proposition}\label{prop_gamma_s<r+1}
Assume $2 \leq s \leq r$ and $d$ is sufficiently large.
Let 
\[
\Gamma=\{  \delta_1 < \delta_2 < \cdots < \delta_N \} \subset  (0, \delta_\sharp]
\] 
be the set of $\delta \in (0,\delta_\sharp]$ such that
\[
\frac{d + \delta}{s} = \frac{d'}{s'} \quad  \text{for some integers} \quad  1 \leq s' \leq s - 1 \quad \text{and} \quad  d' \leq d- e_{s-s'}.
\]
Then
\begin{enumerate}
\item For $\delta>0$,  $M_{C,\delta}(E_0, s, d) \neq M^s_{C,\delta}(E_0, s, d)$ holds if and only if $\delta \in \Gamma$.
\item $\delta_1=\delta_\flat$.
\item $\delta_N=\delta_\sharp =(s-1)d-se_{s-1}$.
\end{enumerate}
\end{proposition}

\begin{proof}
(1) It suffices to show the following claim.

\begin{claim}\label{claim_gamma_s<r+1}
For integers $1 \leq s' \leq s-1,d' \in \Z$ with $\frac{d+\delta}{s} =\frac{d'}{s'}$, $M^s_{C,\delta}(E_0,s-s',d-d') \neq \emptyset $ holds if and only if $d' \leq d-e_{s-s'}$.
\end{claim}

\begin{proof}[Proof of Claim \ref{claim_gamma_s<r+1}]
If $s-s'=1$,
$ M^s_{C,\delta}(E_0,s-s',d-d')  \neq \emptyset $ if and only if 
there exists a non-zero  morphism $\alpha : E_0 \to E$ to a locally free $E$ with $\rank E=1, \deg E=d -d'$,
which is equivalent to $d-d' \geq e_1$.

Assume $s-s'\geq 2$.
The $\delta_\sharp$ of $ M_{C,\delta}(E_0,s-s',d-d')$ is
\begin{align}\label{eq_claim_gamma_s<r+1}
\max_{1 \leq i \leq s-s'-1}\frac{i(d-d')-(s-s')e_i}{s-s'-i}.
\end{align}
Since $\delta = \frac{s}{s'}d'-d$, we see that
$\delta$ is greater than \eqref{eq_claim_gamma_s<r+1} if and only if 
\[
d-d' < \frac{(s-s'-i)d + s'e_i}{s-i}
\]
for any $1 \leq i \leq s-s'-1$,
which is equivalent to
\begin{align}\label{eq_claim_gamma_s<r+1_eq2}
d-d' < \frac{d + s'e_{s-s'-1}}{s'+1} ,
\end{align}
since $d$ is sufficiently large.

If the condition \eqref{eq_claim_gamma_s<r+1_eq2} holds, by Lemma \ref{lem_suff_large_delta},
$ M^s_{C,\delta}(E_0,s-s',d-d')  \neq \emptyset $ if and only if 
there exists a generically surjective morphism $\alpha : E_0 \to E$ to a locally free $E$ with $\rank E=s-s', \deg E=d -d'$,
which is equivalent to $d-d' \geq e_{s-s'}$.

Assume that  the condition \eqref{eq_claim_gamma_s<r+1_eq2} does not hold.
Since we fix $s$, there are finitely many choices of $s'$.
Hence $d-d'$ is sufficiently large since so is $d $,  and $ M^s_{C,\delta}(E_0,s-s',d-d') \neq \emptyset$ by Proposition \ref{prop_SQM_s_leq_r}.
In this case,  $d-d'$ is sufficiently large and hence  $ d' \leq d-e_{s-s'}$ holds.
\end{proof}

\noindent
(2) By the definition of $\delta_\flat$, we have $\delta_\flat \leq \delta_1$.
Let $\frac{d+\delta_\flat}{s} =\frac{d'}{s'}$ with $1 \leq s' \leq s-1$ and $d' \in \Z$. 
By the definition of $\delta_\flat$,
we have $ \frac{d'-1}{s'} \leq \frac{d}s$ 
and hence $d' \leq \frac{s'}{s}d + 1 \leq d -e_{s-s'}$,
where the last inequality holds by $d \gg 0$ and $s' <s$.
Thus $\delta_\flat \in \Gamma$ and hence $\delta_\flat \geq \delta_1$.\\
(3) In this case, $\delta_\sharp = (s-1)d - se_{s-1} $ by \eqref{eq_delta^dagger_suff_large_d, s_leq_r}.
Then we have 
\begin{align*}
 \frac{d+\delta_\sharp}{s} = d -e_{s-1} = \frac{d'}{s'}
\end{align*}
for $s'=1, d'=d-e_{s-1} $.
Thus $\delta_\sharp \in \Gamma$.
\end{proof}

\subsection{The case $s >r$}
In this subsection, we assume $s >r$ and 
$d$ is sufficiently large.
In this case,
\begin{align*}
\delta_\sharp = \max\limits_{1 \leq i \leq r} \frac{id-s e_i }{s-i} = \frac{rd+se}{s-r}.
\end{align*}

\begin{proposition}\label{prop_gamma_s >r}
Assume $s > r$ and $d$ is sufficiently large.
Let 
\[
\Gamma=\{  \delta_1 < \delta_2 < \cdots < \delta_N \} \subset  (0, \delta_\sharp]
\] 
be the set of $\delta \in (0,\delta_\sharp]$ such that
\[
 \frac{d + \delta}{s} = \frac{d'}{s'}  \quad \text{for some integers} \quad  1\leq s'<s-r \quad  \text{ or } \quad  s-r\leq s'\leq s-1,\    d' \leq d- e_{s-s'  } .
\]
Then
\begin{enumerate}
\item For $\delta>0$,  $M_{C,\delta}(E_0, s, d) \neq M^s_{C,\delta}(E_0, s, d)$ holds if and only if $\delta \in \Gamma$.
\item $\delta_1=\delta_\flat$.
\item $\delta_N=\delta_\sharp =\frac{rd+se}{s-r}$.
\end{enumerate}
\end{proposition}

\begin{proof}
(1) It suffices to show the following claim.

\begin{claim}\label{claim_nonemptyness_s>r}
Let $\delta \in (0, \delta_\sharp] $  and  
$ 1 \leq s' \leq s - 1,\ d' \in \mathbb{Z}$
such that $  \frac{d + \delta}{s} = \frac{d'}{s'} $.
Then    $M^s_{C,\delta}(E_0,s-s',d-d') \neq \emptyset $ if and only if 
either
\begin{itemize}
\item     $s-s'>r$, or
\item $s-s'\leq r$ and $d'\leq d-e_{s-s'}$.
\end{itemize}
\end{claim}

\begin{proof}[Proof of Claim \ref{claim_nonemptyness_s>r}]  
Assume $s-s' >r$.
Then the $\delta_\sharp$ of $ M_{C,\delta}(E_0,s-s',d-d')$ is
\begin{align}\label{eq_claim_gamma_s>r}
\max_{1 \leq i \leq r} \frac{i(d-d')-(s-s')e_i}{s-s'-i}.
\end{align}
Hence we see that $\delta$ is greater than \eqref{eq_claim_gamma_s>r} if and only if
\begin{align*}
d-d' < \frac{(s-s'-r)d -s'e}{s-r}
\end{align*}
as in the proof of Claim \ref{claim_gamma_s<r+1}.

On the other hand,
\begin{align}\label{eq_claim_gamma_s>r_eq2}
\begin{aligned}
d-d'  = d -\frac{s'}{s}(d+\delta)  &\geq d -\frac{s'}{s}(d+\delta_\sharp) \\
&= \frac{s-s'}{s} d -\frac{s'}{s}\cdot \frac{rd+se}{s-r} =\frac{(s-s'-r)d -s'e}{s-r} .
\end{aligned}
\end{align}
Hence $\delta$ is less than or equal to \eqref{eq_claim_gamma_s>r}.
Since $d-d' $ is sufficiently large by \eqref{eq_claim_gamma_s>r_eq2} and $s-s'>r$,
we have $M_{C,\delta}(E_0,s-s',d-d')\neq \emptyset$ by Proposition \ref{prop_SQM_s>r}.

If $s-s'\leq r$, 
we see that $M_{C,\delta}(E_0,s-s',d-d') \neq \emptyset $ if and only if $d' \leq d-e_{s-s'}$ by the same proof as Claim \ref{claim_gamma_s<r+1}.
\end{proof}

\noindent
(2) By definition, $\delta_\flat \leq \delta_1$.
Let $\frac{d+\delta_\flat}{s} =\frac{d'}{s'}$ with $1 \leq s' \leq s-1$ and $d' \in \Z$. If $s'<s-r$, then $\delta_\flat \in \Gamma$ and hence $\delta_\flat \geq \delta_1$. If $s'\geq s-r$,
by the definition of $\delta_\flat$,
we have $ \frac{d'-1}{s'} \leq \frac{d}s$ 
and hence $d' \leq \frac{s'}{s}d + 1 \leq d - e_{s-s'}$.
Thus $\delta_\flat \in \Gamma$ and hence $\delta_\flat \geq \delta_1$.\\
(3)
We have $ \frac{d+\delta_\sharp}{s} = \frac{d+e}{s-r}$.
Therefore by taking $s'=s-r, d'=d-e_{r}=d+e$, we have $\delta_\sharp \in \Gamma$.
\end{proof}

\section{Nef Cones}\label{section_nef cone}

In this section, we assume $s \geq 2$ and 
$d$ is sufficiently large.

Let $\Gamma=\{\delta_\flat =\delta_1 < \delta_2 < \cdots < \delta_N=\delta_\sharp\}$ be as in Proposition \ref{prop_gamma_s<r+1} for $s \leq r$ and Proposition \ref{prop_gamma_s >r} for $s > r$.
We set 
\[
M_i \coloneqq M_{C,\delta}(E_0, s, d), \quad M_{i,L} \coloneqq M_{C,\delta}(E_0, s, d)_L
\]
for $\delta \in (\delta_{i-1}, \delta_i)$, where $\delta_0 \coloneqq 0$ and $\delta_{N+1} \coloneqq \infty$.
Let $p_C^*E_0 \to \mc E_i$ be the universal object on $M_{i,L} \times C$. Then we have a group homomorphism
\[
\lambda_{\cale_i} : K(C) \to \Pic(M_{i,L})
\]
from the Grothendieck group $K(C)$ of $C$,
defined as the composite
\[
K(C) \xrightarrow{p_C^*} K^0(M_{i,L} \times C) \xrightarrow{\cdot [\mc E_i]} K^0(M_{i,L} \times C) \xrightarrow{p_!} K^0(M_{i,L}) \xrightarrow{\det} \Pic(M_{i,L}),
\]
where $p_C : M_{i,L} \times C \to C$ and $p : M_{i,L} \times C \to M_{i,L}$ are the natural projections and $p_!$ is defined by
\[
p_![G] = \sum_{\nu \geq 0} (-1)^\nu [R^{\nu} p_* G]
\]
(see \cite[Definition 8.1.1]{HL}).

Since $\Pic(M_{i,L}) \simeq \Z $ or $\Z^2$ by Corollary \ref{cor_pic_s_leq_r}, Proposition \ref{prop_SQM_s>r},
the image $\lambda_{\cale_i}([x])$ does not depend on $x \in C$.
Hence we simply write $[x] \in K(C)$ as $[{\rm pt}]$ for any $x \in C$ when we consider the image by $\lambda_{\cale_i} $.

\begin{lemma}\label{lem_class_P,P'}
Let $w=(a[\calo_C] +b[{\rm pt}]) \cdot \td(C)^{-1}  \in K(C)$,
where $\td(C) $ is the Todd class of $C$.
\begin{enumerate}
\item  Let $[E] \in M^s_C(s,L)$ and $\P \coloneqq \P(\Hom(E_0,E)) $ be the fiber of 
the morphism  
\[
 f : M_{1,L} \to M_C(s,L) : [(E,\alpha)] \mapsto [E]
\]
(see Remark \ref{rem_small_delta}).
Then $ \lambda_{\cale_1} (w)|_{\P} =\calo_{\P}(ad+bs ) $.
\item 
Let $1 \leq i \leq N$, $\dfrac{d + \delta_i}{s} = \dfrac{d'}{s'}$ with $1 \leq s' \leq s - 1$ and $d' \in \mathbb{Z}$. 
Let 
\[
[(E' \oplus F, \alpha \oplus 0)] \in \left( M^s_{C,\delta_i}(E_0, s - s', d - d') \times M_C^s(s', d') \right) \cap M_{C,\delta_i}(E_0, s, d)_L
\]
and $\P_i \coloneqq \mathbb{P}({\rm Ext}^1(F, E'))$,
which is the fiber of 
\[
p_i^- : M_{i,L} =M_{C,\delta_i-\ep}(E_0, s, d)_L \to M_{\delta_i} \coloneqq M_{C,\delta_i}(E_0, s, d)_L
\] over $[(E' \oplus F, \alpha \oplus 0)]$ by Lemma \ref{lemma-fiber of contraction}.
Then 
$ \lambda_{\cale_i} (w)|_{\P_i} =\calo_{\P_i}(-ad'-bs' )  $.
\end{enumerate}
\end{lemma}

\begin{proof}
(1) On $\P \times C$,
we have the universal homomorphism
\begin{align*}
p_C^* E_0 \to p_C^* E \otimes \calo_{\P}(1)
\end{align*}
and hence $\cale_{\P}\coloneqq p_C^* E \otimes \calo_{\P}(1) =\cale_1|_{\P \times C}$. 
Then
\begin{align*}
[\ch (\cale_{\P})]_{\leq 2} = (s,d[{\rm pt}],0) \left(1,l,\frac{l^2}{2}\right)  = \left(s, sl + d [{\rm pt}],  \frac{s}{2}l^2 +d l[{\rm pt}]\right)
\end{align*}
for $l=c_1(\calo_{\P}(1))$.
For $w=(a[\calo_C] + b[{\rm pt}]) \cdot \td(C)^{-1}  \in K(C)$,
\begin{align*}
[\ch (\cale_{\P})\otimes p_C^* w]_{\leq 2} = \left(as, asl + (ad + bs)[{\rm pt}], \frac{as}{2}l^2 +(ad+bs) l [{\rm pt}] \right)\cdot p_C^* \td(C)^{-1}.
\end{align*}
Hence $ \lambda_{\cale_1} (w)|_{\P}  = \lambda_{\cale_{\P}} (w) = c_1( p_* ([\ch (\cale_{\P})\otimes p_C^* w] \cdot p_C^* \td(C))) = (ad+bs) \calo_{\P}(1)$.\\[1mm]
(2) On $\P_i \times C$, we have the universal extension
\[
0 \to p_C^* E' \to\cale_{\P_i} \to p_C^* F \otimes p^* \mathcal{O}_{\P_i}(-1) \to 0
\]
for $\cale_{\P_i} \coloneqq \cale_i|_{\P_i\times C}$ as in \cite[Example 2.2.12]{HL}.
Then
\begin{align*}
[\ch (\cale_{\P_i})]_{\leq 2} &=(s-s', (d-d')[{\rm pt}],0) + (s', d'[{\rm pt}],0)\left(1,-l, \frac{l^2}{2}\right)   \\
&= \left(s, -s' l + d[{\rm pt}],  \frac{s'}{2} l^2 -d' l[{\rm pt}] \right)
\end{align*}
for $l=c_1(\calo_{\P_i}(1))$.
For $w=(a[\calo_C] + b[{\rm pt}]) \cdot \td(C)^{-1}  \in K(C)$,
\begin{align*}
[\ch (\cale_{\P_i})\otimes p_C^* w]_{\leq 2} = \left(as, -as'l + (ad + bs)[{\rm pt}], \frac{as'}{2} l^2 +(-ad'-bs') l [{\rm pt}] \right)\cdot p_C^* \td(C)^{-1}.
\end{align*}
Hence $ \lambda_{\cale} (w)|_{\P_i}  = \lambda_{\cale_{\P_i}} (w) = c_1( p_* ([\ch (\cale_{\P_i})\otimes p_C^* w] \cdot p_C^* \td(C))) = (-ad' -bs') \calo_{\P_i}(1)$.
\end{proof}

Let $p_C^*E_0 \to \t{\mc E}_i$ be the universal object on $M_i \times C$. As before, we have a group homomorphism $ K(C) \xrightarrow{\lambda_{\t \cale_i}} \Pic(M_{i})$ and by \cite[Lemma 8.1.2]{HL} we have $\lambda_{\t \cale_i}(w)|_{M_{i,L}}=\lambda_{\cale_i}(w)$ for $w\in K(C)$. In other words,
the homomorphism 
$\lambda_{\t \cale_i} : K(C) \to \Pic(M_{i})\to \Pic(M_i/\Pic^d(C))$ coincides with $\lambda_{\cale_i}$ under the natural identification $\Pic(M_i/\Pic^d(C))\xrightarrow{\sim} \Pic(M_{i,L})$. For this reason, we denote the image of $\lambda_{\t \cale_i}(w)$ in $\Pic(M_i/\Pic^d(C))$ by $\lambda_{\cale_i}(w)$ as well. The determinant morphism $M_{i}\to \Pic^d(C)$ is denoted by $\pi_i$.

\begin{proposition}\label{lem_Nef}
For simplicity, $\cale$ denotes $\cale_1$.
Assume $s \geq 2$.
\begin{enumerate}
\item  The natural surjective map $\Pic(M_1/\Pic^d(C) )_\Q \to 
N^1(M_1/\Pic^d(C) )_\Q$ is an isomorphism.
\item 
$\{ \lambda_{\mc E}([{\rm pt}]), \lambda_{\mc E}([\mc O_C])\}$ forms a basis of $N^1(M_{1}/\Pic^d(C))_{\R} \simeq  N^1(M_{1,L})_{\mb R}$.
\item   Let $\theta$ be the pullback of an ample class on $M_C(s,d)$ to $M_{1}$
via the morphism $ f : M_{1} \to M_C(s,d) : [(E,\alpha)] \mapsto [E]$.
Then 
\begin{align*}
\R_{\geq 0} \theta = 
\mb R_{\geq 0}\left(\left(\frac{d}{s}-(g-1)\right)\lambda_{\cale}([{\rm pt}] )-\lambda_{\cale} ([\calo_C])  \right)
\end{align*}
is a common edge of $\Nef (M_{1}/\Pic^d(C)), \Mov (M_{1}/\Pic^d(C)) $ and $\Eff (M_{1}/\Pic^d(C)),$ 
as well as
$\Nef (M_{1,L}), \Mov (M_{1,L}) $ and $\Eff (M_{1,L})$.
\item  For $1\leq i\leq N$ and $(s,i) \neq (r+1,N)$,
\footnote{Recall that $M_i$ and $M_{i,L}$ are SQMs of $M_1$ and  $M_{1,L}$ respectively for such $(s,i)$ by Propositions \ref{prop_SQM_s_leq_r}, \ref{prop_SQM_s>r}.
If $(s,i) = (r+1,N)$, $M_{N,L} =M_{C,\delta_\sharp-\ep}(E_0,r+1,d)_L $ has Picard number one by Lemma \ref{lem_s=r+1, last model}.
}
we have
\begin{align*}
\Nef(M_{i}/\Pic^d(C)) & =
\Nef(M_{i,L})  \\
&=  \mb R_{\geq 0}\left(\left(\frac{d+\delta_{i-1}}{s}  -(g-1) \right)\lambda_{\mc E}([{\rm pt}]) -\lambda_{\mc E}([\mc O_C]) \right) \\ 
 & \ + \mb R_{\geq 0}\left(\left(\frac{d+\delta_{i}}{s}  -(g-1) \right)\lambda_{\mc E}([{\rm pt}]) -\lambda_{\mc E}([\mc O_C]) \right).
\end{align*}
\item If $2 \leq s \leq r$, 
\footnote{In this case, $M_{N+1}=\Quot_C(E^0,r-s,d+e)$ and $M_{N+1,L}=\Quot_C(E^0,r-s,d+e)_{L \otimes \det E^0}$ by Lemma \ref{lem_suff_large_delta} and $(\delta_N,\delta_{N+1}) =(\delta_\sharp, \infty)$.}
we have
\begin{align*}
\Nef(M_{N+1}/\Pic^d(C)) & = \Nef(M_{N+1,L}) \\
&=  \mb R_{\geq 0} \left(\left(\frac{d+\delta_{N}}{s} -(g-1)\right)\lambda_{\mc E}([{\rm pt}])-\lambda_{\mc E}([\mc O_C])\right)  + \mb R_{\geq 0}\lambda_{\mc E}([{\rm pt}])\\
&=\mb R_{\geq 0} \left(\left(d-e_{s-1} -(g-1)\right)\lambda_{\mc E}([{\rm pt}])-\lambda_{\mc E}([\mc O_C])\right)  + \mb R_{\geq 0}\lambda_{\mc E}([{\rm pt}]).
\end{align*}
\end{enumerate}
Furthermore, all the extremal rays in (3), (4), (5) are spanned by ($\pi_i$-)semiample line bundles in $N^1(M_i/\Pic^d(C))_{\R}$ and $N^1(M_{i,L})_{\R}$.
\end{proposition}

\begin{proof}
(1) The morphism $\pi_1 $ is decomposed as 
$M_1 \xrightarrow{f} M_C(s,d) \xrightarrow{\det} \Pic^d(C) $.
Since $s \geq 2$, we have $\dim M_C(s,d)> \dim  \Pic^d(C)$.
Let $H \in \Pic (M_1)$ and $H' \in \Pic (M_C(s,d))$ be ample line bundles.
Then $H,f^* H'$ form a basis of $\Pic(M_1/\Pic^d(C) )_\Q$.
Let $C,C' \subset M_1$ be curves such that $C$ is contracted by $f$ and 
$f(C') \subset  M_C(s,d) $ is a contracted curve by $\det$.
If the image of $aH +b f^* H'$ in $N^1(M_1/\Pic^d(C) )$ is zero,
then $0=(aH +b f^* H'.C) = a (H.C)$ and $0 =(aH +b f^* H'.C') = a(H.C') + b (H'. f_*(C')) $.
Since $(H.C) >0 , (H'. f_*(C')) >0$,
we have $a =b=0$.
Thus (1) holds.\\
(2) 
Recall that
$\lambda_{\cale}  ([{\rm pt}])$ does not depend on the choice of ${\rm pt} \in C$.
Hence to compute $\lambda_{\cale}((a[\calo_C] +b[{\rm pt}]) \cdot \td(C)^{-1} )$, we may replace 
$\td(C)^{-1} $ with $[\calo_C] +(g-1) [{\rm pt}]$.

To prove (2),
it suffices to show that 
\[
\{\lambda_{\mc E}([{\rm pt}]), \lambda_{\mc E}([\mc O_C]) +(g-1) \lambda_{\mc E}([{\rm pt}])\} 
=\{\lambda_{\mc E}([{\rm pt}]  \cdot \td(C)^{-1}), \lambda_{\mc E}([\mc O_C] \cdot \td(C)^{-1})\} 
\]
forms a basis of $N^1(M_{1,L})_{\mb R} $.
Let $\P,\P_1 \subset M_{1,L}$ be as in Lemma \ref{lem_class_P,P'}.
By Lemma \ref{lem_class_P,P'},
\begin{align*}
\lambda_{\mc E}([{\rm pt}]  \cdot \td(C)^{-1})|_{\P} &=\calo_{\P}(s) , \quad \lambda_{\mc E}([\mc O_C] \cdot \td(C)^{-1})|_{\P} = \calo_{\P}(d),\\
\lambda_{\mc E}([{\rm pt}]  \cdot \td(C)^{-1})|_{\P_1} &=\calo_{\P_1}(-s') , \quad \lambda_{\mc E}([\mc O_C] \cdot \td(C)^{-1})|_{\P_1} = \calo_{\P_1}(-d') .
\end{align*}
If $  \lambda_{\mc E}([{\rm pt}]  \cdot \td(C)^{-1}) , \lambda_{\mc E}([\mc O_C] \cdot \td(C)^{-1})$ are linearly dependent,
it holds that $d/s =(-d')/(-s')= (d+\delta_1)/s$,
which contradicts $\delta_1 >0$.
Thus $ \lambda_{\mc E}([{\rm pt}]  \cdot \td(C)^{-1}),  \lambda_{\mc E}([\mc O_C] \cdot \td(C)^{-1})$ are linearly independent
and hence form a basis of $N^1(M_{1}/\Pic^d(C))_{\R}=N^1(M_{1,L})_{\mb R} \simeq \R^2$.
\\
(3) Since $\theta$ is the pullback of an ample class on $M_C(s,d)$, it is semiample and hence $f$-semiample and $\theta\in \Nef(M_{1}/\Pic^d(C))$. 
On the other hand, 
since $\dim M_1 > \dim M_C(s,d)$, we see that
$\R_{\geq 0} \theta$ spans a common edge of $\Nef (M_{1}/\Pic^d(C)), \Mov (M_{1}/\Pic^d(C)) $ and $\Eff (M_{1}/\Pic^d(C))$. The same argument shows that $\theta$ also
spans a common edge of $\Nef (M_{1,L}), \Mov (M_{1,L}) $ and $\Eff (M_{1,L})$.
Thus it suffices to compute the class of $\theta$.

Let $\P,\P_1 \subset M_{1,L}$  be  as in Lemma \ref{lem_class_P,P'} and let $\frac{d+\delta_1}{s} =\frac{d'}{s'}$.
By Lemma \ref{lem_class_P,P'},
$ \lambda_{\cale} (w_0)|_{\P}  = \calo_{\P}$ for 
\[
w_0= \left( \frac{d}{s}[{\rm pt}]-[\calo_C]  \right) \cdot \td(C)^{-1} =\left(\frac{d}{s}-(g-1)\right)[{\rm pt}] - [\calo_C] .
\]
Thus $\theta$ is a multiple of
\begin{align}\label{eq_theta}
\lambda_{\cale} (w_0) =  \left(\frac{d}{s}-(g-1)\right)\lambda_{\cale}([{\rm pt}]) -\lambda_{\cale} ([\calo_C]) .
\end{align}
Since $\theta$ is nef and
$ \lambda_{\cale} (w_0)|_{\P_1} = \calo_{\P_1}(d' -\frac{d}{s}s') =\calo_{\P_1}(\frac{s'}{s}\delta_1) $,
$\theta$ is a positive multiple of \eqref{eq_theta}.
Hence (3) holds.\\[1mm]
(4) For $1 \leq i \leq N$, let  $\dfrac{d + \delta_i}{s} = \dfrac{d'}{s'}$ with $1 \leq s' \leq s - 1$ and $d' \in \mathbb{Z}$, and
 $\P_i \subset M_{i,L}$ be a fiber of $p_i^- : M_{i,L} \to M_{\delta_i,L}$
as in Lemma \ref{lem_class_P,P'}.
By Lemma \ref{lem_class_P,P'},
$ \lambda_{\cale} (w_i)|_{\P_i} =  \calo_{\P_i}$ for 
\begin{align*}
    w_i \coloneqq  \left(\frac{d'}{s'}[{\rm pt}] - [\calo_C]  \right) \cdot \td(C)^{-1} &=\left(\frac{d'}{s'}-(g-1)  \right)[{\rm pt}]-[\calo_C]  \\
    & =\left(\frac{d+\delta_i}{s}-(g-1)  \right)[{\rm pt}] -[\mc O_C].   
\end{align*}
Since $N^1(M_{\delta_i}/\Pic^d(C))_{\R} \xrightarrow{\sim} N^1(M_{\delta_i,L})_{\R}$ is a one-dimensional subspace of $N^1(M_{i,L})_{\R} \simeq \mathbb{R}^2$, we have
\[
N^1(M_{\delta_i,L})_{\R} = \mathbb{R} \lambda_{\mc E}(w_i) \subset N^1(M_{i,L}).
\]
Let us denote by $\gamma_i\in N^1(M_i/\Pic^d(C)) \simeq N^1(M_{i,L})$  the pullback of an ample class from $M_i \to M_{\delta_i}$.
Since $ \gamma_1$ is nef and 
$ \lambda_{\cale} (w_1)|_{\P} = \calo_{\P}(-d  +\frac{d'}{s'}s) =\calo_{\P}(\frac{s'}{s}\delta_1) $ for $\P \subset M_{1,L}$ in Lemma \ref{lem_class_P,P'},
$ \gamma_1$ is a positive multiple of $ \lambda_{\cale} (w_1)$.
Hence 
\begin{align*}
\Nef(M_1/\Pic^d(C))= \Nef(M_{1,L}) &=\mb R_{\geq 0} \theta + \mb R_{\geq 0}  \lambda_{\cale} (w_1) \\
&=  \mb R_{\geq 0}\left(\left(\frac{d}{s}  -(g-1) \right)\lambda_{\mc E}([{\rm pt}]) -\lambda_{\mc E}([\mc O_C]) \right) \\ 
 & \  + \mb R_{\geq 0}\left(\left(\frac{d+\delta_{1}}{s}  -(g-1) \right)\lambda_{\mc E}([{\rm pt}]) -\lambda_{\mc E}([\mc O_C]) \right)
\end{align*}
holds.

For $2 \leq i \leq N$,
 $\gamma_i$ is in the opposite side of $\Nef(M_{i-1,L})$ from $\gamma_{i-1}$.
Hence we see that $\gamma_i$ is a positive multiple of $  \lambda_{\cale} (w_i) $ by induction on $i$.
Thus 
\begin{align*}
\Nef(M_i/\Pic^d(C))=\Nef(M_{i,L}) &=\mb R_{\geq 0} \lambda_{\cale} (w_{i-1}) + \mb R_{\geq 0}  \lambda_{\cale} (w_i) \\
&=  \mb R_{\geq 0}\left(\left(\frac{d+\delta_{i-1}}{s}  -(g-1) \right)\lambda_{\mc E}([{\rm pt}]) -\lambda_{\mc E}([\mc O_C]) \right) \\ 
 & \  + \mb R_{\geq 0}\left(\left(\frac{d+\delta_{i}}{s}  -(g-1) \right)\lambda_{\mc E}([{\rm pt}])-\lambda_{\mc E}([\mc O_C])\right)
\end{align*}
holds for $2 \leq i \leq N$.\\[1mm]
(5) Let $2 \leq s \leq r$.
Since $ \mb R_{\geq 0}  \lambda_{\cale} (w_N)$ is an edge of $\Nef(M_{N+1}/\Pic^d(C)) \xrightarrow{\sim} \Nef(M_{N+1,L})$ by (4),
it suffices to show that $\lambda_{\mc E}([{\rm pt}]) $ is a ($\pi_{N+1}$-)semiample but not ($\pi_{N+1}$-)ample class on $M_{N+1}$ and $M_{N+1,L}$. 

Let $p_C^* E_0 \to \t \cale_{N+1}$ be the universal object on $M_{N+1} \times C$.
Then we have
\[
0 \to \t{\mc E}_{N+1}^\vee \to p_C^*E^0 \to p_C^*E^0 / \t{\mc E}_{N+1}^\vee \to 0,
\]
which is the universal exact sequence on $M_{N+1} \times C= \Quot_C(E^0, r - s, d+e) \times C$. Hence we have a homomorphism
${\det}(\t{\mc E}_{N+1}^\vee)\hookrightarrow \wedge^sp_C^*E^0$ on $C\times M_{N+1}$ and this defines a morphism
\begin{align}\label{eq_lem_Nef_s_leq_r}
\varphi : M_{N+1}\to \Quot_C(\wedge^s E^0,\tilde{r}-1, \tilde{d}),
\end{align}
where $\tilde{r}=\rank \wedge^s E^0$ and $\tilde{d}={\deg}(\wedge^s E^0)+d$.
By \cite[Theorem 3.3]{GS-Picard}, we have that $\Quot_C(\wedge^s E^0,\tilde{r}-1, \tilde{d})$ is a projective bundle $\widetilde{\P} \to \Pic^d(C)$ whose fiber over $[L]$ is $\P(\Hom(L^{-1},\wedge^s  E^0 ))$.

\begin{claim}\label{claim_lambda(pt)}
The pullback $\varphi^* \mathcal{O}_{\widetilde{\P} } (1) $ of the tautological line bundle $\mathcal{O}_{\widetilde{\P} } (1)$ on $\widetilde{\P} $
coincides with $\lambda_{\t \cale}([{\rm pt}]) $ in  $ \Pic (M_{N+1}/\Pic^d(C))$.
\end{claim}

\begin{proof}[Proof of Claim \ref{claim_lambda(pt)}]
Since $ \Pic (M_{N+1}/\Pic^d(C)) \simeq \Pic (M_{N+1,L})$,
it suffices to show that the pullback of $\calo(1)$ by $\varphi_L =\varphi|_{ M_{N+1,L}} : M_{N+1,L} \to \P(\Hom(L^{-1},\wedge^s  E^0 ))$
coincides with $\lambda_{\cale}([{\rm pt}]) $ in  $  \Pic (M_{N+1,L})$.

Since $\det \cale_{N+1}|_{ [(E,\alpha)] \times C} =\det E=L$ for any $ [(E,\alpha)]  \in M_{N+1,L}$, 
there exists $H \in \Pic (M_{N+1,L})$ such that $\det \cale_{N+1} \simeq p^*H \otimes p_C^* L$ for the projection $p :  M_{N+1,L} \times C \to M_{N+1,L}$ by Seesaw principle.
Then for $x \in C$, we have 
\[
\lambda_{\cale}([{\rm pt}]) = p_* (\det \cale_{N+1} |_{ M_{N+1,L} \times \{x\}}) \simeq  p_* ((p^*H \otimes p_C^* L)  |_{ M_{N+1,L} \times \{x\}}) = H \in \Pic (M_{N+1,L}),
\]
where the first equality holds by  \cite[Examples 8.1.3 i)]{HL}.

On the other hand, it holds that 
$\varphi_L^*\calo(1) = p_* (\det \cale_{N+1} \otimes p_C^* L^{-1}) \simeq p_* p^* H =H$, where the first equality holds by the definition of $\varphi$.
Hence we have $ \varphi_L^*\calo(1) = H=\lambda_{\cale}([{\rm pt}]) $.
\end{proof}

Since $\mc O(1)$ is relatively ample, we get that  
the class 
\[
\lambda_{\cale}([{\rm pt}]) =\varphi^* \calo(1)  \in N^1(M_{N+1}/\Pic^d(C))_{\R} \xrightarrow{\sim} N^1(M_{N+1,L})_{\mb R}
\]
is ($\pi_{N+1}$-)semiample. 
We see that  $\lambda_{\mc E}([{\rm pt}]) $ is not  ample on $M_{N+1,L}$ by the same argument as \cite[Proposition 4.4]{GS-eff} (see also the proof of Proposition \ref{prop_canonical_divisor} in the case $s \leq r$).
In particular, $\lambda_{\mc E}([{\rm pt}]) $ is not  ($\pi_{N+1}$-)ample on $M_{N+1}$ and $M_{N+1,L}$.
\end{proof}

\section{Canonical divisors}\label{sec_canonical_div}

In this section, we show the following proposition, which describes the class of the canonical divisor of $M_{1,L}$.
This description is used to determine the effective cones of $M_{i,L}$ for $s=r\pm 1$ in the next section.
We use the notation in the previous section.

\begin{proposition}\label{prop_canonical_divisor}
    Assume $s \geq 2$. 
    Then the class of the canonical divisor $K_{M_{1,L}}$ is 
    \begin{align*}
(2s(g-1)-e-2d) \lambda_{\cale}([{\rm pt}]) + (2s-r) \lambda_{\cale} ([\calo_C]) .
\end{align*}
\end{proposition}

\begin{proof}
Let 
\begin{align*}
    K_{M_{i,L}} = \lambda_{\cale} ( ( a [{\rm pt}] +b [\calo_C] )  \cdot \td(C)^{-1})
    = (a +b(g-1)) \lambda_{\cale}([{\rm pt}]) + b \lambda_{\cale} ([\calo_C]) 
\end{align*}
for some $a, b$.

Recall that we have a natural morphism $f : M_{1,L} \to M_C(s,L)$,
which is a $\P^{se +r(d-s(g-1))-1}$-bundle over $M^s_C(s,L)$.
Let $[E] \in M^s_C(s,L)$ and  $\P=\P(\Hom(E_0,E)) = f^{-1}([E]) $ as  in the proof of Lemma \ref{lem_class_P,P'}.
Then $K_{M_{1,L}}|_{\P} = \calo_{\P}( -\dim \P-1)$ holds.
Since $ \lambda_{\cale} ( ( a [{\rm pt}] +b [\calo_C] )  \cdot \td(C)^{-1})|_{\P} = \calo_{\P}(as+bd)$ by Lemma \ref{lem_class_P,P'},
it holds that
\begin{align}\label{eq_ad+bs=}
    as+bd =-\dim \P -1 = -se -r(d-s(g-1)) =  - (e-r(g-1))s -rd .
\end{align}

First, we consider the case $s \geq r+2 $.
In this case,
we have a natural morphism 
\[
p_N^-: M_{N,L} \to M_{\delta_\sharp} (E_0,s,d)_L =M_{C}(s-r, L \otimes \det E^0),
\]
which is a $ \P^{rd+se +r(s-r)(g-1) -1}$-bundle over $M_{C}^s(s-r, L \otimes \det E^0)$ by Lemma \ref{lemma-M_delta^dagger_s>r}.
Let $[F] \in M_{C}^s(s-r, L \otimes \det E^0)$ and  $\P_N=\P(\Ext^1(F,E_0)) = {p_N^-}^{-1}([F]) $ as in the proof of Lemma \ref{lem_class_P,P'}.
Then $K_{M_{N,L}}|_{\P_N} = \calo_{\P_N}( -\dim \P_N-1)$ holds.
Since  $\frac{d+\delta_\sharp}{s}=\frac{d+e}{s-r}$, it holds that 
 $\lambda_{\cale} ( ( a [{\rm pt}]+b [\calo_C] ) \cdot \td(C)^{-1})|_{\P_N} =\calo_{\P_N}( -a(s-r)-b(d+e))  $ by Lemma \ref{lem_class_P,P'}.
Hence
\begin{align}\label{eq_ad'+bs'=}
-a(s-r) -b(d+e)=-\dim \P_N -1 =-rd-se -r(s-r)(g-1)
\end{align}
holds.
By \eqref{eq_ad+bs=} and \eqref{eq_ad'+bs'=},
we have $a=r(g-1) -e-2d$ and  $b= 2s-r$,
and hence
\begin{align*}
   K_{M_{N+1,L}} &=  (a +b(g-1)) \lambda_{\cale}([{\rm pt}])+b \lambda_{\cale} ([\calo_C])  \\
   &=  (2s(g-1) -e-2d ) \lambda_{\cale}([{\rm pt}])+(2s-r) \lambda_{\cale} ([\calo_C]) . 
\end{align*}
Thus this proposition holds in the case $s \geq r+2$.

Next, we consider the case $s = r+1 $.
In this case, $\delta_{N-1}$ and $\delta_N$ coincide with $\kappa_\sharp$ and $\delta_\sharp$ in Lemma \ref{lemma-upper bound on dim of W,s >r} respectively, and 
we have a divisorial contraction
\begin{align*}
M_{N-1,L} =M_{C,\kappa_\sharp-\varepsilon}(E_0, r+1, d)_L \xrightarrow{\mu} M_{C,\delta_\sharp-\varepsilon}(E_0, r+1, d)_L = \P(\Ext^1(L \otimes \det E^0, E_0)) \eqqcolon \P_N
\end{align*}
by Lemma \ref{lem_s=r+1, last model}.
Hence there exists a closed subset $Z \subset \P_N$ whose codimension is at least two such that $\P_N \setminus Z $ is naturally embedded into $ M_{N-1,L} $.
Then $K_{M_{N-1,L}}|_{\P_N \setminus Z} =K_{\P_N \setminus Z}= \calo_{\P_N \setminus Z}( -\dim \P_N-1)$ holds.
Since $\lambda_{\cale} ( ( a[{\rm pt}] +b [\calo_C]) \cdot \td(C)^{-1})|_{\P_N} =\calo_{\P_N}( -a(s-r)-b(d+e))  $ as in the case $s \geq r+2$ above, we have
\begin{align}\label{eq_ad'+bs'=,s=r+1}
\begin{split}
-a(s-r)-b(d+e)  =-\dim \P_N -1 &=-rd-(r+1)e -r(g-1) \\
&= -rd -se -r(s-r)(g-1).
\end{split}
\end{align}
By \eqref{eq_ad+bs=} and \eqref{eq_ad'+bs'=,s=r+1},
we have $a=r(g-1) -e-2d$ and $b= 2s-r$.
Hence this proposition holds in the case $s =r+1$.

Lastly, we consider the case $s \leq r$.
To obtain another equation,
we consider a fiber of the morphism 
\[
\varphi_L: M_{N+1,L} =\Quot_C(E^0,r-s,d+e)_{L\otimes \det E^0} \to \P(\Hom(L^{-1}, \wedge^s E^0)) 
\]
which maps $[E^0 \stackrel{q}{\twoheadrightarrow}Q]$ to  $[\det (\ker q) \hookrightarrow \wedge^s E^0]$.
Fix $x \in C$ and take general 
\[
[E^0 \stackrel{q'}{\twoheadrightarrow} Q'] \in \Quot_C(E^0, r-s,d+e-1)_{L \otimes \det E^0\otimes \calo_C(-x)}.
\]
Since $d $ is sufficiently large, we can take such $[E^0 \stackrel{q'}{\twoheadrightarrow} Q']$
so that $Q'$ is locally free of rank $r-s$ at the point $x$ by \cite[\S 8]{GS-Picard}.
Let $V \subset \Hom(E^0, k(x))$ be a general subspace of dimension two.
Since $Q'$ is locally free, the dimension of $ \ker (q' \otimes k(x)) \subset E^0\otimes k(x)$ is $s \geq 2$.
Since $V$ is general, $V \cap \ker (q' \otimes k(x))^{\perp} =0$ and 
hence $(q',v) : E^0 \to Q' \oplus k(x)$ is surjective for any $v \in V \setminus \{0\}$.
Thus we have an exact sequence
\begin{align*}
 0 \to \bar{\cale}^{\vee} \to p_C^* E^0 \to 
 p_C^* Q' \oplus ( p_C^* k(x) \otimes p^*\calo_{\P(V)}(1) )\to 0
\end{align*}
on $C \times \P(V)$,
which induces a morphism $\P(V) \to \Quot_C(E^0, r-s,d+e)_{L \otimes \det E^0} =M_{N+1,L}$,
which is an inclusion by the generality of $V$.
Since 
\begin{align*}
\det ( p_C^* Q' \oplus ( p_C^* k(x) \otimes p^*\calo_{\P(V)}(1) )) = \det p_C^* Q' \otimes p_C^* \calo_C(x) \simeq p_C^* (L \otimes \det E^0),
\end{align*}
we have $\det \bar{\cale} \simeq p_C^* L$.
Since $\bar{\cale} = \cale_{N+1}|_{\P(V)} $ and $\varphi_L^* \calo(1) \simeq p_* (\det \cale_{N+1} \otimes p_C^* L^{-1} )$ for $p : M_{N+1,L} \times C \to M_{N+1,L}$,
it holds that $\varphi_L^* \calo(1) |_{\P(V)} $ is a trivial bundle.
Hence $\P(V)$ is contracted by the morphism $\varphi_L: M_{N+1,L} \to \P(\Hom(L^{-1}, \wedge^s E^0))$.

Let $\cale=\cale_{N+1}$ on $M_{N+1,L} =\Quot_C(E^0, r-s,d+e)_{L \otimes \det E^0}$ and
\begin{align*}
   0 \to \cale^{\vee} \to p_C^* E^0 \to  p_C^* E^0/ \cale^\vee \to 0
\end{align*}
be the universal exact sequence on $\Quot_C(E^0, r-s,d+e)_{L \otimes \det E^0} \times C$.
Let $T_{M_{N+1}}$ be the tangent sheaf of $M_{N+1}=\Quot_C(E^0, r-s,d+e)$.
Outside a codimension two subset,
\footnote{That is, on the open subset $U$ where $\Ext^1 (E^\vee, E^0/E^\vee) =H^1(E \otimes (E^0/E^\vee)) =0$.
We note that this vanishing holds if $E$ is semistable since $ H^1(E \otimes E^0) =0$ by the semistability of $E$ and $d \gg 0$.
Since $E$ is semistable if $[(E,\alpha)] \in M_{1,L}$,
$M_{N+1,L} \setminus U$ is contained in the indeterminacy locus of $M_{N+1,L} \dashrightarrow M_{1,L}$
and hence the codimension of $M_{N+1,L} \setminus U$ is at least two.
}
the restriction $ T_{M_{N+1}}|_{M_{N+1,L}}$ coincides with
$p_* \calh om (\cale^{\vee}, p_C^* E^0/ \cale^\vee) $
and $R^ip_* \calh om (\cale^{\vee}, p_C^* E^0/ \cale^\vee)=0 $ for $i >0$.
Hence $-K_{M_{N+1,L}} = -K_{M_{N+1}}|_{M_{N+1,L}}$ is the first chern class of 
$p_! ( \calh om (\cale^{\vee}, p_C^* E^0/ \cale^\vee))$.

Now we compute $p_! ( \calh om (\cale^{\vee}, p_C^* E^0/ \cale^\vee))|_{\P(V)}$.
For $\calq \coloneqq  p_C^* Q' \oplus ( p_C^* k(x) \otimes p^*\calo_{\P(V)}(1) )$,
it holds that
\begin{align}\label{eq_chern}
\begin{aligned}
\ch(\calq)_{\leq 2} &= (r-s,(d+e-1)[{\rm pt}] ,0) + (0,[{\rm pt}],0)\left(1,l,\frac{l^2}2\right)\\
& = (r-s,(d+e)[{\rm pt}], l[{\rm pt}]),     \\
\ch(\bar{\cale}^\vee)_{\leq 2} &=(r, e[{\rm pt}],0) - \ch(\calq)_{\leq 2} 
= (s,-d[{\rm pt}], -l[{\rm pt}]), \\
\ch(\bar{\cale} \otimes \calq)_{\leq 2}
&= (s,d[{\rm pt}], -l[{\rm pt}]) (r-s,(d+e)[{\rm pt}], l[{\rm pt}]) \\
&=(s(r-s), (s(d+e) +(r-s)d)[{\rm pt}], (2s-r)l[{\rm pt}])
\end{aligned}
\end{align}
and hence 
\begin{align*}
c_1(p_! (\bar{\cale} \otimes \calq)) = c_1(p_* ([\ch(\bar{\cale} \otimes \calq)]\cdot p_C^* \td(C)))   = \calo_{\P(V)}(2s-r).
\end{align*}
Thus we have $K_{M_{N+1,L}}|_{\P(V)} =-c_1(p_! (\bar{\cale} \otimes \calq))  = \calo_{\P(V)}(-2s+r) $.

On the other hand,
\begin{align*}
\lambda_{\cale}([{\rm pt}])|_{\P(V)} =\calo_{\P(V)}, \quad \lambda_{\cale} ([\calo_C])|_{\P(V)} = c_1(p_! (\bar{\cale})) = \calo_{\P(V)}(-1) 
\end{align*}
by \eqref{eq_chern}.
Thus we have
\begin{align*}
    K_{M_{N+1,L}}|_{\P(V)} 
    = \left( (a +b(g-1)) \lambda_{\cale}([{\rm pt}]) +b \lambda_{\cale} ([\calo_C])  \right) |_{\P(V)} = \calo_{\P(V)}(-b).
\end{align*}
Hence $b =2s-r$ holds.
Then $a= r(g-1)-e-2d $ by \eqref{eq_ad+bs=}
and hence
\begin{align*}
   K_{M_{N+1,L}} &=  (a +b(g-1)) \lambda_{\cale}([{\rm pt}])  +b \lambda_{\cale} ([\calo_C])  \\
   &=  (2s(g-1) -e-2d ) \lambda_{\cale}([{\rm pt}])+ (2s-r) \lambda_{\cale} ([\calo_C]). \qedhere
\end{align*}
\end{proof}

\section{Proof of the main result}\label{sec_proof_of_main_result}

Let 
\begin{align*}
\tilde{\Gamma} =\begin{cases}
\emptyset &  \text{ if } s=1\\
\Gamma  & \text{ if } 2 \leq s \leq r\\
\Gamma \cup (\delta_\sharp, \infty )  & \text{ if } s > r,
\end{cases}
\end{align*}
where $\Gamma$ is as in Proposition \ref{prop_gamma_s<r+1} for $s \leq r$ and in Proposition \ref{prop_gamma_s >r} for $s >r$.
For sufficiently large $d$, we have
\begin{align*}
\delta_\sharp=
\begin{cases}
(s-1)d-se_{s-1} & \text{ if }  2 \leq s \leq r \\
\frac{rd +se}{s-r} &  \text{ if }  s > r .
\end{cases}
\end{align*}
Furthermore, let $\kappa_\sharp =rd+(r+1)(e-1)$ for $s >r$ as in Lemma \ref{lemma-upper bound on dim of W,s >r}.

In this section, we prove the following theorem, which is the main result of this paper.

\begin{theorem}\label{thm_MDS}
Let $C$ be a smooth projective curve of genus $g\geq 2$ over $\C$.
Let $E^0$ be a vector bundle on $C $ of rank $r \geq 1$ and degree $e$.
Let $s \geq 1$.
Assume  $d $  is sufficiently large depending on $g, E_0, r,s$.
Then the following hold for any $\delta \in (0,\infty) \setminus \tilde{\Gamma}$ and $[L] \in \Pic^d (C)$.
\begin{enumerate}
\setlength{\itemsep}{0mm}
\item $\pi_\delta: M_{C,\delta}(E_0,s,d)\to \Pic^d(C) , \ [(E,\alpha)] \mapsto [\det E]$ is a Mori dream morphism and $M_{C,\delta}(E_0,s,d)_L$ is a Mori dream space.
\item If $s=1$, or $s=r+1$ and $\kappa_\sharp < \delta < \delta_\sharp$, then $\pi_\delta$ is a projective bundle and
$M_{C,\delta}(E_0,s,d)_L$ is a projective space.
\item  If $2 \leq s  \leq r$, we have
\begin{align*}
   \Mov (M_{C,\delta}(E_0,s,d)/\Pic^d(C))&=\Mov (M_{C,\delta}(E_0,s,d)_L) \\
   &=  
   \begin{cases}
     \R_{\geq 0} \theta_0  + \R_{\geq 0} \lambda_{\cale}([{\rm pt}])  &  \text{ if } s=r-1, r \\
     \R_{\geq 0} \theta_0  + \R_{\geq 0} \theta'_0   & \text{ if }  s \leq r-2
   \end{cases}
\end{align*}
and 
\begin{align*}
   \Eff(M_{C,\delta}(E_0,s,d)/\Pic^d(C)) &= \Eff (M_{C,\delta}(E_0,s,d)_L) \\
   &=  
   \begin{cases}
     \R_{\geq 0} \theta_0  + \R_{\geq 0} \lambda_{\cale}([{\rm pt}])  &  \text{ if } s= r \\
     \R_{\geq 0} \theta_0  + \R_{\geq 0} \theta'_0   & \text{ if }  s \leq r-1
   \end{cases}
\end{align*}
for
\begin{align*}
\theta_0&=  \left(\frac{d}{s}-(g-1)\right)\lambda_{\cale}([{\rm pt}])  -\lambda_{\cale} ([\calo_C])  ,\\
\theta'_0&=    \left(\frac{d+e}{r-s}+(g-1)\right)\lambda_{\cale}([{\rm pt}]) + \lambda_{\cale} ([\calo_C]) .
\end{align*}
\item If $s= r+1$ and $\delta < \kappa_\sharp$, or $ s \geq r+2$,
we have
\begin{align*}
   \Mov(M_{C,\delta}(E_0,s,d)/\Pic^d(C)) &= \Mov (M_{C,\delta}(E_0,s,d)_L) \\
   &=  
   \begin{cases}
     \R_{\geq 0} \theta_0  + \R_{\geq 0} \gamma'  &  \text{ if } s=r+1 \\
     \R_{\geq 0} \theta_0  + \R_{\geq 0} \gamma   & \text{ if }  s \geq r+2
   \end{cases}
\end{align*}
and 
\begin{align*}
  \Eff(M_{C,\delta}(E_0,s,d)/\Pic^d(C)) = \Eff (M_{C,\delta}(E_0,s,d)_L) =       \R_{\geq 0} \theta_0  + \R_{\geq 0} \gamma  
\end{align*}
for 
\begin{align*}
 \theta_0&=     \left(\frac{d}{s}-(g-1)\right)\lambda_{\cale}([{\rm pt}])  -\lambda_{\cale} ([\calo_C]) ,\\
 \gamma&=   \left(\frac{d+e}{s-r}-(g-1)\right)\lambda_{\cale}([{\rm pt}]) -\lambda_{\cale} ([\calo_C]) ,\\
 \gamma'&=   \left(d+e-g\right)\lambda_{\cale}([{\rm pt}]) -\lambda_{\cale} ([\calo_C]) .
\end{align*}
\end{enumerate}
\end{theorem}

The slices of $\Mov(M_{i,L})$ for $2 \leq s \leq r$ are described in Figures \ref{mov_0,1}, \ref{mov_2}.
The slices of $\Mov(M_{i,L})$ for $s >r$ are described in the following figures.

\begin{figure}[htbp]
\centering
\begin{tikzpicture}
\draw(0,0)--(12.5,0);

\foreach \x in {2.5,5,10,12.5,15}
  \draw[very thick] (15-\x,5pt)--(15-\x,-5pt);

\foreach \y/\ytext in {2.5/$\theta_N=\gamma$,5/$\theta_{N-1}$,10/$\theta_2$,12.5/$\theta_1$,15/$\theta_0$}
  \draw (15-\y,0) node[above=1ex] {\ytext};

\foreach \z/\ztext in {3.75/$M_{N,L}$,7.5/$ \cdots\cdots$,11.25/$M_{2,L}$,13.75/$M_{1,L}$}
  \draw (15-\z,0) node[below] {\ztext};
\end{tikzpicture}
\caption{The slice of $\Mov(M_{i,L})=\Eff(M_{i,L})$ for $s \geq r +2$} \label{mov_geq_r+2}
\end{figure}
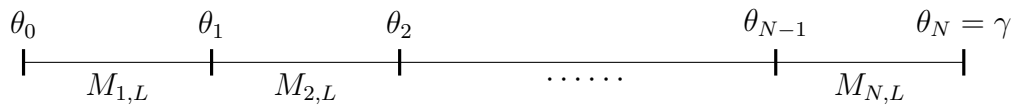

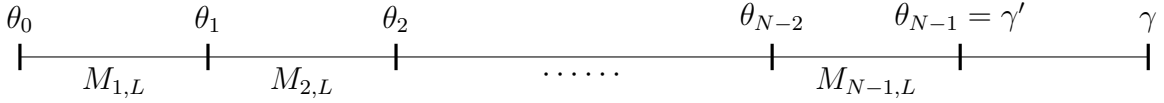
\begin{figure}[htbp]
\centering
\begin{tikzpicture}
\draw(0,0)--(15,0);

\foreach \x in {0,2.5,5,10,12.5,15}
  \draw[very thick] (15-\x,5pt)--(15-\x,-5pt);

\foreach \y/\ytext in {0/$\gamma$,2.5/$\theta_{N-1}=\gamma'$,5/$\theta_{N-2}$,10/$\theta_2$,12.5/$\theta_1$,15/$\theta_0$}
  \draw (15-\y,0) node[above=1ex] {\ytext};

\foreach \z/\ztext in {3.75/$M_{N-1,L}$,7.5/$ \cdots\cdots$,11.25/$M_{2,L}$,13.75/$M_{1,L}$}
  \draw (15-\z,0) node[below] {\ztext};
\end{tikzpicture}
\caption{The slice of $\Mov(M_{i,L})= \mb R_{\geq 0} \theta_0 +  \mb R_{\geq 0} \gamma'$ and $\Eff(M_{i,L})=\mb R_{\geq 0} \theta_0 +  \mb R_{\geq 0} \gamma$ for $s = r +1 \geq 2$} \label{mov_r+1}
\end{figure}

In the rest of this section,
we prove this theorem.

\subsection{The case when the (relative) Picard number is one}

\begin{proof}[Proof of Theorem \ref{thm_MDS}, the case $s =1$, or $s=r+1$ and $\kappa_\sharp < \delta < \delta_\sharp$]
\mbox{}

If $s=1$, $M_{C,\delta} (E_0,1,d) = \Quot_C(E^0,r-1,d+e) $ for any $\delta >0$.
By Lemma \ref{lemma-small delta} and Remark \ref{rem_small_delta},
\[
\pi_\delta: M_{C,\delta} (E_0,1,d) \to M_C(1,d)=\Pic^d (C) , \ [(E,\alpha)] \mapsto [E] 
\]
is a projective bundle 
and hence $M_{C,\delta} (E_0,1,d)_L$ is a projective space.
In particular,
$\pi_\delta$ is  a Mori dream morphism and
$M_{C,\delta} (E_0,1,d)_L$ is a Mori dream space.

If $s=r+1$ and $\kappa_\sharp < \delta <\delta_\sharp$,
the morphism 
\[
M_{C,\delta} (E_0,r+1,d) =M_{C,\delta_\sharp -\ep} (E_0,r+1,d) \to M_{C,\delta_\sharp} (E_0,r+1,d)=\Pic^{d+e} (C) ,
\]
which maps $[(E,\alpha)] $ to $[E/\im \alpha]$,
is a projective bundle by Lemmas \ref{lemma-M_delta^dagger_s>r}, \ref{lem_s=r+1, last model}.
Hence $M_{C,\delta} (E_0,r+1,d)  \to \Pic^{d+e} (C) \simeq \Pic^d(C) $ is a Mori dream morphism and
$M_{C,\delta} (E_0,r+1,d)_L$ is a projective space.
\end{proof}

\subsection{The case $s \geq r+2$}

\begin{proof}[Proof of Theorem \ref{thm_MDS}, the case $s \geq r+2$]
\mbox{}

We use the notation in the previous sections.
In this case,
the morphism 
\[
M_{N}\to M_{C,\delta_\sharp}(E_0,s,d)=M_{C}(s-r,d+e), \ [(E,\alpha) ] \mapsto [E/\im \alpha]
\] in Lemma \ref{lemma-M_delta^dagger_s>r} is a surjective morphism with 
$\dim M_{N} > \dim M_{C}(s-r,d+e)$ and $\dim M_{N,L} > \dim M_{C}(s-r,d+e)_{L \otimes \det E^0}$.
Hence the pullback of an ample line bundle on $M_{C}(s-r,d+e)$ to 
$M_{N}$
spans a common edge of $\Mov(M_N/\Pic^d(C))$, $\Eff(M_N/\Pic^d(C))$, $\Mov(M_{N,L}) $ and $\Eff(M_{N,L})$. 
By the proof of Proposition \ref{lem_Nef} (4),
the pullback is a positive multiple of 
\begin{align*}
\left(\frac{d+\delta_{N}}{s}  -(g-1) \right)\lambda_{\mc E}([{\rm pt}]) -\lambda_{\mc E}([\mc O_C])=\left(\frac{d+e}{s-r}-(g-1)\right)\lambda_{\cale}([{\rm pt}]) -\lambda_{\mc E}([\mc O_C])=\gamma
\end{align*}
by $\delta_N=\delta_\sharp =\frac{rd+se}{s-r}$.
By Proposition  \ref{lem_Nef} (1),
$\R_{\geq 0} \theta_0 $
spans another  common edge of 
 $\Mov(M_N/\Pic^d(C))$, $\Eff(M_N/\Pic^d(C))$, $\Mov(M_{N,L}) $ and $\Eff(M_{N,L})$. Hence $\Mov(M_{N,L}) = \Eff(M_{N,L}) =     \R_{\geq 0} \theta_0  + \R_{\geq 0} \gamma  $ holds. 

By Proposition \ref{lem_Nef},
$\Mov(M_{N}/\Pic^d (C)) =\Mov(M_{N,L})$ is the union of $\Nef (M_i/\Pic^d (C)) = \Nef (M_{i,L}) $ for $1 \leq i \leq N$
and each nef cone is spanned by ($ \pi_i$-)semiample line bundles,
where $\pi_i =\pi_\delta$ for $\delta \in (\delta_{i-1},\delta_i)$.
Hence $\pi_i$ is a Mori dream morphism and $M_{i,L}$ is a Mori dream space.
\end{proof}

\subsection{The case $ s =r +1$ and $0 < \delta < \kappa_\sharp$}

\begin{proof}[Proof of Theorem \ref{thm_MDS}, the case $s =r+1$ and $0 < \delta < \kappa_\sharp$]
\mbox{}

By Lemma \ref{lem_s=r+1, last model},
the morphism $p^-_{\kappa_\sharp} : M_{N-1} =M_{C,\kappa_\sharp -\ep}(E_0,s,d) \to M_{C,\kappa_\sharp}(E_0,s,d)$ factors as 
\[
M_{N-1} =M_{C,\kappa_\sharp -\ep}(E_0,s,d) \xrightarrow{\mu}  M_{N} = M_{C,\delta_\sharp-\varepsilon}(E_0,s,d) \xrightarrow{p^{+}_{\kappa_\sharp}} M_{C,\kappa_\sharp}(E_0,s,d)
\]
with a divisorial contraction $\mu$  and a finite morphism $ p^{+}_{\kappa_\sharp}$. 
Hence the pullback $\gamma_{N-1}$ of an ample line bundle on $M_{C,\kappa_\sharp}(E_0,s,d)$ 
spans an edge of $\Mov(M_{N-1}/\Pic^d(C))$.
Since $\mu_L :M_{N-1,L} \to M_{N,L}$ also contracts a divisor by Lemma \ref{lem_similar_to_M_L} (2),
the same proof shows that the class $\gamma_{N-1}$ spans an edge of $\Mov(M_{N-1,L})$ as well.
Moreover,
the class $[W]$ of the exceptional divisor $W$ of $\mu$ in $N^1(M_{N-1}/\Pic^d(C))_{\R} = N^1(M_{N-1,L})_{\mb R}$ spans  an edge of $\Eff(M_{N-1}/\Pic^d(C))$ and $\Eff(M_{N-1,L})$.

By the proof of Proposition \ref{lem_Nef} (4),
$\gamma_{N-1}$ is a positive multiple of 
\begin{align}\label{eq_thm_MDS_s=r+1}
\left(\frac{d+\delta_{N-1}}{s}  -(g-1) \right)\lambda_{\mc E}([{\rm pt}]) -\lambda_{\mc E}([\mc O_C])=  (d+e-g)\lambda_{\cale}([{\rm pt}]) -\lambda_{\mc E}([\mc O_C])=\gamma'
\end{align}
by $\delta_{N-1} =\kappa_\sharp=rd +(r+1)(e-1)$.
Hence $\R_{\geq 0} \gamma'$ is an edge of  $\Mov(M_{N-1}/\Pic^d(C))$ and $\Mov(M_{N-1,L})$.

To determine $\R_{\geq 0} [W]$,
we compare the canonical divisors of $ M_{N-1,L} $ and $M_{N,L}$.
In this case,
\[
M_N =M_{C,\delta_\sharp-\ep} (E_0,r+1,d) \to M_{C,\delta_\sharp} (E_0,r+1,d)=M_C(1,d+e) =\Pic^{d+e}(C)
\]
is a projective bundle and $M_{N,L}=\P(\Ext^1(L\otimes \det E^0,E_0))$ by Lemmas \ref{lemma-fiber of contraction}, \ref{lem_s=r+1, last model}.
By Proposition \ref{prop_canonical_divisor},
\begin{align*}
 K_{M_{N-1,L}} = (2(r+1)(g-1) -e- 2d)\lambda_{\cale}([{\rm pt}])  + (r+2) \lambda_{\cale} ([\calo_C]).
\end{align*}
On the other hand,
$\mu^*  K_{M_{N,L} }  $ is a multiple of
\begin{align*}
\gamma' = (d+e-g)\lambda_{\cale}([{\rm pt}] ) -\lambda_{\cale} ([\calo_C])  =\lambda_{\cale} (((d+e-1)[{\rm pt}] -[\calo_C] )\cdot \td(C)^{-1})
\end{align*}
by \eqref{eq_thm_MDS_s=r+1}.
Applying Lemma \ref{lem_class_P,P'} (2) to $\delta=\delta_\sharp$,
the restriction of $\gamma'$ to $M_{N-1,L} \setminus W=M_{N,L} \setminus \mu(W) $ is $\calo(1)$.  
Hence we have
\begin{align*}
\mu^* K_{M_{N,L}} = \mu^* K_{\P(\Ext^1(L\otimes \det E^0,E_0))}  &= -(\dim \P(\Ext^1(L\otimes \det E^0,E_0))+1) \gamma' \\
&=-(dr+(r+1)e + r(g-1)) \gamma'.
\end{align*}
Since
\begin{align*}
\gamma= (d+e-g+1)\lambda_{\cale}([{\rm pt}] )  -\lambda_{\cale} ([\calo_C])  =\gamma' + \lambda_{\cale}([{\rm pt}]) ,
\end{align*}
it holds that 
\begin{align*}
K_{M_{N-1,L}} &-\mu^*    K_{M_{N,L}} \\
&=   (2(r+1)(g-1) -e- 2d)\lambda_{\cale}([{\rm pt}]) + (r+2) \lambda_{\cale} ([\calo_C]) +(dr+(r+1)e + r(g-1)) \gamma'\\
 &=(dr +(r+1)e +r(g-1) -r-2) \gamma.
\end{align*}
Since $dr +(r+1)e +r(g-1) -r-2 \neq 0$ by $d \gg 0$, 
$[W]$ is a multiple of $\gamma$.
Since $[W]$ is on the opposite side of 
\begin{align*}
\Nef (M_{N-1,L}) =\mb R_{\geq 0}& \left(\left(\frac{d+\delta_{N-2}}{s} -(g-1)\right)\lambda_{\mc E}([{\rm pt}])-\lambda_{\mc E}([\mc O_C]) \right)  \\
&+ \mb R_{\geq 0} \left(\left(\frac{d+\delta_{N-1}}{s} -(g-1)\right)\lambda_{\mc E}([{\rm pt}]) -\lambda_{\mc E}([\mc O_C]) \right) 
\end{align*}
from $ \mb R_{\geq 0} \left(\left(\frac{d+\delta_{N-1}}{s} -(g-1)\right)\lambda_{\mc E}([{\rm pt}]) -\lambda_{\mc E}([\mc O_C]) \right)=\R_{\geq 0} \gamma' $,
$[W]$ is a positive multiple of $\gamma$.

Since $\R_{\geq 0} \theta_0 $ is a common edge of $\Mov(M_{N-1}/\Pic^d(C))$, $\Eff(M_{N-1}/\Pic^d(C))$, $\Mov(M_{N-1,L})$ and $\Eff(M_{N-1,L}) $,
we have  $\Mov(M_{N-1}/\Pic^d(C))=\Mov(M_{N-1,L}) =    \R_{\geq 0} \theta_0  + \R_{\geq 0} \gamma'  $
and $\Eff(M_{N-1}/\Pic^d(C)) =\Eff(M_{N-1,L}) = \R_{\geq 0} \theta_0  + \R_{\geq 0} \gamma$.

By Proposition \ref{lem_Nef},
$ \Mov(M_{N-1}/\Pic^d(C))=\Mov (M_{N-1,L})$ is the union of $\Nef (M_{i}/\Pic^d(C))= \Nef (M_{i,L})$ for $1 \leq i \leq N-1$
and each nef cone is spanned by ($\pi_i$-)semiample line bundles.
Hence $\pi_i$ is a Mori dream morphism and $M_{i,L}$ is a Mori dream space.
\end{proof}

\subsection{The case $ s =r \geq 2$}
In this case, $\wedge^s E^0 =\wedge^r E^0$ is a line bundle and hence 
the morphism \eqref{eq_lem_Nef_s_leq_r} is 
\begin{align}\label{eq_quot-to-chow}
M_{N+1}=\Quot_C(E^0, 0, d+e) \to \Quot_C(\wedge^r E^0,0,d+e) \simeq \Sym^{d+e} (C).
\end{align}
As in the proof of Lemma \ref{lem_Nef} (5), 
$\Quot_C(\wedge^r E^0,0,d+e)$ is a projective bundle over $\Pic^d(C)$ and
$\lambda_{{\mc E}}([{\rm pt}])$ is the class of the pullback of $\calo(1)$ in 
$N^1(M_{N+1}/\Pic^d(C))$.

\begin{proof}[Proof of Theorem \ref{thm_MDS}, the case $s=r \geq 2$]
\mbox{}

In this case, $\dim  \Quot_C(\wedge^r E^0,0,d+e)=d+e $ is less than  $\dim M_{N+1} = r(d+e)$ by $r = s \geq  2$.
Hence $\lambda_{\mc E}([{\rm pt}])$ spans a common edge of $\Mov (M_{N+1}/\Pic^d(C)), \Eff (M_{N+1}/\Pic^d(C))$, $\Mov (M_{N+1,L}), \Eff (M_{N+1,L})$.
Since $\theta_0$ spans a common edge of these cones as well,
all these cones coincide with
$ \R_{\geq 0} \theta_0 + \R_{\geq 0}\lambda_{\cale}([{\rm pt}])$.
By Proposition \ref{lem_Nef},
$ \Mov (M_{N+1}/\Pic^d(C))= \Mov (M_{N+1,L})= \R_{\geq 0} \theta_0 + \R_{\geq 0}\lambda_{\cale}([{\rm pt}])$ is the union of $ \Nef (M_{i}/{\Pic^d(C)})= \Nef (M_{i,L})$ for $1 \leq i \leq N+1$
and each nef cone is spanned by ($\pi_i$-)semiample line bundles.
Hence $\pi_i$ is a Mori dream morphism and $M_{i,L}$ is a Mori dream space.
\end{proof}

\subsection{The case $s =r-1 \geq 2$}

\begin{proof}[Proof of Theorem \ref{thm_MDS}, the case $s=r-1 \geq 2$]
\mbox{}

By  $\wedge^{r-1} E^0 \simeq E_0\otimes \det E^0$, the morphism \eqref{eq_lem_Nef_s_leq_r} is 
\begin{align*}
\varphi : M_{N+1} =\Quot_C(E^0,1,d+e) \to 
 \Quot_C(E_0,r-1,d) . 
\end{align*}
Since $\varphi$ maps $[E^0 \stackrel{q}{\twoheadrightarrow} Q]$ to $[E_0 \twoheadrightarrow (\ker q)^\vee]$ if $Q$ is locally free,
$\varphi $ is an isomorphism on the locus where $Q$ is locally free.
Since the Picard numbers of $M_{N+1}$ and $ \Quot_C(E_0,r-1,d)$ are $\rho(\Pic^d(C))+2$ and $\rho(\Pic^d(C)) +1$ respectively,
$\varphi$ is a divisorial contraction. 
Similarly, 
\[
\varphi_L : M_{N+1,L} =\Quot_C(E^0,1,d+e)_{L\otimes \det E^0} \to \Quot_C(E_0,r-1,d)_{L}
\]
is also a divisorial contraction. 
By the proof of Proposition \ref{lem_Nef} (5),
this shows that $\lambda_{\mc E}([{\rm pt}])$ is represented by a ($\pi_{N+1}$-)semiample line bundle 
and spans a common edge of $\Mov (M_{N+1}/\Pic^d(C))$, $\Mov(M_{N+1,L})$.
Hence we have
\begin{align*}
 \Mov (M_{N+1}/\Pic^d(C))= \Mov (M_{N+1,L})= \R_{\geq 0} \theta_0 + \R_{\geq 0}\lambda_{\cale}([{\rm pt}]),
\end{align*}
which is the union of $ \Nef (M_{i}/{\Pic^d(C)})= \Nef (M_{i,L})$ for $1 \leq  i \leq N+1$.
Then we see that $\pi_i$ is a Mori dream morphism and 
$M_{i,L} $ is a Mori dream space as in the case $s=r$.

Since $\varphi$ and $\varphi_L$ are divisorial contractions,
$\Eff(M_i/\Pic^d(C))$ and $  \Eff (M_{i,L})$ are spanned by $ \theta_0$ and 
the class $[W]$ of  the exceptional divisor $W$ of $\varphi$.
To determine $\R_{\geq 0} [W]$,
we compare the canonical divisors.
By Claim \ref{claim_lambda(pt)},
$ \lambda_{\cale}([{\rm pt}])$ is the pullback of $\calo(1)$ on $ \Quot_C(E_0,r-1,d)_{L} =\P(\Hom(L^{-1}\otimes \det E_0, E_0))$.
Hence  we have
\begin{align*}
 &K_{M_{N+1,L}} -\varphi^*    K_{\Quot_C(E_0,r-1,d)_{L}  }  \\
 &=  (2(r-1)(g-1) -e- 2d)\lambda_{\cale}([{\rm pt}])+(r-2) \lambda_{\cale} ([\calo_C])  
 + (dr +(r-1)e -r(g-1)) \lambda_{\cale}([{\rm pt}])\\
 &= (r-2) (d+e + g-1)\lambda_{\cale}([{\rm pt}]) +(r-2) \lambda_{\cale} ([\calo_C]) =(r-2) \theta'_0
\end{align*}
by Proposition \ref{prop_canonical_divisor}.
Since $r-2 =s-1 \geq 1$,
$[W]$  is a multiple of 
$\theta'_0$.
Since $[W]$ is on the opposite side of 
\[
\Nef(M_{N+1,L}) =\mb R_{\geq 0} \left(\left(d-e_{s-1} -(g-1)\right)\lambda_{\mc E}([{\rm pt}])-\lambda_{\mc E}([\mc O_C])\right)  + \mb R_{\geq 0}\lambda_{\mc E}([{\rm pt}])
\]
from $\mb R_{\geq 0}\lambda_{\mc E}([{\rm pt}])$,
$[W]$ is a positive multiple of 
$\theta'_0$ and hence
\begin{align*}
\Eff(M_i/\Pic^d(C))=\Eff (M_{i,L} ) &=    \R_{\geq 0} \theta_0 +\R_{\geq 0} \theta'_0.\qedhere
\end{align*}
\end{proof}

\subsection{The case $2 \leq s \leq r-2$}

Let $k=r-s$ and set $\delta'_\sharp \coloneqq (k-1)d -k e'_{k-1} $ 
for 
\[
e'_{i} \coloneqq \min\{ \deg F \mid E^0 \twoheadrightarrow F, \rank F=i\} 
\]
and define $\Gamma'  =  \left\{ \delta'_\flat = \delta'_1 < \delta'_2 < \cdots < \delta'_{N'} =\delta'_\sharp  \right\}$ by
\begin{align*}
\Gamma'  \coloneqq \left\{ 
\delta \in (0, \delta'_\sharp] \,\middle|\, \frac{d + \delta}{k} = \frac{d'}{k'}  \text{ for some integers } 1 \leq k' \leq k - 1,\ d' \leq d -e'_{k-k'} \right\}.
\end{align*}
For $\delta \in (\delta'_{j-1}, \delta'_j)$,
we set $M'_j \coloneqq M_{C,\delta}(E^0, r-s, d+e)$ and $M'_{j,L} \coloneqq M_{C,\delta}(E^0,r- s, d+e)_{L\otimes \det E^0}$, where $\delta'_0 \coloneqq 0$ and $\delta'_{N'+1} \coloneqq \infty$. 

Note that the natural birational map
\begin{align*}
   M_{N+1}=\Quot(E^0,r-s,d+e) \dashrightarrow  M'_{N'+1}=\Quot(E_0,s,d)
\end{align*}
gives an isomorphism $M_{N+1} \setminus Z \simeq M'_{N'+1} \setminus Z' $,
where $Z,Z'$ are the loci parametrizing non-locally free quotients.
By \cite[Lemma 5.2]{GS-Picard}, 
the codimensions of $Z$ and $Z'$ are $r-s$ and $s$  respectively.
Hence $M'_{N'+1}$ is a SQM of $ M_{N+1}$.
Thus all the $M_{i},M'_{j}$ are SQMs of each other.
Similarly, the codimensions of $Z|_{M_{i,L}}$ and $Z'|_{M'_{j,L}}$ are $r-s$ and $s$ respectively by \cite[(8.5)]{GS-Picard}.
Hence all the $M_{i,L},M'_{j,L}$ are SQMs of each other.

Let $p_C^*E^0 \to \mc E'_j$ be the universal object on $M'_{j,L} \times C$ for $j \geq 1$. As before, we have the group homomorphism
\[
\lambda_{\mc E'_j} : K(C) \to \Pic(M'_{j,L})
\]
induced by $\mc E'_j$. We set $\mc E \coloneqq\mc E_{N+1}$ and $\mc E'\coloneqq\mc E'_{N'+1}$. 
Therefore we have two bases 
\[
\{\lambda_{\mc E}([\rm pt]), \lambda_{\mc E}([\mc O_C])\},\quad \{\lambda_{\mc E'}([{\rm pt}]),\lambda_{\mc E'}([\mc O_C])\}
\]
of $N^1(M_{i}/\Pic^d(C))=N^1(M_{i,L})_\R=N^1(M'_{j,L})_\R=N^1(M'_{j}/\Pic^d(C))$. 
We determine the relation between these bases.
\begin{lemma}\label{lemma-change of basis}
    We have  
\[
 \lambda_{\mc E}([{\rm pt}])=\lambda_{\mc E'}([{\rm pt}]), \quad 
\lambda_{\mc E'}([\mc O_C]) =-\lambda_{\mc E}([\mc O_C])-2(g-1)\lambda_{\mc E}([{\rm pt}]).
 \]
\end{lemma}
\begin{proof}
The birational map 
\[
M_{N+1}\dashrightarrow M'_{N'+1}, \quad [E_0\xrightarrow{\alpha} E]\mapsto [E^0\xrightarrow{} E^{0}/\im \alpha^{\vee}]
\]
induces a commutative diagram
\[
\begin{tikzcd}
    M_{N+1,L} \ar[rd, "\varphi_L" swap] \ar[rr, dashed]& & M'_{N'+1,L} \ar[ld, "\varphi'_L"] \\
    & \P({\rm Hom}(L^{-1},\wedge^s E^0)) & 
\end{tikzcd}
\]
where the morphism $\varphi_L$  is the restriction of \eqref{eq_lem_Nef_s_leq_r},
and $\varphi'_L$ is defined similarly.
Since $\lambda_{\mc E}([{\rm pt}]), \lambda_{\mc E'}([{\rm pt}])$ are pullbacks of $\calo(1)$ on 
$\P({\rm Hom}(L^{-1},\wedge^s E^0))$,
we have $\lambda_{\mc E}([{\rm pt}])=\lambda_{\mc E'}([{\rm pt}])$.

Note that under the rational map $M_{N+1}\dashrightarrow M'_{N'+1}$,
we have $p_C^*E^0/\mc E^{\vee}\simeq \mc E'$. Hence 
\[
\lambda_{\mc E'}([\mc O_C])=\lambda_{p_C^*E^0/\mc E^{\vee}}([\mc O_C])=\lambda_{p_C^*E^0}([\mc O_C])-\lambda_{\mc E^{\vee}}([\mc O_C]).
\]
By definition, we have
\begin{align*}
\lambda_{p_C^*E^0}([\mc O_C]) =\det (p_! [p_C^*E^0] ) &= \det ( [p_*  p_C^*E^0] -[R^1p_* p_C^*E^0]) \\
&= \det ([H^0(E^0) \otimes \calo] - [H^1(E^0) \otimes \calo] ) =0.
\end{align*}
On the other hand,
we have
\begin{align*}
R\calh om ( Rp_* (\cale^\vee) ,\calo) \simeq Rp_* R\calh om (\cale^\vee, p_C^* \omega_C[1]) =Rp_* (\cale \otimes p_C^* \omega_C) [1] 
\end{align*}
by Grothendieck duality and hence
\begin{align*}
-\lambda_{\mc E^{\vee}}([\mc O_C]) =-\det Rp_* (\cale^\vee) &= \det R\calh om ( Rp_* (\cale^\vee) ,\calo)\\
& \simeq \det (Rp_* (\cale \otimes p_C^* \omega_C) [1] ) \\
&=- \det (Rp_* (\cale \otimes p_C^* \omega_C) ) \\
&=-\lambda_{\cale} ([\omega_C]) = -\lambda_{\mc E}([\mc O_C])-2(g-1)\lambda_{\mc E}([{\rm pt}]).
\end{align*}
This completes the proof of the lemma.
\end{proof}

\begin{corollary}\label{cor_nef_cone_M'_j,L}
For $1\leq j\leq N'$,
\begin{align*}
\Nef (M'_{j}/\Pic^d(C)) = &  \Nef (M'_{j,L}) \\
= & \mb R_{\geq 0}\left(\left(\frac{d+e+\delta'_{j-1}}{r-s}+(g-1)\right)\lambda_{\mc E}([{\rm pt}])+\lambda_{\mc E}([\mc O_C])\right) \\ 
 & + \mb R_{\geq 0}\left(\left(\frac{d+e+\delta'_{j}}{r-s}+(g-1)\right)\lambda_{\mc E}([{\rm pt}]) + \lambda_{\mc E}([\mc O_C])\right).
\end{align*}
For $j=N'+1$, we have
\begin{align*}
\Nef (M'_{N'+1}/\Pic^d(C)) = & \Nef (M'_{N'+1,L}) \\
= & \mb R_{\geq 0}\left(\left(\frac{d+e+\delta'_{N'}}{r-s}+(g-1)\right)\lambda_{\mc E}([{\rm pt}]) + \lambda_{\mc E}([\mc O_C])\right)  + \mb R_{\geq 0}\lambda_{\mc E}([{\rm pt}]) \\
= & \mb R_{\geq 0}\left(\left(d+e -e'_{s-1}+(g-1)\right)\lambda_{\mc E}([{\rm pt}])+\lambda_{\mc E}([\mc O_C])\right)  + \mb R_{\geq 0}\lambda_{\mc E}([{\rm pt}]),
\end{align*}
where $e'_{s-1} \coloneqq \min\{ \deg F \mid E^0 \twoheadrightarrow F, \rank F=s-1\}$.
Furthermore, all the extremal rays are spanned by (relative) semiample line bundles.
\end{corollary}
\begin{proof}
    This follows from Proposition \ref{lem_Nef} and Lemma  \ref{lemma-change of basis} .
\end{proof}

Recall that $\theta_0$ is the pullback of an ample class on $M_C(s,d)$ to $M_{1}$
via the morphism $ f : M_{1} \to M_C(s,d) : [(E,\alpha)] \mapsto [E]$.
Similarly,
let $\theta'_0$ be the pullback of an ample class by $M'_{1} \to M_C(r-s,d+e)$.

\begin{lemma}\label{lemma-effective cone'}
All the cones 
$\Eff (M_{i}/\Pic^d(C))$, $\Eff (M'_{j}/\Pic^d(C))$, $\Mov (M_{i}/\Pic^d(C))$, $\Mov (M'_{j}/\Pic^d(C))$ and  $\Eff (M_{i, L}), \Eff (M'_{j,L}) ,\Mov (M_{i, L}), \Mov (M'_{j,L})$  
coincide with
    \begin{align*}
      \mb R_{\geq 0}\theta_0+\mb R_{\geq 0}\theta'_0   
    &= \mb R_{\geq 0}\left( \left(\frac{d}{s}-(g-1)\right)\lambda_{\cale}([{\rm pt}])-\lambda_{\cale} ([\calo_C]) \right) \\
    & \qquad  + \mb R_{\geq 0} \left( \left(\frac{d+e}{r-s}+(g-1)\right)\lambda_{\cale}([{\rm pt}]) + \lambda_{\cale} ([\calo_C])\right).
    \end{align*}
\end{lemma}

\begin{proof}
By Proposition \ref{lem_Nef}, $\R_{\geq 0} \theta_0 $
is a common edge of $\Eff (M_{i}/\Pic^d(C)), \Eff (M_{i, L}),\Mov (M_{i, L})$ and $\Mov (M_{i}/\Pic^d(C))$.
Similarly,
\begin{align*}
\R_{\geq 0} \theta'_0 &=
\R_{\geq 0}\left( \left(\frac{d+e}{r-s}-(g-1)\right)\lambda_{\cale'}([{\rm pt}]) -\lambda_{\cale'} ([\calo_C])  \right)\\
&=\R_{\geq 0} \left( \left(\frac{d+e}{r-s}+(g-1)\right)\lambda_{\cale}([{\rm pt}])+\lambda_{\cale} ([\calo_C]) \right)
\end{align*}
is a common edge of $\Eff (M'_{j, L})= \Eff (M_{i, L}),\Mov (M'_{j, L}) =\Mov (M_{i, L})$,
where the second equality holds by Lemma \ref{lemma-change of basis}.
Hence this lemma holds.
\end{proof}

\begin{proof}[Proof of Theorem \ref{thm_MDS}, the case $2\leq s \leq r-2$]
\mbox{}

By Proposition \ref{lem_Nef}, Corollary \ref{cor_nef_cone_M'_j,L}, and Lemma \ref{lemma-effective cone'},
$\Mov(M_{i,L})$ is the union of $ \Nef (M_{i,L}), \Nef (M'_{j,L})$ for $1 \leq i \leq N+1, 1 \leq j \leq N' +1$
and each nef cone is spanned by (relative) semiample line bundles.
Hence $M_{i,L}$ is a Mori dream space. Similarly, we get that $\pi_i$ is a Mori dream morphism.
\end{proof}

\begin{proof}[Proof of Theorem \ref{intro_thm_cones}]
For $[L] \in \Pic^n(C)$, we apply  Theorem \ref{thm_MDS} to $ M_{C,\delta}(E_0,r-k,n-e)_{L \otimes \det E_0}$.
Since $\mathcal{Q}_L =\Quot(E^0,k,n)_L \simeq   M_{C,\delta}(E_0,r-k,n-e)_{L \otimes \det E_0} $ for sufficiently large $\delta$,
Theorem \ref{intro_thm_cones} holds by taking $\lambda  =\lambda_{\cale}([{\rm pt}]), \mu= \lambda_{\cale} ([\calo_C])  $.
\end{proof}

\appendix 
\section{$g=2 $ and $s=2$}\label{appendix_s=2}

In this appendix, we use the notation in \S \ref{sec_SQM}.
Furthermore, we assume $g=2, s=2 \leq r$ and $d$ is sufficiently large.
Set $J^{d'} =\Pic^{d'}(C) $ for $d' \in \Z$.

First, assume $d $ is even.
By \cite[Theorem 3]{MR242185},
$M \coloneqq M_C(2,d)$ is a $\P^3$-bundle over $J^d$.
Let $\calo_M(1)$ be the tautological line bundle on the projective bundle $M$.
Then $D \coloneqq  M_C(2,d) \setminus M^s_C(2,d)$ is a prime divisor with $\calo_M(D) \simeq \calo_M(4) \in \Pic( M/J^d)$
since each fiber of $D \to J^d$ is isomorphic to the Kummer surface $\Pic^0(C)/\langle  -1_{\Pic^0(C)}\rangle$ naturally embedded into $\P^3$.

\begin{lemma}\label{lem_g=s=2,irreducible}
 Assume $d $ is even.
 Let $0 < \delta < \delta_\flat$ and  $f:  M_{C,\delta} (E_0,2,d) \to M_C(2,d)$ be the morphism in Lemma \ref{lemma-small delta}.
Then $f^{-1} (D)$ is irreducible.
\end{lemma}

\begin{proof}
Let $P_1\to J^{\frac{d}{2}}$ be the scheme whose fiber over $[N_1]\in J^{\frac{d}{2}}$ is given by $\P({\Hom(N_1^\vee,E^0)})$.
    Note that since $d\gg 0$, we have $\Ext^1(N_1^\vee,E^0)=0$ and hence $P_1\to J^{\frac{d}{2}}$ is a projective bundle.
    Now consider the scheme $P_2\to P_1\times J^{\frac{d}{2}}$ whose fiber over a point 
    \begin{align}\label{eq_irreducible_g=2,s=2}
    ([N_1^\vee \xrightarrow{\beta_1} E^0],[N_2])\in P_1\times J^{\frac{d}{2}}
    \end{align}
    is given by $\P(\Hom(N_2^\vee,\coker \beta_1))$. 
    Since $\Ext^1(N_2^\vee,\coker \beta_1)=0$, $P_2\to P_1\times J^{\frac{d}{2}}$ is a projective bundle.

 Let $[N_2^\vee \xrightarrow{\beta_2} \coker \beta_1] \in P_2$ be a point over  \eqref{eq_irreducible_g=2,s=2}.
Then we obtain a diagram
\[
\begin{tikzcd}
0 \ar[r] & N_1^\vee  \ar[r, "\beta_1"] & E^0 \ar[r]  & \coker \beta_1 \ar[r]  & 0 \\
0 \ar[r]& N_1^\vee  \ar[r] \ar[u, "\id"]  & E^\vee \ar[r]  \ar[u, "\beta"] & N_2^\vee \ar[r] \ar[u, "\beta_2"] & 0
\end{tikzcd}
\]
where the bottom row corresponds to $\eta \circ \beta_2 \in   \Ext^1(N_2^\vee,  N_1^\vee)$
for $\eta \in  \Ext^1(\coker \beta_1,  N_1^\vee) $ which corresponds to the top row,
and $\beta$ is the induced map.

\begin{claim}\label{claim_irreducible}
Let  $\alpha \coloneqq \beta^\vee : E_0 \to E$.
Then $\alpha$ is non-zero and $(E,\alpha)$ is $\delta$-stable.
\end{claim}

\begin{proof}[Proof of Claim \ref{claim_irreducible}]
Since $\alpha =\beta^\vee \neq 0$ by $\beta_1 \neq 0$,
we need to show that $(E,\alpha)$ is $\delta$-stable.

Assume that there exists a quotient $q : E \to F$ with $\rank F=1$ and 
$\mu(E,\alpha) \geq \mu(F,q \circ \alpha)$.
Since $E$ is semistable by the exact sequence $0 \to N_2 \to E \to N_1 \to 0$, we have $\deg F \geq \mu(E) =\frac{d}2$.
Hence $q \circ \alpha =0 $ and $\mu(E,\alpha) =\frac{d+\delta}{2} \geq \deg F$, which implies $\deg F=\frac{d}2$ by $\delta <\delta_\flat$.
If the composite map $F^\vee \xrightarrow{q^\vee} E^\vee \to N_2^\vee$ is zero, 
then $q^\vee$ factors as $F^\vee \hookrightarrow N_1^\vee \hookrightarrow E^\vee $.
Since $\deg N_1 =\deg F =\frac{d}2$,
 $F^\vee \hookrightarrow N_1^\vee$ is an isomorphism.
Then we have $ \beta_1=0 $ by $\beta \circ q^\vee=(q \circ \alpha)^\vee =0$,
which is a contradiction.
Thus the composite map $F^\vee \xrightarrow{q^\vee} E^\vee \to N_2^\vee$ is non-zero
and hence an isomorphism by $\deg N_2 =\deg F =\frac{d}2$.
Then we have $ \beta_2=0 $ by $\beta \circ q^\vee=0$,
which is a contradiction.
Thus $(E,\alpha)$ is $\delta$-stable.
\end{proof}

Note that we have a natural morphism
\[
P_2 \to P_1 \times J^{\frac{d}{2}} \to J^{\frac{d}2} \times  J^{\frac{d}2}  \to J^d,
\]
where the last morphism maps $([N_1],[N_2])$ to $[N_1 \otimes N_2]$.
By Claim \ref{claim_irreducible}, there exists  a morphism
\begin{align*}
\varphi : P_2 \to M_{C,\delta} (E_0,2,d)
\end{align*}
over $J^d$.
Since $E$ in  Claim \ref{claim_irreducible}  is not stable,
the image of $\varphi$ is contained in $f^{-1}(D)$.

On the other hand,
for any $[(E,\alpha)] \in f^{-1}(D)$,
there exist $N_i \in J^{\frac{d}{2}} $ and an extension $0 \to N_2 \to E \to N_1 \to 0$, which  corresponds to  $\eta' \in \Ext^1(N_1,N_2)$.
Since $(E,\alpha)$ is $\delta$-stable and $\mu(E)=\mu(N_1)$, the induced map $\alpha_1 : E_0 \to E \to N_1$ is non-zero.
Hence $\alpha_1^\vee$ is injective and we have a diagram
\[
\begin{tikzcd}
0 \ar[r] & N_1^\vee  \ar[r, "\alpha_1^\vee"] & E^0 \ar[r]  & \coker \alpha_1^\vee \ar[r]  & 0 \\
0 \ar[r]& N_1^\vee  \ar[r] \ar[u, equal]  & E^\vee \ar[r]  \ar[u, "\alpha^\vee"] & N_2^\vee \ar[r] \ar[u, "\beta_2"] & 0
\end{tikzcd}
\]

If the induced map $\beta_2$ is zero, 
there exists $ \gamma:  E^\vee \to N_1^\vee$ with $\alpha^\vee =\alpha_1^\vee \circ \gamma $.
Then $\alpha$ is decomposed as $ E_0 \xrightarrow{\alpha_1} N_1 \xrightarrow{\gamma^\vee} E $.
By $\alpha \neq 0$, so is $\gamma^\vee$ and hence $ \im \gamma^\vee \simeq N_1$.
Since $ \im \gamma^\vee \subset E$ contains $\im \alpha$,
we have $\mu( \im \gamma^\vee , \alpha) =\frac{d}{2} +\delta > \mu(E,\alpha)$, which contradicts the $\delta$-semistability of  $(E,\alpha)$.
Hence $ \beta_2 $ is non-zero and $\beta_1 \coloneqq \alpha_1^\vee$ and $\beta_2$ corresponds to a point $ p_2 \in P_2 $ with $\varphi(p_2) =[(E,\alpha)]$. 
Thus $ f^{-1}(D) = \im \varphi$, which is irreducible.
\end{proof}

\begin{lemma}\label{lem_reduced}
 Assume $d $ is even.
Then the divisor $f^*D$ is a prime divisor.
\end{lemma}

\begin{proof}
For $0 <\delta < \delta_\flat$, let $U\subset M_{C,\delta}(E_0,s,d)$ be 
the open subset given by  the locus of  points $[(E,\alpha)] $ such that $E$ is simple. Note that $f^{-1}(D)\cap U \neq \emptyset$.
Indeed, take $[N_1]\neq [N_2]\in J^{\frac{d}{2}}$ and a non-split extension  $0 \to N_2\to E\to N_1\to 0$.
Since $r \geq 2$ and $d \gg 0$, there exists a generically surjective map $\alpha : E_0 \to E$.
Then $E$ is simple and $(E,\alpha)$ is $\delta$-stable. Hence $[(E,\alpha)] \in f^{-1}(D) \cap U$.

    Since $f^{-1}(D) \cap U \neq \emptyset $,
    it suffices to show that $f|_U: U\to M_C(2,d)$ is smooth. Let us denote by $S\coloneqq{\rm Spl}_C^{ss}(2,d)$ the algebraic space parametrizing simple semistable vector bundles of rank $2$ and degree $d$ (see \cite[Theorem 1]{Narasimhan_Seshadri}, \cite[Theorem 7.4]{Altman_Kleiman}). The map $f|_U$ factors as 
    $U\to S\to M_C(2,d)$. It is known that $S$ is smooth, quasiseparated, locally of finite type over $\C$, and the tangent space at a point $[F]\in S$ is given by ${\Ext}^1(F,F)$.
    If $[(E,\alpha)]\in U$, by the proof of Lemma \ref{lem_smoothness_M_L}, the map $U\to S$ at the level of tangent spaces is given by the natural map $\Hom (I^\bullet,E)\to  \Ext^1(E,E)$ for $I^\bullet =\{ E_0 \xrightarrow{\alpha} E \}$, which is surjective 
    since $\Hom (I^\bullet,E)\to  \Ext^1(E,E) \to \Ext^1(E_0,E)  $ is exact and 
    $\Ext^1(E_0,E)=0$ by $d \gg 0$.
    By Remark (ii) on p.23 of \cite{MR242185}, $S \to M_C(2,d)$ is locally injective  and hence locally an isomorphism at such $E \in S$ since $S$ and $M_C(2,d)$ are smooth. 
    Hence $f|_U : U \to M_C(2,d)$ is smooth.
\end{proof}

In the rest of Appendix \ref{appendix_s=2}, $d$ could be odd. 

\begin{lemma}\label{lem_Pic_M}
The morphism $f$ induces 
$\Pic ( M_{C,\delta} (E_0,2,d)) \simeq \Pic (M_C(2,d)) \oplus \Z$.
\end{lemma}

\begin{proof}
If $d$ is odd, this follows from Remark \ref{rem_small_delta} since $M^s_C(2,d) =M_C(2,d) $.

Assume $d $ is even.
Since $f^* D$ is a prime divisor,
this lemma follows from the diagram
\[
\begin{tikzcd}
&&& 0 \arrow[d] &\\
0 \arrow[r]  & \Z [D] \arrow[r] \arrow[d, "{\rotatebox{90}{$\sim$}}"]& \Cl(M_C(2,d)) \arrow[r] \arrow[d, "f^*"]& \Cl(M_C(2,d) \setminus D) \arrow[r] \arrow[d] & 0\\
0 \arrow[r] & \Z [f^*(D)] \arrow[r] &\Cl (M_{C,\delta}(E_0,2,d)) \arrow[r]   & \Cl (M_{C,\delta}(E_0,2,d)\setminus f^{-1}(D))  \arrow[r] \arrow[d]  & 0 \\
&&& \Z \arrow[d] &\\
&&& 0 &
\end{tikzcd}
\]
and the factoriality of $M_{C,\delta} (E_0,2,d) ,M_C(2,d) $.
We note that the right column is exact since $f$ is a projective bundle over $M_C(2,d) \setminus D =M^s_C(2,d)$.
\end{proof}

If $r \geq 3$, there exists the composite of  birational maps
\begin{align*}
\psi:  \Quot_C(E_0, 2, d) \dashrightarrow \Quot_C(E^0, r-2, d+e) \dashrightarrow M_{C,\delta}(E_0,2,d).
\end{align*}

\begin{lemma}\label{lem_codim_semistable}
The birational map $\psi$ does not contract divisors.
\end{lemma}

\begin{proof}
Let  $\calq\coloneqq \Quot_C(E_0, 2, d)$  and
\begin{align*}
   \calq^{\mathrm{tf}}_{\mathrm{g}} \coloneqq \{ [E_0 \twoheadrightarrow E] \in \calq \mid \text{$E$ is torsion free and $H^1(E_0^{\vee} \otimes E) =0$} \} .
\end{align*}
By \cite[Lemmas 4.9,  5.2]{GS-Picard},
$ \calq \setminus \calq^{\mathrm{tf}}_{\mathrm{g}} $ has codimension $\geq 2$ in $\calq$.

Take $x \in C$, $n \gg 1$ and set $N=h^0(E\otimes \calo_C(nx)) = d +2n -2(g-1) =d+2n-2$ for $ [E_0 \twoheadrightarrow E] \in \calq^{\mathrm{tf}}_{\mathrm{g}}  $.
Let $R \subset \Quot_C(\calo_C^N, 2, d+2n)$ be the open subset consisting of $ [\calo_C^N \to G']$ such that 
$G'$ is locally free, $h^1(G')=0$, and the induced map $ \C^N \to H^0(G')$ is  an isomorphism.
Let $R'$ be the open subset of $R$ defined by the condition $H^1( E_0^{\vee} \otimes G'(-nx)) =0$.

Let $\calo_{R' \times C}^N \to \calg'$ be the universal quotient and set $\calg =\calg' \otimes p_C^* \calo_C(-nx)$.
Then $ {p_{R'}}_* \calh om(E_0, \calg)$ is locally free. Let $\mb U\subset \P_{R'}( {p_{R'}}_* \calh om(E_0, \calg))$ be the open subset parametrizing surjections $(E_0\to G' { \otimes \calo_C(-nx)}, \mc O^N_C\to G')$ for $[\mc O^N_C\to G']\in R'$.  We obtain a diagram
\[
\begin{tikzcd}
 \P_{R'}( {p_{R'}}_* \calh om(E_0, \calg)) \supset \mb U \arrow[r, "\Psi"] \arrow[d]  & \calq^{\mathrm{tf}}_{\mathrm{g}} \\
R'& 
\end{tikzcd}
\]
Let $\calq^{ss} \coloneqq \{[E_0 \twoheadrightarrow E] \in \calq  \mid \text{$E $ is semistable} \}$
and $ R^{ss} \coloneqq \{[\calo_C^N \twoheadrightarrow G'] \in R  \mid \text{$G' $ is semistable} \} $.
Since $\dim R' -\dim R' \setminus R^{ss} \geq g=2$ by \cite[Proposition 1.2]{Bhosle},
we obtain $ \dim \calq^{\mathrm{tf}}_{\mathrm{g}} - \dim \calq^{\mathrm{tf}}_{\mathrm{g}}\setminus \calq^{ss} \geq 2$
as in the proof of \cite[Lemma 7.11]{GS-Picard}.

For any $[E_0 \stackrel{q}{\twoheadrightarrow} E] \in \calq^{\mathrm{tf}}_{\mathrm{g}} \cap \calq^{ss}$,
$(E,q)$ is $\delta$-stable.
In fact, for any quotient $(E',q')$ of $(E,q)$ with $\rank E' =1$, we have $q' \neq 0$ and hence 
\[
\mu (E', q') = \deg E' +\delta > \mu(E) +\frac{\delta}2 =\mu(E,q).
\]
Hence $\calq^{\mathrm{tf}}_{\mathrm{g}} \cap \calq^{ss} $ is embedded into $M_{C,\delta}(E_0,2,d)$ by $\psi$ and  this lemma holds.
\end{proof}

\begin{lemma}\label{lem_appendix_picard_group_s=2}
For $r \geq 4$ and $0 < \delta < \delta_\flat$, 
$\Quot_C(E_0, 2, d ), \Quot_C(E^0, r-2, d+e),M_{C,\delta} (E_0,2,d)$ are SQMs of each other, and their Picard groups are isomorphic to $ \Pic (M_C(2,d)) \oplus \Z \simeq \Pic (\Pic^d (C)) \oplus \Z^{2}$.
\end{lemma}

\begin{proof}
By \cite[Lemma 5.2]{GS-Picard}, $\Quot_C(E^0, r-2, d+e) \dashrightarrow \Quot_C(E_0, 2, d)$ is an isomorphism in codimension one.
By Lemma \ref{lem_codim_semistable},
$\psi :\Quot_C(E_0, 2, d)\dashrightarrow M_{C,\delta}(E_0, 2,d) $ does not contract divisors.
Since so does $M_{C,\delta}(E_0, 2,d)  \dashrightarrow \Quot_C(E^0, r-2, d+e)$ by Lemma \ref{lemma-upper bound on dim of W},
these three varieties are SQMs of each other.
Hence this lemma holds by Lemma \ref{lem_Pic_M}.
\end{proof}

Similar statements hold for $ \Quot_C(E^0, r-2, d+e)_L $.
In fact, if $d$ is even,
let $f_L:  M_{C,\delta} (E_0,2,L) \to M_C(2,L) \simeq \P^3$ be the restriction of $f$
and let $D_L = D \cap M_C(2,L) \sim \calo_{\P^3}(4)$,
which is a prime divisor.
Then $f_L^{-1}(D_L) =\varphi(P_{2,L})$ holds,
where  $P_{2,L}$ is the image of the fiber of 
$ P_2 \to J^d $ over $[L] \in J^d$.
Since $ P_2 \to J^d $ is obtained as
\begin{align*}
  P_2 \xrightarrow{p_2} P_1 \times J^{\frac{d}{2}} \xrightarrow{p_1}   J^{\frac{d}{2}} \times J^{\frac{d}{2}} \to J^{d},
\end{align*}
$P_{2,L} =(p_1 \circ p_2)^{-1} (J')$ holds,
where $J' =\{([N_1],[L\otimes N_1^{-1}]) \mid N_1 \in J^{\frac{d}{2}} \} \subset J^{\frac{d}{2}} \times J^{\frac{d}{2}}$.
Since $J' \simeq J^{\frac{d}{2}}$ is irreducible and $p_1,p_2$ are projective bundles,
$P_{2,L} =(p_1 \circ p_2)^{-1} (J')$ is irreducible.
Hence $f_L^{-1}(D_L)$ is irreducible as well.
Since $f_L^{-1}(D_L) \cap U \neq \emptyset$ for $U$  in Lemma \ref{lem_reduced},
$f_L^{-1}(D_L)$ is reduced by the computation in Lemma \ref{lem_reduced}.

By  \cite[(8.5)]{GS-Picard} and $r \geq 4$,
$\Quot_C(E^0, r-2, L \otimes \det E^0)$ and $\calq_L \coloneqq \Quot_C(E_0, 2, L)$ are isomorphic in codimension one.
Let 
\[
\calq^{\mathrm{tf}}_{\mathrm{g}} \coloneqq \{ [E_0 \twoheadrightarrow E] \in \calq \mid \text{$E$ is torsion free and $H^1(E_0^{\vee} \otimes E) =0$} \}.
\]
For $\calq^{\mathrm{tf}}_{\mathrm{g}} $ in Lemma \ref{lem_codim_semistable},
$ \calq_L \setminus \calq^{\mathrm{tf}}_{\mathrm{g}} $ has codimension $\geq 2$ in $\calq_L$ by \cite[(8.1), (8.5)]{GS-Picard}.
Hence we obtain the following lemma.

\begin{lemma}\label{lem_appendix_picard_group_s=2,L}
For $r \geq 4$, $0 < \delta < \delta_\flat$ and $[L]\in \Pic^d (C)$, 
$\Quot_C(E_0, 2, d +e)_{L \otimes \det E^0}, \Quot_C(E^0, r-2, d)_{L}, M_{C,\delta} (E_0,2,d)_L$ are SQMs of each other, and their Picard groups are isomorphic to $ \Pic (M_C(2,L)) \oplus \Z \simeq  \Z^{2}$.
\end{lemma}

\section{$g=2, s-r=2$}\label{sec_appendix_s-r=2}

In this appendix, we use the notation in \S \ref{sec_SQM}.
Furthermore, we assume $g=2, s-r=2$ and $d+e$ is sufficiently large.

First, assume $d+e$ is even.
Define $D\coloneqq M_C(2,d+e)\setminus M^s_C(2,d+e)$. 
By Lemma \ref{lemma-upper bound on dim of W,s >r} (3),
we have a morphism 
\[
h:M_{C,\delta_\sharp-\varepsilon}(E_0,s,d)\to M_{C,\delta_\sharp}(E_0,s,d) =M_C(2,d+e) : [(E,\alpha)] \mapsto [E/\im \alpha]
\] for sufficiently small $\ep >0$.

\begin{lemma}\label{leq_irreducible_g=2,s-r=2}
Assume $d+e$ is even. 
Then  $h^{-1}(D)$ is irreducible.
\end{lemma}

\begin{proof}
    Let $P'_1\to J^{\frac{d+e}{2}}$ be the scheme whose fiber over $[N_1]\in J^{\frac{d+e}{2}}$ is given by $\P({\Ext^1(N_1,E_0)})$. Note that since $d\gg 0$, we have $\Hom(N_1,E_0)=0$ and hence $P'_1\to J^{\frac{d+e}{2}}$ is a projective bundle. 
    Now consider the scheme $P'_2\to P'_1\times J^{\frac{d+e}{2}}$ whose fiber over a point 
    \begin{align}\label{eq_irreducible_g=2,s-r=2}
    ([0\to E_0\to E'\to N_1\to 0],[N_2])\in P'_1\times J^{\frac{d+e}{2}}
    \end{align}
    is given by $\P(\Ext^1(N_2,E'))$. 
    We claim that $P'_2$ is again a projective bundle over $P'_1\times J^{\frac{d+e}{2}}$. 
    To show this, it is sufficient to show that $\Hom(N_2,E')=0$. Consider the exact sequence
    \[
    0\to \Hom(N_2,E_0)\to \Hom(N_2,E')\to \Hom(N_2,N_1)\to \Ext^1(N_2,E_0).
    \]
    Since $d\gg 0$, we have $\Hom(N_2,E_0)=0$. If $N_2\not \simeq N_1$ then $\Hom(N_2,N_1)=0$ and if $N_2 \simeq N_1$ then the map $\Hom(N_2,N_1)\to \Ext^1(N_2,E_0)$ is injective since the extension $0 \to E_0 \to E' \to N_1 \to 0$ is non-split. Therefore we have $\Hom(N_2,E')=0$.  

    Let $[0\to E'\to E\to N_2\to 0]\in P'_2$ be a point over \eqref{eq_irreducible_g=2,s-r=2}.
    Then we have a diagram
\begin{equation}\label{eq_B2}
        \begin{tikzcd}
                 &   & 0  \ar[d] & 0  \ar[d] &  \\
        0 \ar[r] & E_0 \ar[r] \ar[d,equal] & E' \ar[r] \ar[d] & N_1 \ar[r] \ar[d] & 0 \\
        0 \ar[r] & E_0 \ar[r,"\alpha"] & E \ar[r,"\pi"] \ar[d] & F \ar[r] \ar[d] & 0 \\
                  &          & N_2 \ar[r,equal] \ar[d] & N_2 \ar[d] & \\
    &   & 0   & 0   &  
    \end{tikzcd}   
\end{equation}

\begin{claim}\label{claim_irreducible_g=2,s-r=2}
    $(E,\alpha)$ is $(\delta_\sharp-\varepsilon)$-stable.
\end{claim}

\begin{proof}
In this case, $\delta_\sharp =\frac{rd+se}{s-r} =\frac{rd+(r+2)e}{2}$
and hence
$\mu_{\delta_\sharp}(E,\alpha) = \mu_{\delta_\sharp}(E_0,\id_{E_0}) = \mu(F) =\frac{d+e}{2}$.

Consider the exact sequence $0 \to (E_0,\id) \to (E,\alpha) \to (F,0) \to 0$.
Note that $(E_0,\id)$ is $\delta_\sharp$-stable since  $\delta_\sharp $ is sufficiently large.
If $(E,\alpha)$ is not $(\delta_\sharp-\varepsilon)$-stable,
there exists a subsheaf $G \subset E$ such that $G \cap E_0 =0$ and $\mu(G)=\mu(F)$ by Lemma \ref{lem_preliminary} (5).
Since $ G \simeq \pi(G)$, we have $\mu(\pi(G)) =\mu(G) =\mu(F)$.

Assume $\pi(G) =F$.
Then $\pi|_G : G \to F$ is an isomorphism and hence so is $(\pi|_G)^{-1}(N_1) =G \cap E' \to N_1$.
This contradicts the non-splitness of $0 \to E_0 \to E' \to N_1 \to 0$.

Thus $ \pi(G) \neq F$.
By $\mu(\pi(G)) =\mu(F)$, we have $\rank \pi(G) \leq \rank F-1=1$. 
Hence $\pi(G) $ is an invertible sheaf of degree $\frac{d+e}{2}$.

Assume $\pi(G) \subset N_1 $ as a subsheaf of $F$.
Then $\pi(G) = N_1 $  holds by $\deg \pi(G)=\deg N_1$.
Since $G \subset \pi^{-1}(N_1) =E'$,
$G \simeq \pi(G) =N_1$ implies that $0 \to E_0 \to E' \to N_1 \to 0$ is split, which is a contradiction.

Assume $\pi(G) \neq N_1 $.
Then the composite $G \simeq \pi(G) \hookrightarrow F \to N_2$ is non-zero and hence an isomorphism by $\deg G=\deg N_2$.
This contradicts the non-splitness of $0 \to E' \to E \to N_2 \to 0$.
 
Hence such $G$ does not exist and     $(E,\alpha)$ is $(\delta_\sharp-\varepsilon)$-stable.
\end{proof}

    By Claim \ref{claim_irreducible_g=2,s-r=2},
    we have a morphism $P'_2\to M_{C,\delta_\sharp-\varepsilon}(E_0,s,d)$. Clearly the image of this morphism is contained in $h^{-1}(D)$. 
    On the other hand, given any $[0 \to E_0\xrightarrow{\alpha} E\to F\to 0]\in h^{-1}(D)$, we have $[F]\in D$,  that is, $F$ is strictly semistable, so there exists an exact sequence $0\to N_1\to F\to N_2\to 0$ with $N_1,N_2\in J^{\frac{d+e}{2}}$. Hence we again get a diagram of the above form \eqref{eq_B2}.
    Since $E$ is $(\delta_\sharp-\varepsilon)$-stable and $\pi: E\to F$ is a quotient, we have 
    $\mu_{\delta_\sharp-\varepsilon}(E,\alpha)<\mu(F)=\mu(N_1)=\mu(N_2)$. Therefore,
    both the extensions $0\to E_0\to E'\to N_1\to 0$ and $0\to E' \to E\to N_2\to 0$ are non-split since otherwise there would be an injection $N_1\to E' \hookrightarrow E$ or $N_2\to E$,
    which contradicts the $(\delta_\sharp-\varepsilon)$-stability of $E$.
    This implies that the map $P'_2\to h^{-1}(D)$ is surjective and hence $h^{-1}(D)$ is irreducible. 
\end{proof}

\begin{lemma}\label{lem_appB_reduced}
    The scheme $h^{-1}(D)$ is reduced.
\end{lemma}

\begin{proof}
    Let $U\subset M_{C,\delta_\sharp-\varepsilon}(E_0,s,d)$ be 
the open subset given by the locus of points $[E_0\hookrightarrow E]$ such that $E/E_0$ is simple. Note that $h^{-1}(D)\cap U\neq \emptyset$. 
Indeed, by Claim \ref{claim_irreducible_g=2,s-r=2} if $[N_1]\neq [N_2]\in J^{\frac{d+e}{2}}$ and if $0\to E_0\to E'\to N_1\to 0$ is a non-split exact sequence, then any non-split exact sequence $\eta:0\to E'\to E\to N_2\to 0$ gives rise to a $(\delta_\sharp-\varepsilon)$-stable pair $E_0 \hookrightarrow E' \hookrightarrow E$. We have an exact sequence $0 \to N_1 \to E/E_0 \to N_2\to 0$ and this extension class is precisely the image  of $\eta$ under the map ${\rm Ext}^1(N_2,E')\to {\rm Ext}^1(N_2,N_1) \simeq \C$. Since this map is surjective, there exists $\eta$ such that the induced $0 \to N_1\to E/E_0\to N_2\to 0$ is non-split, which implies that $E/E_0$ is simple.

    Since $h^{-1}(D) \cap U \neq \emptyset $,
    it suffices to show that $h|_U: U\to M_C(2,d+e)$ is smooth. Let us denote by $S\coloneqq{\rm Spl}_C^{ss}(2,d+e)$ the algebraic space parametrizing simple semistable vector bundles of rank $2$ and degree $d+e$ \cite[Theorem 1]{Narasimhan_Seshadri}, \cite[Theorem 7.4]{Altman_Kleiman}. The map $h|_U$ factors as 
    $U\to S\to M_C(2,d+e)$. 
    If $[E_0 \hookrightarrow E]\in U$, by the proof of Lemma \ref{lem_smoothness_M_L}, the map $U\to S$ at the level of tangent spaces is given by the natural map $\Ext^1(E/E_0,E)\to  \Ext^1(E/E_0,E/E_0)$ which is surjective.   Hence $h|_U:U \to M_C(2,d+e)$ is a smooth morphism as in the proof of Lemma \ref{lem_reduced}. 
\end{proof}

In the following corollary, $d+e$ could be odd. 

\begin{corollary}\label{cor_appB_Pic}
The morphism $h$ induces 
$\Pic ( M_{C,\delta_\sharp-\varepsilon} (E_0,s,d)) \simeq \Pic (M_C(2,d+e)) \oplus \Z $
and $\Pic ( M_{C,\delta_\sharp-\varepsilon} (E_0,s,d))_L \simeq \Pic (M_C(2,L \otimes \det E^0)) \oplus \Z $.
\end{corollary}

\begin{proof}
If $d+e$ is odd, this corollary holds since $h$ is a projective bundle.

Assume $d+e$ is even.
By Lemmas \ref{leq_irreducible_g=2,s-r=2}, \ref{lem_appB_reduced},  $h^*D$ is a prime divisor.
Furthermore, the pullback of the divisor $D_L=D\cap M_C(2,L \otimes \det E^0)$ by the restriction $h_L : M_{C,\delta_\sharp-\varepsilon} (E_0,s,d)_L \to M_C(2,L \otimes \det E^0)$ is irreducible by the same argument just before Lemma \ref{lem_appendix_picard_group_s=2,L}.
Hence we obtain this corollary as in the proof of Lemma \ref{lem_Pic_M}.
\end{proof}

\end{document}